\documentclass[11pt,letterpaper]{article}
\usepackage[margin=1in]{geometry}
\usepackage[utf8]{inputenc} 
\usepackage[T1]{fontenc}    
\usepackage{dsfont}
\usepackage[maxfloats=100]{morefloats}[2015/07/22]
\usepackage{multirow}
\usepackage{wrapfig}
\usepackage[T1]{fontenc}
\usepackage{lmodern}

\usepackage{booktabs}       
\usepackage{amsfonts}       
\usepackage{nicefrac}       
\usepackage{microtype}      
\usepackage{enumerate}
\usepackage{lipsum}
\usepackage{mathtools}
\usepackage{cuted}
\usepackage{float}
\usepackage[dvipsnames,table,xcdraw]{xcolor}
\usepackage{bbm}
\usepackage{varioref}
\usepackage{hyperref}
                                
\usepackage{amssymb} 
\usepackage{rotating}
\usepackage{graphicx}
\usepackage{subcaption}
\usepackage[normalem]{ulem}

\usepackage{amsthm}
\DeclareUnicodeCharacter{00A0}{~}

\newtheorem{assumption}{Assumption}
\newtheorem{theorem}{Theorem}
\newtheorem{lemma}{Lemma}
\newtheorem{proposition}{Proposition}

\theoremstyle{plain}
\newtheorem{remark}{Remark}
\newtheorem{example}{Example}
\RequirePackage[capitalize,nameinlink]{cleveref}[0.19]

\usepackage{caption}
\usepackage{amsthm}
\usepackage{algorithm}
\usepackage{algpseudocode}
\usepackage{multirow}
\usepackage{diagbox}
\usepackage{comment}
\usepackage{xcolor} 
\newcommand{\fy}[1]{{\color{black}#1}}
\newcommand{\fyy}[1]{{\color{black}#1}}
\newcommand{\far}[1]{{\color{black}#1}}
\newcommand{\me}[1]{{\color{black}#1}}

\newcommand{\Rme}[1]{{\color{black}#1}}
\newcommand{\RRme}[1]{{\color{black}#1}}
\newcommand{\RRRme}[1]{{\color{black}#1}}
\newcommand{\mje}[1]{{\color{black}#1}}
\newcommand{\yq}[1]{{\color{black}#1}}
\newcommand{\fyr}[1]{{\color{black}#1}}
\begin{document}

\title{Self-tuned Regularized Federated Methods with Guarantees for Optimal Solution Selection}

\author{Mohammadjavad Ebrahimi\textsuperscript{a}, Yuyang Qiu\textsuperscript{a}, Shisheng Cui\textsuperscript{b} and Farzad Yousefian\textsuperscript{a}
\thanks{\textsuperscript{a}Department of Industrial and Systems Engineering, Rutgers, the State University of New Jersey, United States; \textsuperscript{b}School of Automation, Beijing Institute of Technology, Beijing 100081, China. \\
This work was funded in part by the NSF under CAREER grant ECCS-2323159, in part by the ONR under grant N00014-22-1-2757, and in part by the DOE under grants DE-SC0023303 \RRRme{and DE-SC0025570}.}
}

\maketitle

\begin{abstract}
We study a hierarchical federated learning (FL) problem, where clients cooperatively seek to select among multiple optimal solutions of a primary distributed learning problem, a solution that minimizes a secondary loss function. This problem arises from over-parameterized learning and ill-posed optimization problems. First, we consider the setting where the inner-level objective is convex and the outer-level objective is either convex or strongly convex. We propose a self-tuned regularized federated averaging (StR-FedAvg) method where the stepsize and regularization parameter are characterized by the number of communication rounds and problem parameters. We derive new complexity guarantees for addressing the optimal solution selection problem in FL. Second, when the outer-level objective is nonconvex, we propose a two-loop FL scheme in which the outer loop employs an inexact projected first-order method and the inner loop applies StR-FedAvg with an iteratively updated regularization parameter. We derive new communication complexity guarantees for computing a stationary point of the nonconvex solution-selection problem. To our knowledge, this is the first work to establish complexity guarantees for this class of problems in FL. Preliminary experiments validate our theoretical findings.
\end{abstract}

  \section{Introduction}
Federated learning (FL) has emerged as a \RRme{collaborative} framework for training \RRme{machine learning} models over a network of clients associated with distributed and
private datasets~\cite{mcmahan2017communication}. \RRme{Recently, there has been progress} in the design, understanding, and analysis of communication-efficient and privacy-preserving federated methods~\cite{zhao2018federated,stich2019local,kairouz2021advances,khaled2020tighter,ghiasvand2024communication},  accommodating key challenges such as data heterogeneity and the need for personalization~\cite{li2020federated,karimireddy2020scaffold,
li2021ditto}. 
In this paper, we study a \fyy{distributed \RRRme{optimization} problem}, of the form
\begin{align}\label{problem: main problem} 
& \min_{x \in \mathbb{R}^n} \  f(x):=\tfrac{1}{N} \textstyle\sum_{i=1}^N \mathbb{E}_{\xi_i \in \mathcal{D}_i}[\tilde f_i( x,\xi_i)] \quad \\
 &\hbox{s.t.}\ \  x \in \hbox{arg}\min_{y\in \mathbb{R}^n}\  h(y):= \tfrac{1}{N} \textstyle \sum_{i=1}^N\mathbb{E}_{\zeta_i \in \tilde{\mathcal{D}}_i}[\tilde h_i(y,\zeta_i)],\notag
\end{align}
where $N$ clients cooperatively seek to select among multiple optimal solutions of minimizing a primary global loss function $h$, one that minimizes a secondary global metric $f$. Here, each client $i \in [N]$, with $[N]:=\{1,\ldots, N\}$, is associated  with a pair of local loss functions $\tilde f_i$ and $\tilde h_i$, and private datasets $\mathcal{D}_i$ and $\tilde{\mathcal{D}}_i$. \fyy{In this setting, we assume that $h$ is convex, and $f$ is either convex, strongly convex, or nonconvex.} \mje{We focus on convex inner problems to ensure that the solution set is well-defined and computationally tractable.}

\noindent\fyy{A key motivation arises from over-parameterized machine learning, including over-parameterized regression~\cite{jiang2023conditional} and lexicographic learning~\cite{gong2021bi}\RRme{, as well as in ill-conditioned \RRRme{optimization} problems\RRRme{~\cite{tikhonov1977solutions,friedlander2008exact,kaushik2023incremental}}}. \mje{In such settings, e.g., sparse regression~\cite{tibshirani1996regression} or fairness-aware~\cite{hardt2016equality} model selection}, the training optimization problem characterized by $\min_{y} \ \tfrac{1}{N} \sum_{i=1}^N\mathbb{E}_{\zeta_i \in \tilde{\mathcal{D}}_i}[\tilde h_i(y,\zeta_i)]$ may \RRRme{either} admit multiple optimal model parameters \RRRme{or may be sensitive to perturbations in the data}. To improve solution quality, one may seek to obtain among the multiple optimal solutions, a solution that minimizes a secondary metric denoted by $f$, leading to the \RRRme{formulation in} \eqref{problem: main problem}. A popular choice for the function $f$ lies in regularization where $f$ is either a convex or a nonconvex regularizer.} \mje{Such bilevel structures can also arise in control-oriented FL applications, such as federated learning-based distributed model predictive control~\cite{xu2025federated}, where one selects among feasible solutions based on additional performance criteria under communication constraints.}

\RRRme{{\noindent {\bf Our contributions.} We address problem~\eqref{problem: main problem} under three settings for the outer-level objective $f$, including strongly convex, convex, and nonconvex. First, for the strongly convex and convex cases, we develop a self-tuned federated averaging (StR-FedAvg) method in which each client employs a regularized variant of the stochastic local objectives $\nabla \tilde f_i$ and $\nabla \tilde h_i$. For each setting, we derive explicit self-tuned rules for selecting the stepsize and regularization parameters characterized by the number of communication rounds and problem parameters (see Algorithm~\ref{Alg:FEDAVG}). Under these rules, we establish convergence rates and communication complexity bounds for solving problem~\eqref{problem: main problem}. Table~\ref{Table:Comunication complexities} summarizes the main assumptions and performance guarantees. We consider both the outer-level error metric $\mathbb{E}[f(\bullet)] - f^*$ and the inner-level error metric $\mathbb{E}[h(\bullet)] - h^*$, where $f^*$ and $h^*$ denote the optimal objective values of the outer- and inner-level problems in~\eqref{problem: main problem}, respectively.
}  

Second, when the inner-level objective is convex and the outer-level objective is nonconvex, we propose a two-loop FL scheme. The outer loop employs an inexact projected first-order method, while the inner loop uses StR-FedAvg in which the regularization parameter is updated to track the optimal solution of the projection problem associated with the solution set of the inner-level FL problem. For this scheme, we derive new communication complexity guarantees for computing a stationary point of the nonconvex solution-selection optimization problem (see Table~\ref{Table:Comunication complexities}). To the best of our knowledge, this is the first work to establish complexity guarantees for the optimal solution selection problem in FL.}

\Rme{\noindent {\RRme{\bf Related work.}} 
Consider a centralized variant of \eqref{problem: main problem} given as $\min_{x\in X^*_h}  f(x)$ where $X^*_h:=\arg\min_{x\in\mathbb{R}^n} h(x)$. This problem has recently gained attention~\cite{yamada2005hybrid,solodov2007explicit,solodov2007bundle,beck2014first,kaushik2021method,yousefian2021bilevel,kaushik2023incremental,jiang2023conditional,samadi2023achieving,jalilzadeh2024stochastic} and finds its roots in the study of ill-posed optimization problems~\cite{friedlander2008exact,tikhonov1963solution}. Early efforts in addressing the deterministic setting date back to the 1960s and include the seminal work of Tikhonov~\cite{tikhonov1963solution,tikhonov77solutions}. The notion of {\it exact regularization} was studied in~\cite{exact1,exact2,ferris91,friedlander2008exact}, and {\it iterative regularization} methods were developed in~\cite{solodov2007explicit,solodov2007bundle,kaltenbacher2008iterative,lin2016iterative,kaushik2021method,samadi2025improved}, with recent accelerated regularized proximal methods in~\cite{samadi2023achieving}. In centralized stochastic settings, iterative regularized mirror descent methods~\cite{amini2019iterative}, iteratively penalized stochastic approximation methods~\cite{jalilzadeh2024stochastic}, and projection-free schemes~\cite{jiang2023conditional} were proposed and analyzed.

\RRRme{Some existing} approaches \RRRme{rely} on approximating the constraint as $h(x)\le h^*+\epsilon_h$ for some $\epsilon_h>0$, leading to projection-free schemes~\cite{jiang2023conditional} and inexact dual methods~\cite{shen2023online}. However, even for a small $\epsilon_h$, this approach may cause an undesirable deviation from the solution set $X^*_h$, which may \RRRme{affect} the solution quality of the bilevel problem. Another avenue is {\it sequential regularization}, where \eqref{problem: main problem} is approximated by $\min_x \{h(x)+\eta f(x)\}$ for some $\eta>0$; it is known~\cite{tikhonov1963solution} that accumulation points of $\{x^*_\eta\}$ converge to solutions of \eqref{problem: main problem}, but this approach is computationally exhaustive and does not readily yield performance guarantees. \mje{In centralized settings, complexities typically range from $\mathcal{O}(\epsilon^{-0.5})$ to $\mathcal{O}(\epsilon^{-4})$, depending on the assumptions.} \mje{In distributed settings, the work in~\cite{yousefian2021bilevel} studies bilevel optimization over networks and establishes rates for both upper-level suboptimality and lower-level infeasibility, with a balanced rate of $\mathcal{O}(\epsilon^{-3})$, although the work is not in the FL setting. More recently,~\cite{boontawee2025federated} proposes a federated incremental subgradient method with inner-level suboptimality rate $\mathcal{O}(\epsilon^{-1/(0.5-\delta)})$, where $\delta>0$ is arbitrarily small. However, this work considers a centralized outer-level objective and does not provide feasibility guarantees, and neither work establishes lower bounds on suboptimality.}

However, communication-efficient methods with explicit guarantees for solving~\eqref{problem: main problem} remain unavailable, and these challenges are exacerbated in the FL setting due to communication overhead and privacy concerns.}  \mje{In contrast, our results in Table~\ref{Table:Comunication complexities} establish explicit communication complexity guarantees for both outer-level suboptimality and inner-level feasibility in a fully federated setting. Despite the distributed bilevel structure, our rates remain comparable to existing centralized and distributed results~\cite{yousefian2021bilevel,boontawee2025federated}.}

\begin{table}

\centering

\tiny

\caption{Communication complexities for solving \RRRme{problem~\eqref{problem: main problem}}.}

{\renewcommand{\arraystretch}{1.2}

{\footnotesize\RRme{\begin{tabular}{|c|c|c|c|c|c|c|c|}

\hline

\multirow{3}{*}{\textbf{FL Methods}} & \multicolumn{3}{c|}{\textbf{Main assumptions}}& \multicolumn{4}{c|}{\textbf{Communication complexity bounds}} \\

\cline{2-8}

& \multirow{2}{*}{$f$} &\multirow{2}{*}{$h$} &\multicolumn{1}{c|}{$X_h^*$}&\multicolumn{2}{c|}{$\mathbb{E}[f(\bullet)] - f^*$ } &\multicolumn{2}{c|}{$\mathbb{E}[h(\bullet)] - h^*$}\\

\cline{4-8}

&&&weak sharp&L.B.&U.B.&L.B.&U.B.\\

\hline

\multirow{6}{*}{\begin{tabular}[c]{@{}l@{}}{\RRRme{StR-FedAvg}}  \end{tabular}}& \multirow{3}{*}{\begin{tabular}[c]{@{}l@{}}Str. Cvx.\\ \end{tabular}} & \multirow{3}{*}{\begin{tabular}[c]{@{}l@{}}Cvx.\end{tabular}} &\multicolumn{1}{c|}{not required}&\multicolumn{1}{c|}{asymptotic}&\multicolumn{1}{c|}{$\mathcal{O}(\epsilon^{-3})$}&\multicolumn{1}{c|}{0}&\multicolumn{1}{c|}{$\mathcal{O}( \epsilon^{-3})$}\\

\cline{4-8}

&&&$\kappa>1$&$\mathcal{O}(\epsilon^{-3\kappa})$&$\mathcal{O}(\epsilon^{-3})$&0&$\mathcal{O}( \epsilon^{-3})$\\

\cline{4-8}

&&&$\kappa=1$&$\mathcal{O}(\epsilon^{-1})$&$\mathcal{O}(\epsilon^{-1})$&0&$\mathcal{O}(\epsilon^{-1})$\\

\cline{2-8}

& \multirow{3}{*}{\begin{tabular}[c]{@{}l@{}}Cvx.\\  \end{tabular}} & \multirow{3}{*}{Cvx.} &\multicolumn{1}{c|}{not required}&\multicolumn{1}{c|}{asymptotic}&\multicolumn{1}{c|}{$\mathcal{O}(\epsilon^{-4})$} &\multicolumn{1}{c|}{0}&\multicolumn{1}{c|}{$\mathcal{O}(\epsilon^{-4})$}\\

\cline{4-8}

&&&$\kappa>1$&$\mathcal{O}(\epsilon^{-4\kappa})$&$\mathcal{O}(\epsilon^{-4})$&0&$\mathcal{O}(\epsilon^{-4})$\\

\cline{4-8}

&&&$\kappa=1$&$\mathcal{O}(\epsilon^{-2})$&$\mathcal{O}(\epsilon^{-2})$&0&$\mathcal{O}(\epsilon^{-2})$\\

\cline{2-8}

\hline

\multirow{3}{*}{\begin{tabular}[c]{@{}l@{}}{\RRRme{IPIR-FedAvg}}  \end{tabular}}& \multirow{3}{*}{\begin{tabular}[c]{@{}l@{}}\begin{tabular}{c}
Non-Cvx.,\\
\RRRme{globally known}
\end{tabular}\\ \end{tabular}} & \multirow{3}{*}{\begin{tabular}[c]{@{}l@{}}Cvx.\end{tabular}} &\multicolumn{1}{c|}{weak sharp}& \multicolumn{2}{c|}{\RRRme{Stationarity} } & \multicolumn{2}{c|}{$\mathbb{E}[\hbox{dist}(\bullet, X^*_h)]$}\\

\cline{4-8}

&&&$\kappa>1$& \multicolumn{2}{c|}{${\tilde{\mathcal{O}}( \epsilon^{-24\kappa/5})}$ } & \multicolumn{2}{c|}{${{\tilde{\mathcal{O}}}( \epsilon^\mje{{-24\kappa/5}})}$ } \\

\cline{4-8}

&&&$\kappa=1$& \multicolumn{2}{c|}{${\tilde{\mathcal{O}}( \epsilon^{-1})}$ } & \multicolumn{2}{c|}{${\tilde{\mathcal{O}}( \epsilon^{-1})}$ } \\

\cline{2-8}
\hline
\end{tabular}}}\label{Table:Comunication complexities}}

\end{table}


 \fyy{\noindent {\RRme{\bf Notation.}} Throughout, we let $x \in \mathbb{R}^n$ and $x^{\top}$ denote a column vector $x$ and its transpose, \me{respectively}. Given $p>0$, we let $\|\cdot\|_p$ denote the $\ell_p$-norm of a vector defined as $\|x\|_p= \left(\sum_{i=1}^n |x_i|^p\right)^{1/p}$. A continuously differentiable function $f:\mathbb{R}^n \to \mathbb{R}$ is called $\mu_f$-strongly convex if $f(x) \geq f(y)+\nabla f(y)\fyy{^\top}(x-y)+\frac{\mu_f}{2}\|x-y\|^2$ for all $x,y \in \mathbb{R}^n$. The Euclidean projection of vector $x$ onto a closed convex set $X$ is denoted by $ \Pi_{X}[x] $, where $\Pi_{X}[x]\triangleq  \mbox{arg}\min_{ y \in  X}\| x- y\|$. We let $\mbox{dist}(x,X)\triangleq \|x-\Pi_{X}[x]\|$ denote the distance of $x$ from the set $X$. Given a random variable $z \in \mathbb{R}^n$, we use $\mathbb{E}[z]$ to denote the expectation of the random variable. In addressing the convex program $\min_{x\in X} h(x)$, the optimal solution set (denoted by $X^*_h$) is said to be $\alpha$-weak sharp ($\alpha >0$) with the order $\kappa\geq 1$, if $h(x)-h^* \geq \alpha\, \mbox{dist}^\kappa(x,X^*_h)$ for all $x \in X$.} \me{We denote the Moore-Penrose pseudoinverse of a matrix by $(\bullet)^{\dagger}$ \RRRme{and use $\mathbb{P}[\bullet]$ to denote probability. We use $\mathcal{O}(\cdot)$ for big-O notation and $\tilde{\mathcal{O}}(\cdot)$ to suppress logarithmic factors.}}
     
\fy{
\section{\RRme{Proposed algorithms }}
\Rme{In this section, we \RRme{present the outline of} the algorithms for the three \RRme{settings considered in this paper, corresponding to strongly convex, convex, and nonconvex outer-level objective functions}. Algorithm~\ref{Alg:FEDAVG} is \RRme{proposed for the setting where} the outer-level function is either strongly convex or convex, while Algorithm~\ref{algorithm:ncvx} is \RRme{proposed for the setting in which} the outer-level function is nonconvex. \RRme{Next, we present the outline of} each algorithm in detail.}
\subsection{Strongly convex and convex settings: \RRRme{StR-FedAvg} Algorithm}
\Rme{\RRme{To} this end, let us define for each $i \in [N]$, the regularized local function $\tilde f_{\eta,i}(x,\rho_i):= \tilde h_i(x,\zeta_i) + \eta \tilde f_i(x,\xi_i)$ where $\eta>0$ is a regularization parameter and $\rho_i:=(\xi_i,\zeta_i)$. Throughout, we assume that $\xi_i$ and $\zeta_i$ are independent random variables. \RRme{We propose a \RRRme{self-tuned} regularized variant of the standard FedAvg~\cite{mcmahan2017communication}, namely} \RRRme{StR-FedAvg}, presented in Algorithm~\ref{Alg:FEDAVG}. The key distinction with the standard variant of this method is the use of regularization in local updates. {In \RRRme{StR-FedAvg}, each client updates its local model during the local update phase, and upon aggregation, each participating client initializes its local model with the server model.} \RRme{In Algorithm~\ref{Alg:FEDAVG}, the local stepsize $\gamma_\ell$ and the regularization parameter $\eta$ are prescribed as functions of the total number of communication rounds $R$. These choices are informed by the theoretical analysis and are designed to balance the trade-off between progress toward the optimal solution set of the inner problem and improvement with respect to the outer objective. The parameters $a$ and $b$ specify how $\gamma_\ell$ and $\eta$ scale with $R$, respectively, with $0 <b < a \leq 1$. In this way, the regularization parameter is tuned systematically according to the communication horizon rather than selected heuristically.}}

\begin{algorithm}
\caption{{\RRRme{StR-FedAvg}} }\label{Alg:FEDAVG}
\begin{algorithmic}[1]
\State {{\color{black}{\bf server input:}}} \RRme{initial global model $\bar{x}^0 \in \mathbb{R}^n$, global stepsize $\gamma_g \geq 1$, number of communication rounds $R\geq 1$, and number of local updates $K\geq 1$}
\RRme{\State {{\color{black}{\bf client $i$'s input:}} Let $0 < b < a \leq 1$ and $p \ge1$. 
\begin{tabular}[t]{@{}l@{}}
\qquad Convex setting: local stepsize $\gamma_\ell:=\frac{1}{\gamma_g K R^a}$ and $\eta:=\frac{1}{R^b}$.\\
\qquad Strongly convex setting: local stepsize $\gamma_\ell:=\frac{1}{\gamma_g K \mu_f^a R^a}$ and $\eta:=\frac{p\ln(R)}{\mu_f^b R^b}$.
\end{tabular}}}
 \For {each round \RRme{$r = 1,2 \ldots,R$}} 
\State sample clients $\RRme{\mathcal{S}_r} \subseteq \{1,2,\ldots,N\}$
\State {\bf communicate} \far{$\RRme{\bar{x}^{r-1}}$ to all clients $i \in \RRme{\mathcal{S}_r}$}
\For {$i\in \RRme{\mathcal{S}_r}$ {\bf in parallel} }
\State  initialize local model $\RRme{y_{i,k-1}^r:= \bar{x}^{r-1}}$
\For {$k=1,2,...,K$}
 
\State \RRme{Generate random samples $\xi_{i,k-1}^r$ and $\zeta_{i,k-1}^r$
\State 
$y_{i,k}^r : = y_{i,k-1}^r - \gamma_\ell (\eta\nabla \tilde{f}_i(y_{i,k-1}^r,\xi_{i,k-1}^r) +\nabla \tilde{h}_i(y_{i,k-1}^r,\zeta_{i,k-1}^r))$}
\EndFor
\State {\bf communicate} { $\RRme{\Delta y_{i}^r : = y_{i,K}^r - \bar{x}^{r-1}}$}
\EndFor
\State {$\RRme{\Delta \bar{x}^r := \tfrac{1}{S}\textstyle\sum_{i \in \mathcal{S}_r} \Delta y_{i}^r }$} \far{ and  {\RRme{$\bar{x}^{r+1} := \bar{x}^r + \gamma_g \Delta \bar{x}^r$}}  }
\EndFor
\end{algorithmic}\label{algorithm:R-FedAvg and R-SCAFFOLD}
\end{algorithm}

\subsection{Nonconvex setting: \Rme{IPIR-FedAvg Algorithm}}
Consider the following problem
\begin{align}\label{problem: main problem-ncvx} 
 \textstyle{\min_{y \RRme{\in \mathbb{R}^n}}} \  f(y)\quad &\hbox{s.t.}\ \  y \in \hbox{arg}\textstyle{\min_{x\RRme{\in \mathbb{R}^n}}}\  h(x):= \tfrac{1}{N} \textstyle{\sum_{i=1}^N}\mathbb{E}_{\zeta_i \in \tilde{\mathcal{D}}_i}[\tilde h_i(y,\zeta_i)].
\end{align}
Here, in the \fyy{inner-level} problem each client $i \in [N]$, with $[N]:=\{1,\ldots, N\}$, is associated  with a convex local loss function $\tilde h_i$, and a private dataset $\tilde{\mathcal{D}}_i$. Throughout, we use the notation $h_i(y) := \mathbb{E}[\tilde{h}_i(y,\zeta_i)] $. \RRme{A motivation arises in nonconvex over-parameterized learning, where the goal is to minimize a sparsity-inducing regularizer, such as $\ell_{0.5}$, in order to select an optimal model parameter.} This problem can be reformulated as $\min_{y \in X^*_h} f(y)$, where $X^*_h := \hbox{arg}\min_{x \in X} \ h(x)$. \RRme{Equivalently, it can be written as} $\min f(y) + \mathbb{I}_{X^*_h}(y)$ (cf.~\cite{qiu2023zeroth}), where $\mathbb{I}_{X^*_h}(y)$ is the indicator function \RRme{associated with the set $X^*_h$. Recall that the indicator function is} a nonsmooth extended-valued function. To address the nonsmoothness, \RRme{we} employ a Moreau smoothing of $\mathbb{I}_{X^*_h}(y)$. Recall \RRme{that}~\cite{qiu2023zeroth}, the $\lambda$-smoothed \mje{approximation} of $\mathbb{I}_{X^*_h}(y)$ is $\frac{1}{\lambda}\|y-\Pi_{X^*_h}\fyy{[y]}\|^2$, where $\fyy{\Pi_{X^*_h}[\bullet]}$ denotes the projection operator and $\lambda>0$ is a smoothing parameter. Thus, problem~\eqref{problem: main problem-ncvx} can be rewritten as follows
\begin{align}
\textstyle{\min_{y \in \mathbb{R}^n}} \quad \RRme{F_\lambda(y)} := f(y) + \tfrac{1}{\lambda}\|y-\Pi_{X^*_h}[y]\|^2.\label{problem: main problem-ncvx-approximation}
\end{align}
 \RRme{Note that, for sufficiently small $\lambda$, any Clarke stationary point of problem~\eqref{problem: main problem-ncvx-approximation} yields a $\lambda$-approximate Clarke stationary point of problem~\eqref{problem: main problem-ncvx}, as shown in~\cite[Eq.~(2)]{qiu2023zeroth}.} \mje{Such stationarity guarantees are standard in first-order nonconvex optimization~\cite{beck2017first}.} \RRme{In the next step, consider applying the gradient descent (GD) method for addressing problem~\eqref{problem: main problem-ncvx-approximation}, where the gradient of the objective function is given by} $\nabla f(y) + \tfrac{1}{\lambda}(y - \Pi_{X^*_h}[y])$. \RRme{A key challenge is that $\Pi_{X^*_h}[y]$ is unavailable. Let us define $g(x) \triangleq \frac{1}{2}\|x-y\|^2$ for a given $y$. Then, $\Pi_{X^*_h}[y]$ is the unique optimal solution to the projection problem given as}
\begin{align}
&\min \quad g(x)\quad \text{s.t.} \quad x \in \hbox{arg} \textstyle{\min_x} \quad h(x),\label{problem:peojection problem}
\end{align}
where $g(x)$ is a strongly convex function. As a result, we can inexactly apply Algorithm~\ref{algorithm:R-FedAvg and R-SCAFFOLD} to compute $\Pi_{X^*_h}\fyy{[y]}$. This connection between the nonconvex and strongly convex settings motivates the development of a two-loop scheme \Rme{(IPIR-FedAvg)} to address the nonconvex case, as outlined in Algorithm~\ref{algorithm:ncvx}. In this algorithm, \RRme{in the outer loop, we employ an inexact variant of the GD method for solving}~\eqref{problem: main problem-ncvx-approximation}. At each iteration of the outer loop, we employ \RRRme{StR-FedAvg} in Algorithm~\ref{algorithm:R-FedAvg and R-SCAFFOLD} to inexactly evaluate the solution of the projection problem \RRme{in the inner loop}. \Rme{We emphasize that the regularization parameter used to solve the projection subproblem is updated across outer iterations; in particular, at iteration $t$ we use $\eta_t$.} \mje{The choices of $\eta_t$ and $\tilde{\gamma}_{\ell_t}$ balance the competing effects of regularization bias and stochastic/heterogeneity errors in the projection step..In summary, IPIR-FedAvg separates the nonconvex solution-selection problem into two tasks: the outer loop performs a gradient step on the smoothed objective $F_\lambda$, while the inner loop approximately computes the projection $\Pi_{X_h^*}[y^t]$ by solving a strongly convex regularized subproblem using StR-FedAvg. Thus, the inner loop controls the accuracy of the projected direction, and the outer loop drives stationarity of the smoothed nonconvex objective.}

\begin{algorithm}
\caption{Inexact projected \Rme{iteratively} \far{regularized FedAvg \Rme{(IPIR-FedAvg)}}}
\begin{algorithmic}[1]
\State {\textbf {input}} \RRme{Total number of iterations $T$}, initial point $\RRme{y^0}$, Moreau smoothing parameter $\lambda$, and stepsize $\gamma$.
 \For {$t = 1,2, \ldots, T $} 
\State Call \RRRme{StR-FedAvg} in Algorithm~\ref{algorithm:R-FedAvg and R-SCAFFOLD} for \RRme{$R_t := t$ communication} rounds, \RRme{with regularization parameter $\eta_t:= \tfrac{2\ln(t)}{t^{5/12}}$ \mje{and stepsize $\mje{\tilde \gamma{_t}}:= \tfrac{1}{t^{7/12}}$}, and strong convexity parameter equal to $1$,} to evaluate an inexact solution to the projection problem, i.e., $ \Pi_{X^*_h}[y^t]$, denoted by \Rme{$x^t_{\eta_t}$}
\State $y^{t+1}:= y^t - \gamma (\nabla f(y^t) + \frac{1}{\lambda}(y^t - \Rme{x^t_{\eta_t}}))$
\EndFor
\end{algorithmic}\label{algorithm:ncvx}
\end{algorithm}

} 
\section{Main results}\label{section: Convergence}
\Rme{In this section, we provide communication-complexity guarantees for the strongly convex and convex cases, and error bounds for the nonconvex case.
\subsection{Strongly convex and convex cases}
In this subsection, we \RRme{consider problem \eqref{problem: main problem} and} first state the main assumptions and then present a lemma establishing the bounded gradient dissimilarity (BGD) property. We note that BGD is assumed in some state-of-the-art works (e.g.,~\cite{karimireddy2020scaffold}), whereas here we \RRme{show that it holds under some mild conditions}. We then derive complexity guarantees for \RRRme{StR-FedAvg} for solving problem~\eqref{problem: main problem}, covering both the strongly convex and convex cases of~$f$.
\begin{assumption}\em\label{assump:URS} For each $i\in[N]$, assume that $f_i$ is ($\mu_f$-strongly) convex and $L_f$-smooth, and that $h_i$ is convex and $L_h$-smooth. Moreover, assume that the set $X_h^* := \arg\min_{x\in \mathbb{R}^n} h(x)$ is nonempty, and that $\inf_{x\in\mathbb{R}^n} f(x) > -\infty$.
\end{assumption}}
\RRme{\begin{remark}\em
In Assumption~\ref{assump:URS}, the notation ``($\mu_f$-strongly) convex'' with $\mu_f \geq 0$ is used to cover both the strongly convex and convex cases. Throughout the paper, results for the strongly convex setting are stated under $\mu_f>0$, while results for the convex setting are stated under $\mu_f=0$.
\end{remark}}
\begin{assumption}\label{assumption:main3}\em
\RRme{A random realization of} \Rme{$\nabla \tilde{f}_i(x,\xi_i)$} is an unbiased stochastic gradient of $f_i$ and admits a variance bounded by $\sigma_f^2$. Similarly, \RRme{a random realization of} $\nabla \tilde{h}_i(x,\zeta_i)$ is an unbiased stochastic gradient of $h_i$ with a variance bounded by $\sigma_h^2$. 
\end{assumption}

\Rme{\begin{lemma}\em\label{assumption:main4}[Bounded gradient dissimilarity (BGD)]
\RRme{Consider problem~\ref{problem: main problem}.} Suppose Assumption~\ref{assump:URS} holds \RRme{for $\mu_f=0$}. Then for each $\RRme{\Psi\in\{f,h\}}$ there exist constants $G_\RRme{\Psi}\triangleq  \sqrt{\frac{2L_\RRme{\Psi}}{N}\Big(\RRme{\inf_{z}\RRme{\Psi}(z)}-\tfrac{1}{N}\sum_{i=1}^N \RRme{\Psi}_i(x_i^*)\Big)}$\RRme{, where $x_i^*\in\arg\min_x \Psi_i(x)$,} and $B_\RRme{\Psi} \triangleq1$ such that for all $x\in\mathbb{R}^n$, $\frac{1}{N}\sum_{i=1}^N \|\nabla \RRme{\Psi}_i(x)\|^2\le G_\RRme{\Psi}^2 + 2L_\RRme{\Psi} B_\Psi^2\bigl(\RRme{\Psi}(x)-\inf_{z}\RRme{\Psi}(z)\bigr).$
\end{lemma}}
\begin{lemma}\em\label{remark:bound in terms of the obj} \RRme{Consider problem~\ref{problem: main problem} under Assumption~\ref{assump:URS} with $\mu_f=0$.} The BGD condition holds for the regularized functions, i.e., $\tfrac{1}{N}\textstyle\sum_{i=1}^N\|\nabla h_{i}(x)+\eta \nabla f_{i}(x)\|^2\le G^2+B^2\left(f_\eta(x)-f_\eta^*\right),$
where $G^2 := 2  G_f^2+4L_f B_f^2(f^*-\inf_{x \in \mathbb{R}^n} f(x))+2 G_h^2+4L_h B_h^2(h^*-\inf_{x \in \mathbb{R}^n} h(x))$ and $B^2:= max\{4\eta L_f B_f^2,4L_h B_h^2\},$ under the assumption that $f$ and $h$ are bounded below. 
\end{lemma}
\Rme{The following theorems establish communication complexity bounds for StR-FedAvg for the strongly convex and convex outer-level settings, respectively.}

\begin{theorem}\em[Strongly convex \me{outer-level}]\label{thm:thm2} Consider \RRRme{StR-FedAvg} given in Algorithm~\ref{Alg:FEDAVG}. Let Assumptions~\ref{assump:URS} and~\ref{assumption:main3} hold \RRme{with $\mu_f>0$}, \RRme{$\gamma_g\geq 1$,} $\gamma_l \le \tfrac{1}{16(L_h+\RRme{\hat\eta} L_f)^2(1+B^2)K\gamma_g}$\RRme{, where $\hat\eta\triangleq \frac{p}{\exp(1)b \mu_f^b}$, for some \RRme{$0< b<1$}, and define {$\tilde \gamma := \gamma_l \gamma_g K$}}.

\noindent [Case i] For $\tilde \gamma:= {1}/{\mu_f^aR^a}$, $\eta:= {p\RRme{\ln}(R)}/{\mu_f^bR^b}$, \RRme{where $R \;\ge\;
\max\!\left\{\tfrac{\bigl(64(1+B^2)L_h^2\bigr)^{1/a}}{\mu_f},\left(\frac{32\sqrt{1+B^2}\,pL_f}{\exp(1)(a+2b)\,\mu_f^{(a+2b)/2}}\right)^{\frac{4}{a+2b}}\right\}$,} for some \RRme{$0< b<a\le1$}, we have

~\\
 (i-1) \ $R^{up}_f=\mathcal{O}\left(\sqrt[\tfrac{p}{2}]{\tfrac{\mu_fQ^2}{\epsilon}}+\sqrt[2a-b]{\tfrac{W}{p\mu_f^{2a-b}\epsilon\gamma_g^2}}+\sqrt[a-b]{\tfrac{Y}{p\mu_f^{a-b}\epsilon S}}\right).$
~\\
(i-2) \ $R^{up}_h=\mathcal{O}\left(\sqrt[\tfrac{bp}{2}]{\tfrac{\mu_f^{1-b}pQ^2}{\epsilon}}+\sqrt[2a]{\tfrac{W}{\mu_f^{2a}\epsilon\gamma_g^2}}+\sqrt[a]{\tfrac{Y}{\mu_f^{a}\epsilon S}}+\sqrt[b]{\tfrac{pM}{\mu_f^b\epsilon}}\right).$

\noindent [Case ii] In the case where $X_h^*$ is $\alpha$-weak sharp of the order of $\kappa> 1$, in addition to (i-1) and (i-2),  
~\\
 (ii-1)\ $
R^{lo}_f=\mathcal{O}\left(\|\nabla f(x^*)\|^\kappa\left(\sqrt[\tfrac{bp}{2}]{\tfrac{\mu_f^{1-b}pQ^2}{\alpha\epsilon^\kappa}}+\sqrt[2a]{\tfrac{W}{\mu_f^{2a}\alpha\epsilon^\kappa\gamma_g^2}}+\sqrt[a]{\tfrac{Y}{\mu_f^{a}\alpha\epsilon^\kappa S}}+\sqrt[b]{\tfrac{pM}{\mu_f^b\alpha\epsilon^\kappa}}\right)\right).$

\noindent [Case iii] When $X^*_h$ is weak sharp with $\kappa=1$, for $\tilde \gamma:= \tfrac{p\RRme{\ln}(R)}{\mu_fR}$ and for any $\eta\le\tfrac{\alpha}{2\|\nabla f(x^*)\|}$,  we have
~\\
 (iii-1)\ $
R^{up}_f=\mathcal{O}\left(\sqrt[\tfrac{p}{2}]{\tfrac{\mu_fQ^2}{{\epsilon}}}+\sqrt{\tfrac{pW}{\eta\mu_f^{2}\epsilon\gamma_g^2}}+\tfrac{pY}{\eta\mu_f\epsilon S}\right).
$
~\\
 (iii-2)\ $
R^{up}_h=\mathcal{O}\left(\sqrt[\tfrac{p}{2}]{\tfrac{\eta\mu_fpQ^2}{\epsilon}}+\sqrt{\tfrac{pW}{\mu_f^{2}\epsilon\gamma_g^2}}+\tfrac{pY}{\mu_f\epsilon S}\right).
$
~\\(iii-3)\ $
R^{lo}_f=\mathcal{O}\left(\tfrac{\|\nabla f(x^*)\|}{\alpha}\left(\sqrt[\tfrac{p}{2}]{\tfrac{\eta\mu_fpQ^2}{\epsilon}}+\sqrt{\tfrac{pW}{\mu_f^{2}\epsilon\gamma_g^2}}+\tfrac{pY}{\mu_f\epsilon S}\right)\right).
$
~\\(iii-4)\ $
R^{up}_{\|\RRme{\bar x}^R-x^*\|^2}=\mathcal{O}\left(\sqrt[\tfrac{p}{2}]{\tfrac{ pQ^2}{{\epsilon}}}+\sqrt{\tfrac{pW}{\eta\mu_f^{3}\epsilon \gamma_g^2}}+\tfrac{pY}{\eta\mu_f^2\epsilon S}\right).
$

where $Q:= \|\RRme{\bar{x}}^0-x_\eta^*\|$, $W:= \tfrac{9(L_h+\RRme{\hat\eta} L_f)(2\RRme{\hat\eta}^2\sigma_f^2+2\sigma_h^2)}{K}+{18(L_h+\RRme{\hat\eta} L_f)G^2}$, $Y:=\tfrac{(2\RRme{\hat\eta}^2\sigma_f^2+2\sigma_h^2)}{K}+4\left(1-\tfrac{S}{N}\right)G^2$, and $p\ge 1$.
\end{theorem}

\begin{theorem}\em[Convex \me{outer-level}]\label{thm:thm3} Consider the assumptions in Theorem~\ref{thm:thm2} with $\mu_f=0$. The bounds are updated as follows.

\noindent [Case i] For $\tilde \gamma:= \tfrac{1}{R^a}$, $\eta:= \tfrac{1}{R^b}$, \RRme{where $R \;\ge\;\max\!\left\{\bigl(64(1+B^2)L_h^2\bigr)^{1/a},\;\bigl(64(1+B^2)L_f^2\bigr)^{1/(a+2b)}\right\}$,} for some \RRme{$0<b<a\le1$} we have
~\\ (i-1) \ $
R^{up}_f=\mathcal{O}\left(\sqrt[1-(a+b)]{\tfrac{\tilde Q^2}{\epsilon}}+\sqrt[2a-b]{\tfrac{W}{\epsilon\gamma_g^2}}+\sqrt[a-b]{\tfrac{Y}{\epsilon S}}\right).
$
~\\(i-2) \ $
R^{up}_h=\mathcal{O}\left(\sqrt[{1-a}]{\tfrac{\tilde Q^2}{\epsilon}}+\sqrt[2a]{\tfrac{W}{\epsilon\gamma_g^2}}+\sqrt[a]{\tfrac{Y}{\epsilon S}}+\sqrt[b]{\tfrac{M}{\epsilon}}\right).
$

\noindent [Case ii] In the case where $X_h^*$ is $\alpha$-weak sharp of the order of $\kappa$, where $\alpha >0$ and  $\kappa> 1$, and for $\tilde \gamma:= \tfrac{1}{R^a}$, $\eta:= \tfrac{1}{R^b}$, for some $0\le a,b\le1$ we get
~\\$
R^{lo}_f=\mathcal{O}\left(\|\nabla f(x^*)\|^\kappa \left(\sqrt[1-a]{\tfrac{\tilde Q^2}{\alpha\epsilon^\kappa}}+\sqrt[2a]{\tfrac{W}{\alpha\epsilon^\kappa\gamma_g^2}}+\sqrt[a]{\tfrac{Y}{\alpha\epsilon^\kappa S}}+\sqrt[b]{\tfrac{M}{\alpha\epsilon^\kappa}}\right)\right).
$

\noindent [Case iii]  When $X^*_h$ is weak sharp with $\kappa=1$, for $\tilde \gamma:= \tfrac{1}{R^a}$ and for any $\eta\le\tfrac{\alpha}{2\|\nabla f(x^*)\|}$  we have
~\\ (iii-1) \ $
R^{up}_f=\mathcal{O}\left(\sqrt[1-a]{\tfrac{\tilde Q^2}{{\eta \epsilon}}}+\sqrt[2a]{\tfrac{W}{\eta\epsilon\gamma_g^2}}+\sqrt[a]{\tfrac{Y}{\eta\epsilon S}}\right).
$
~\\ (iii-2) \ $
R^{up}_h=\mathcal{O}\left(\sqrt[1-a]{\tfrac{\tilde Q^2}{{ \epsilon}}}+\sqrt[2a]{\tfrac{W}{\epsilon\gamma_g^2}}+\sqrt[a]{\tfrac{Y}{\epsilon S}}\right).
$
~\\ (iii-3) \ $
R^{lo}_f=\mathcal{O}\left(\tfrac{ \|\nabla f(x^*)\|}{\alpha}\left(\sqrt[1-a]{\tfrac{\tilde Q^2}{{ \epsilon}}}+\sqrt[2a]{\tfrac{W}{\epsilon\gamma_g^2}}+\sqrt[a]{\tfrac{Y}{\epsilon S}}\right)\right),
$
~\\where $\tilde Q^2:= \mathbb{E}\left[\|\RRme{\bar{x}}^{0}-x_\eta^*\|^2\right]-\mathbb{E}\left[\|\RRme{\bar{x}}^R-x_\eta^*\|^2\right]$, $W:= \tfrac{9(L_h+\RRme{\hat\eta} L_f)(2\RRme{\hat\eta}^2\sigma_f^2+2\sigma_h^2)}{K}+{18(L_h+\RRme{\hat\eta} L_f)G^2}$, $Y:=\tfrac{(2\RRme{\hat\eta}^2\sigma_f^2+2\sigma_h^2)}{K}+4\left(1-\tfrac{S}{N}\right)G^2$, and $p\ge 1$.
\end{theorem}
 \Rme{\begin{remark}\em{A key implication} in the above results is that our analysis leads to the prescribed rules for the regularization parameter $\eta$ and the stepsize. Indeed, this provides {guidance} on how one should tune the regularization parameter in terms of the maximum number of communication rounds to balance the trade-off between the errors in the primary metric $h$ and the secondary metric $f$. \mje{Intuitively, the local stepsize $\gamma_\ell$ controls the convergence toward the solution set $X_h^*$, while the regularization parameter $\eta$ controls the influence of the outer objective in selecting among these solutions. The choices of $a$ and $b$ ensure that these two effects are balanced so that neither objective dominates the other.} Notably, the bounds in Table~\ref{Table:Comunication complexities} are obtained by setting $a:= \tfrac{2}{3}$ and $b:= \tfrac{1}{3}$ in the strongly convex case, and $a:= \tfrac{1}{2}$ and $b:= \tfrac{1}{4}$ in the convex case.
\end{remark}}
\mje{\begin{remark}\em
In Theorem~\ref{thm:thm2}, the term $Q$ captures the initialization error, while $W$ and $Y$ capture the effects of stochastic noise and client heterogeneity through $\sigma_f^2$, $\sigma_h^2$, and $G^2$. In particular, the term $(1-S/N)G^2$ in $Y$ reflects partial client participation and vanishes under full participation.
\end{remark}}

\subsection{Nonconvex case}
\Rme{In this subsection, we first state the main assumptions and then present a lemma establishing the bounded gradient dissimilarity (BGD) property. We then derive complexity guarantees for IPIR-FedAvg for solving problem~\eqref{problem: main problem-ncvx}, covering the nonconvex case of~$f$.
 \begin{assumption}\em\label{assump:assump1-ncvx} Consider problem~\eqref{problem: main problem-ncvx}. Let function $f$ be $L_f$-smooth,  and $h_i$ be convex and $L_h$-smooth for all $i \in [N]$. Suppose the set $X^*_h$ is nonempty. 
\end{assumption}}

\begin{remark}\em
As previously mentioned, we apply the Algorithm~\ref{algorithm:ncvx} to solve the projection problem~\eqref{problem:peojection problem}. The regularized objective of this problem is defined as $g_\eta(x)\triangleq \eta g(x)+h(x)$. Notably, function $g$ is $1$-strongly convex and $1$-smooth. Consequently, wherever the parameters $\mu_f$ and $L_f$ appear, we set them to $1$ accordingly. In addition, based on Assumption~\ref{assump:assump1-ncvx}, function $\RRme{F_\lambda}$ in problem~\eqref{problem: main problem-ncvx-approximation} is $L_f+\frac{2}{\lambda}$-smooth.
\end{remark}
\begin{lemma}\em\label{remark:bound in terms of the obj-ncvx} \RRme{Consider problem~\eqref{problem:peojection problem}.} The BGD condition holds for the regularized function, i.e., $\tfrac{1}{N}\textstyle\sum_{i=1}^N\|\nabla h_{i}(x)+\eta \nabla g(x)\|^2\le G_{ncvx}^2+B_{ncvx}^2\left(g_\eta(x)-g_\eta^*\right),$
where $G_{ncvx}^2 := 2  G_g^2+4 B_g^2(g^*-\inf_{x \in \mathbb{R}^n} g(x))+2 G_h^2+4L_h B_h^2(h^*-\inf_{x \in \mathbb{R}^n} h(x))$ and $B_{ncvx}^2:= max\{4\eta B_g^2,4L_h B_h^2\},$ under the assumption that $g$ and $h$ are bounded below. 
\end{lemma}
\Rme{\begin{theorem}\em\label{thm:thm ncvx} Consider problem~\eqref{problem: main problem-ncvx} and Algorithm~\ref{algorithm:ncvx}. Let function $f$ be nonconvex and $h_i$ be convex for all $i \in [N]$. Let Assumptions~\ref{assump:assump1-ncvx} holds. Let $T^*\in \{0,\ldots,T-1\}$ be a random variable such that $\mathbb{P}[T^*=i]=\frac{1}{T},\ \text{for all } i\in\{0,\ldots,T-1\}$. Then, the following results hold true.

\noindent[Case i] Suppose $X^*_h$ is $\alpha$-weak sharp with order $\kappa >1$. Then, for $\tilde \gamma_t:= \tfrac{1}{t^{7/12}}$ and $\eta_t:= \tfrac{2\RRme{\ln}(t)}{t^{5/12}}$, \RRme{with $\hat\eta\triangleq\max_t\eta_t=11$,} the following error bounds hold. 
\begin{align*}
\text{\noindent (i-1) }&\mathbb{E}[\|\nabla \RRme{F_\lambda}(y^{T^*})\|^2] \leq  \tfrac{2\gamma^{-1}\left(\RRme{F_\lambda}(y^0)- \RRme{F^*_\lambda}\right)}{T}+\left(\tfrac{2}{\RRme{\lambda^2}}+\tfrac{2\gamma(L_f\lambda+2)}{\lambda^3}\right)\mathbb{E}\left[\|\RRme{\bar x}^{0}-x^{{t}*}\|^2\right]\tfrac{\RRme{1+\ln}(T)}{T}\\
&+\left(\tfrac{2}{\RRme{\lambda^2}}+\tfrac{2\gamma(L_f\lambda+2)}{\lambda^3}\right)\left(\tfrac{(2\RRme{\hat\eta}^2\sigma_f^2+2\sigma_h^2)}{KS}+\left(1-\tfrac{S}{N}\right)\tfrac{4}{S}G_{ncvx}^2\right)\RRRme{\tfrac{6}{5T^\frac{1}{6}\ln(T)}}\\
&+\left(\tfrac{2}{\RRme{\lambda^2}}+\tfrac{2\gamma(L_f\lambda+2)}{\lambda^3}\right)\left(\tfrac{9(L_h+\RRme{\hat\eta} )(2\RRme{\hat\eta}^2\sigma_f^2+2\sigma_h^2)}{K\Rme{\gamma_{g_t}^2}}+\tfrac{18(L_h+\RRme{\hat\eta} )G_{ncvx}^2}{{\Rme{\gamma_{g_t}^2}}}\right)\RRRme{\tfrac{3}{2T^{3/4}\RRme{\ln}(T)}}+\RRRme{\tfrac{(4M)^{1/\kappa}(24\kappa-5)}{12\kappa-5}\tfrac{(\ln T)^{1/\kappa}}{T^{5/(12\kappa)}}}\\
&+\left(\tfrac{2}{\RRme{\lambda^2}}+\tfrac{2\gamma(L_f\lambda+2)}{\lambda^3}\right)\tfrac{\|\nabla f(x^{t*})\|}{\alpha^\frac{1}{\kappa}}\left(\mathbb{E}\left[\|\RRme{\bar x}^{0}-x^{{t}*}\|^2\right]\right)^{\tfrac{1}{\kappa}}\RRRme{
\tfrac{2(4)^{1/\kappa}(\ln T+1)^{1/\kappa}}{T^{17/12\kappa}}}\\
&+\left(\tfrac{2}{\RRme{\lambda^2}}+\tfrac{2\gamma(L_f\lambda+2)}{\lambda^3}\right)\tfrac{\|\nabla f(x^{t*})\|}{\alpha^\frac{1}{\kappa}}\left(\tfrac{2(2\RRme{\hat\eta}^2\sigma_f^2+2\sigma_h^2)}{KS}+2\left(1-\tfrac{S}{N}\right)\tfrac{4}{S}G_{ncvx}^2\right)^\frac{1}{\kappa}\RRRme{\tfrac{2^{1/\kappa}}{1-\tfrac{7}{12\kappa}}\tfrac{1}{T^{7/(12\kappa)}}}\\
&+\left(\tfrac{2}{\RRme{\lambda^2}}+\tfrac{2\gamma(L_f\lambda+2)}{\lambda^3}\right)\tfrac{\|\nabla f(x^{t*})\|}{\alpha^\frac{1}{\kappa}}\left(\tfrac{18(L_h+\RRme{\hat\eta})(2\RRme{\hat\eta}^2\sigma_f^2+2\sigma_h^2)}{K\Rme{\gamma_{g_t}^2}}+\tfrac{36(L_h+\RRme{\hat\eta} )G_{ncvx}^2}{{\Rme{\gamma_{g_t}^2}}}\right)^\frac{1}{\kappa}\RRRme{\tfrac{1}{T^{1/\kappa}}}.\\
\text{\noindent (i-2) }&\Rme{\mathbb{E}[\hbox{dist}({\RRme{\bar x}}^t_{\eta}, X^*_h)]}\le\tfrac{1}{\alpha}\left(\tfrac{\RRme{\ln}(t)}{t^{{5}/{12}}\left(t-1\right)}\mathbb{E}\left[\|\RRme{\bar x}^{0}-x^{t*}\|^2\right]+\tfrac{1}{t^{{7}/{12}}}\left(\tfrac{(2\RRme{\hat\eta}^2\sigma_f^2+2\sigma_h^2)}{KS}+\left(1-\tfrac{S}{N}\right)\tfrac{4G_{ncvx}^2}{S}\right)\right.\\
&\left.+\tfrac{1}{t^{{7}/{6}}}\left(\tfrac{9(L_h+\RRme{\hat\eta} )(2\RRme{\hat\eta}^2\sigma_f^2+2\sigma_h^2)}{K\Rme{\gamma_{g_t}^2}}+\tfrac{18(L_h+\RRme{\hat\eta} )G_{ncvx}^2}{{\Rme{\gamma_{g_t}^2}}}\right)+\tfrac{\RRme{\ln}(t)}{t^{{5}/{12}}}4 M\right)^{1/\kappa}.
\end{align*}
\noindent [Case ii] If $X^*_h$ is $\alpha$-weak sharp with order $\kappa=1$, $\tilde \gamma_t:= \tfrac{2\RRme{\ln}(t)}{t}$, and $\eta_t:=\eta\le{\alpha}/{2\|\nabla F_\lambda(x^{t*})\|}$, \RRme{with $\hat\eta=11$,} the following error bounds hold.
\begin{align*}
\text{\noindent (ii-1) }\mathbb{E}[\|\nabla \RRme{F_\lambda}(y^{T^*})\|^2]\leq\!  &\frac{2\gamma^{-1}\left(\RRme{F_\lambda}(\RRme{y}^0)- \RRme{F^*_\lambda}\right)}{T}+\left(\tfrac{2}{\RRme{ \eta\lambda^2}}+\tfrac{2\gamma(L_f\lambda+2)}{ \eta\lambda^3}\right)\mathbb{E}\left[\|\bar{x}^{0}-x^{{t}*}\|^2\right]\tfrac{2\ln(T)}{T}\\
&+\left(\tfrac{8}{\RRme{ \eta\lambda^2}}+\tfrac{8\gamma(L_f\lambda+2)}{ \eta\lambda^3}\right)\left(\tfrac{(2\hat\eta^2\sigma_f^2+2\sigma_h^2)}{KS}+\tfrac{(N-S)4G_{ncvx}^2}{NS}\right)\tfrac{\left(\ln(T+2)\right)^2}{T}\\
&+\left(\tfrac{16}{\RRme{ \eta\lambda^2}}+\tfrac{16\gamma(L_f\lambda+2)}{ \eta\lambda^3}\right)\left(\tfrac{9(L_h+\hat\eta )(2\hat\eta^2\sigma_f^2+2\sigma_h^2)}{K\gamma_g^2}+\tfrac{18(L_h+\hat\eta )G_{ncvx}^2}{{\gamma_g^2}}\right)\tfrac{2}{T}.\\
\text{\noindent (ii-2) } \mathbb{E}[\hbox{dist}({\RRme{\bar x}}^t_{\eta_t}, X^*_h)]\le&\tfrac{\eta}{ \alpha t}\mathbb{E}\left[\|\RRme{\bar x}^{0}-x^{t*}\|^2 \right]+\tfrac{4\RRme{\ln}(t)}{t}\left(\tfrac{(2\hat\eta^2\sigma_f^2+2\sigma_h^2)}{KS}+\tfrac{4G^2(N-S)}{NS}\right)\\
&+\tfrac{8\left(\RRme{\ln}(t)\right)^2}{t^2}\left(\tfrac{9(L_h+\hat\eta L_f)(2\hat\eta^2\sigma_f^2+2\sigma_h^2)}{K\gamma_{g_t}^2}+\tfrac{18(L_h+\hat\eta L_f)G^2}{{\gamma_{g_t}^2}}\right).
\end{align*}
\end{theorem}}
\RRRme{\begin{remark}\em
In the first part of each case of Theorem~\ref{thm:thm ncvx}, we establish the iteration complexity bound for the stationarity metric. At the $t$-th iteration of the outer loop in Algorithm~\ref{algorithm:ncvx}, the StR-FedAvg method is employed for $t$ communication rounds. Therefore, the total communication complexity is $\sum_{t=0}^{T_\epsilon-1} t  = \mathcal{O}(T_\epsilon^2)$, where $T_\epsilon$ denotes the number of outer iterations required to guarantee $\mathbb{E}[\|\nabla F_\lambda(y^{T^*})\|^2] \le \epsilon$. The communication complexity bounds for the stationarity metric reported in Table~\ref{Table:Comunication complexities} then follow directly. \mje{The same cumulative complexity argument also applies to the feasibility guarantees. The higher complexity arises from the nested bilevel structure of the proposed nonconvex framework relative to standard single-level nonconvex methods.}
\end{remark}}
~\\
{\mje{{\bf Weak sharpness property.} The notion of {\it weak sharp minima} was first introduced in~\cite{burke1993weak} to account for optimization problems with nonunique solutions. The set $X_h^*$ is said to be $\alpha$-weak sharp with order $\kappa\geq 1$ if
\[
h(x)-h^* \geq \alpha\, \mathrm{dist}^\kappa(x,X_h^*),
\qquad
\forall x\in X.
\]
Several important optimization and machine learning problems are known to satisfy weak sharpness conditions. Below, we provide representative examples arising in quadratic optimization, convex feasibility problems, and logistic regression.

\begin{example}\em[Weak sharpness in quadratic optimization]
Consider the constrained quadratic optimization problem
\begin{align}
&\text{minimize}_x \quad h(x)=x^\top \mathbf{A}x-2b^\top x \label{problem:quadratic example}\\
&\text{subject to} \quad \|x\|\le1,\notag
\end{align}
where $\mathbf{A}$ is a nonzero, real symmetric matrix of size $n \times n$, and $b\in \mathbb{R}^n$. Assume that $\lambda$ is the smallest eigenvalue of $\mathbf{A}$. The result in \cite[Lemma 3.7]{jiang2022holderian} establishes that if $\lambda < 0$, or if $\lambda = 0$ and $\mathbf{A}^\dagger b < 1$, then the optimal solution set satisfies the weak sharpness property with order $\kappa = 2$. Moreover, when $\lambda < 0$, problem~\eqref{problem:quadratic example} is nonconvex. \cite[Lemma 3.8]{jiang2022holderian} further shows that, in the nonconvex case, the quadratic function still satisfies the weak sharpness property with order $\kappa = 2$.
\end{example}

\begin{example}\em[Weak sharpness in convex feasibility problems]
Weak sharpness conditions naturally arise in convex feasibility and projection-based optimization problems appearing in distributed optimization and control applications~\cite{bauschke1996projection,lee2015asynchronous}. Consider the convex feasibility problem
\[
\text{find } x\in C:=\bigcap_{j=1}^m C_j,
\]
where $C_1,\ldots,C_m$ are closed convex sets. A standard reformulation used in projection and random projection methods is
\[
\min_{x\in\mathbb{R}^d}
\quad
h(x):=
\frac12
\sum_{j=1}^m
\mathrm{dist}^2(x,C_j),
\]
where $\mathrm{dist}(x,C_j):=\inf_{y\in C_j}\|x-y\|$.
The optimal solution set satisfies
\[
X_h^*=C,
\qquad
h^*=0.
\]
Suppose that the collection $\{C_j\}_{j=1}^m$ satisfies the linear regularity condition~\cite{bauschke1996projection}, i.e., there exists $\beta>0$ such that
\[
\mathrm{dist}(x,C)^2
\le
\beta
\sum_{j=1}^m
\mathrm{dist}^2(x,C_j),
\qquad
\forall x\in\mathbb{R}^d.
\]
Then, we obtain
\[
h(x)-h^*
=
\frac12
\sum_{j=1}^m
\mathrm{dist}^2(x,C_j)
\ge
\frac{1}{2\beta}
\mathrm{dist}^2(x,C).
\]
Since $X_h^*=C$, we obtain
\[
h(x)-h^*
\ge
\frac{1}{2\beta}
\mathrm{dist}^2(x,X_h^*).
\]
Therefore, the weak sharpness condition holds with order $\kappa=2$ and $\alpha=\frac{1}{2\beta}$.
\end{example}

\begin{example}\em[Weak sharpness of logistic loss on compact domains]
Consider the empirical logistic loss over a compact convex set $X\subset \mathbb{R}^d$:
\[
h(x):=\frac{1}{N}\sum_{i=1}^N \log\bigl(1+\exp(-b_i a_i^\top x)\bigr),
\qquad x\in X,
\]
where $b_i\in\{-1,1\}$ and $a_i\in\mathbb{R}^d$. Suppose that the data matrix ${\bf A}\in\mathbb{R}^{N\times d}$, whose $i$-th row is $a_i^\top$, has full column rank. Then the weak sharpness condition holds on $X$ with order $\kappa=2$.
\end{example}

\begin{proof}
The Hessian of $h$ is given by
\[
\nabla^2 h(x)
=
\frac{1}{N}\sum_{i=1}^N
\sigma(b_i a_i^\top x)\bigl(1-\sigma(b_i a_i^\top x)\bigr)a_i a_i^\top,
\]
where $\sigma(s):=\frac{1}{1+\exp(-s)}$ denotes the sigmoid function. Since $X$ is compact and the mappings $x\mapsto b_i a_i^\top x$ are continuous, the set $\{b_i a_i^\top x:\ x\in X,\ i\in[N]\}$ is bounded. Hence, there exists $\mathcal{B}>0$ such that
\[
|b_i a_i^\top x|\le \mathcal{B},
\qquad
\forall x\in X,\ \forall i\in[N].
\]
Moreover, the function $s\mapsto \sigma(s)(1-\sigma(s))$ is continuous and strictly positive for every finite $s$. Therefore, over the compact interval $[-\mathcal{B},\mathcal{B}]$, it admits a positive lower bound
\[
c_\mathcal{B}:=\min_{|s|\le \mathcal{B}}\sigma(s)(1-\sigma(s))>0.
\]
Therefore, we obtain
\[
\sigma(b_i a_i^\top x)\bigl(1-\sigma(b_i a_i^\top x)\bigr)
\ge c_\mathcal{B},
\qquad
\forall x\in X,\ \forall i\in[N].
\]
Substituting this bound into the Hessian expression yields
\[
\nabla^2 h(x)
\succeq
\frac{c_\mathcal{B}}{N}
\sum_{i=1}^N a_i a_i^\top.
\]
Since the rows of ${\bf A}$ are given by $a_i^\top$, we have $\sum_{i=1}^N a_i a_i^\top={\bf A}^\top {\bf A}$. Thus, we obtain
\[
\nabla^2 h(x)
\succeq
\frac{c_\mathcal{B}}{N}{\bf A}^\top {\bf A}.
\]
Because ${\bf A}$ has full column rank, the matrix ${\bf A}^\top {\bf A}$ is positive definite. Hence,
\[
\lambda_{\min}(A^\top A)>0,
\]
and therefore
\[
\nabla^2 h(x)\succeq \mu I,
\qquad
\mu:=\frac{c_\mathcal{B}}{N}\lambda_{\min}({\bf A}^\top {\bf A})>0.
\]
This shows that $h$ is $\mu$-strongly convex on $X$. Let $x^*$ denote the unique minimizer of $h$ over X. By $\mu$-strong convexity, for all $x\in X$, we obtain
\[
h(x)
\ge
h(x^*)+\langle \nabla h(x^*),x-x^*\rangle
+\frac{\mu}{2}\|x-x^*\|^2.
\]
Since $x^*$ minimizes $h$ over $X$, we have $\langle \nabla h(x^*),x-x^*\rangle \ge 0$, and therefore
\[
h(x)-h(x^*)
\ge
\frac{\mu}{2}\|x-x^*\|^2,
\qquad
\forall x\in X.
\]
Since $X_h^*=\{x^*\}$, we obtain
\[
h(x)-h^*
\ge
\frac{\mu}{2}\operatorname{dist}^2(x,X_h^*).
\]
Therefore, the weak sharpness condition holds with order $\kappa=2$ and $\alpha=\frac{\mu}{2}$. 
\end{proof}}

\begin{remark}\em
Although the proof of the previous example relies on the strong convexity induced by the full column-rank assumption on ${\bf A}$, the proposed framework is primarily motivated by ill-posed and over-parameterized problems, where the inner-level objective may exhibit severe ill-conditioning or admit multiple minimizers. In such settings, even when strong convexity may hold theoretically, the problem can remain numerically sensitive.
\end{remark}

\section{Convergence analysis}
\Rme{In this section, we provide a detailed analysis of each case separately to establish the convergence and communication guarantees stated in Section~\ref{section: Convergence}.}
\subsection{Strogly convex and convex cases}
\Rme{{In this subsection, we analyze the \RRRme{StR-FedAvg} method for solving problem~\eqref{problem: main problem} in the cases where the outer-level loss function is strongly convex and convex. We provide a detailed analysis showing how Theorems~\ref{thm:thm2} and~\ref{thm:thm3} are derived.}}}

\fy{
\vspace{-6pt}
\begin{lemma}[{\cite[Lemma 5]{karimireddy2020scaffold}}]\em\label{lemma:bound for f eta and z-y} Suppose Assumption~\ref{assump:URS} holds with $\mu_f>0$. Then, for any \RRme{$x,y \in \mathbb{R}^n$}, we have
\vspace{-6pt}
\[\langle \nabla f_\eta(x),z-y\rangle\ge f_\eta(z)-f_\eta(y)+\tfrac{\eta \mu_f}{4}\|y-z\|^2-(L_h+\eta L_f)\|z-x\|^2.\]
\end{lemma}
}

\begin{lemma}[{\cite[Lemma 7]{karimireddy2020scaffold}}]\em\label{lemma:first bound for x-x*}
Let $\RRme{\bar{x}}^r$ be generated by \RRRme{StR-FedAvg} in Algorithm~\ref{Alg:FEDAVG}. Let Assumptions~\ref{assump:URS} and~\ref{assumption:main3} hold where $\mu_f>0$. Let $\gamma_l\le\tfrac{1}{16K \gamma_g(L_h+\eta L_f)^2(1+B^2)}$ and $\tilde \gamma := \gamma_l \gamma_g K$, then the following holds
\vspace{-6pt}
\begin{align*}
\mathbb{E}\left[\|\RRme{\bar{x}}^r-x_\eta^*\|^2\right]\le&\left(1-\tfrac{\eta \mu_f \tilde \gamma}{2}\right)\mathbb{E}\left[\|\RRme{\bar{x}}^{r-1}-x_\eta^*\|^2\right]-\tilde \gamma\mathbb{E}\left[(f_\eta(\RRme{\bar{x}}^{r-1})-f_\eta(x_\eta^*))\right]\\
&+3\tilde \gamma(L_h+\eta L_f)\varepsilon_r+\tfrac{\tilde \gamma^2(2\eta^2\sigma_f^2+2\sigma_h^2)}{KS}+\left(1-\tfrac{S}{N}\right)\tfrac{4\tilde \gamma^2}{S}G^2.
\end{align*}
where $\varepsilon_r$ is the drift caused by the local updates on the clients defined to be $\varepsilon_r:= \tfrac{1}{KS}\textstyle{\sum_{k=1}^K\sum_{i=1}^N} \mathbb{E}\left[\|y_{i,k}^r-\RRme{\bar{x}}^{r-1}\|^2\right]$.
\end{lemma}
\begin{lemma}[{\cite[Lemma 8]{karimireddy2020scaffold}}]\em\label{lemma: bound for varepsilon-k}Suppose Assumptions~\ref{assump:URS} and~\ref{assumption:main3} hold. Let $f_i$ be $\mu_f$-strongly convex for all $i$ and $h_i$ be convex for all $i$. For all $\gamma_l\le\tfrac{1}{27KB^2(L_h+\eta L_f)}$ we have
\vspace{-6pt}
\begin{align*}
3\tilde \gamma(L_h+\eta L_f)\varepsilon_r&\le \tfrac{2\tilde \gamma}{3}\mathbb{E}[f_\eta(\RRme{\bar{x}}^{r-1})]-f_\eta(x_\eta^*)+\tfrac{9\tilde \gamma^3(L_h+\eta L_f)(2\eta^2\sigma_f^2+2\sigma_h^2)}{K\gamma_g^2}+\tfrac{18(L_h+\eta L_f)\tilde \gamma^3G^2}{{\gamma_g^2}}.
\end{align*}
\end{lemma}
\vspace{-10pt}
In the following proposition, we derive an upper bound for the regularized function $f_\eta(x)$.

\begin{proposition}\em\label{Proposition:Proposition1} Let $\bar x^R=\sum_{r=1}^{R}\tfrac{\theta^{r}}{\sum_{\hat r=1}^{R-1}\theta^{\hat r}} \RRme{\bar{x}}^{r-1}$ be generated by \RRRme{StR-FedAvg} in Algorithm~\ref{Alg:FEDAVG}, where $\theta:=\tfrac{1}{1-0.5{\eta \mu_f \tilde \gamma}{}}$. Let Assumptions~\ref{assump:URS} and~\ref{assumption:main3} hold \RRme{with} $\mu_f>0$. Suppose $\gamma_l\le\min\{\tfrac{1}{27KB^2(L_h+\eta L_f)},\tfrac{1}{16K \gamma_g(L_h+\eta L_f)^2(1+B^2)}\}$. Then, we have
\vspace{-10pt}
\begin{align*}
\mathbb{E}\left[f_\eta(\bar x^R)-f_\eta(x_\eta^*)\right]\le&\tfrac{1}{ \tilde \gamma\sum_{r=1}^{R}\theta^{r}}\mathbb{E}\left[\|\RRme{\bar{x}}^{0}-x_\eta^*\|^2\right] +{\tilde \gamma}\left(\tfrac{(2\eta^2\sigma_f^2+2\sigma_h^2)}{KS}+\left(1-\tfrac{S}{N}\right)\tfrac{4}{S}G^2\right)\\
&+{\tilde \gamma}^2\left(\tfrac{9(L_h+\eta L_f)(2\eta^2\sigma_f^2+2\sigma_h^2)}{K\gamma_g^2}+\tfrac{18(L_h+\eta L_f)G^2}{{\gamma_g^2}}\right).
\end{align*}
\end{proposition}
\Rme{In Proposition~\ref{Proposition:Proposition1}, we derived a bound for the regularized function $f_\eta(x)$. However, problem~\eqref{problem: main problem} is a bilevel optimization problem. In the following proposition, we connect the bound on the regularized function to bounds on the outer- and inner-level objective functions in problem~\eqref{problem: main problem}. This result plays a key role in establishing our communication complexity guarantees.}

\Rme{ \begin{proposition}\em\label{thm:URS}
Let $\texttt{Err}_{\eta}$ be the upper bound we derived in Proposition~\ref{Proposition:Proposition1} on the optimality error metric $\mathbb{E}\left[f_\eta(\bar x^R)-f_\eta(x_\eta^*)\right]$, where $f_\eta(x) := h(x)+\eta f(x)$ and $f_\eta(x_\eta^*):= \inf_{x \in \mathbb{R}^n} f_\eta(x)$. Let Assumption~\ref{assump:URS} hold \RRme{with $\mu_f>0$.} Let $x^*$ denote the unique optimal solution to the bilevel problem $\min_{x\in X^*_h} \ f(x)$. Let $f^* := \inf_{x\in X^*_h} \ f(x)$ denote the optimal objective value of the bilevel problem and $h^* := \inf_{x \in \mathbb{R}^n} h(x)$. 

\noindent [Case i] The following error bounds hold true {for some $M>0$}.
~\\
 (i-1) \ $0\le \mathbb{E}[ h(\RRme{\bar x^R})]-h^*\le  \texttt{Err}_{\eta}+\eta M$.
~\\
 (i-2) \ $({\mu_f}/{2})\mathbb{E}[\|\RRme{\bar x^R}-x^*\|^2] -\|\nabla f(x^*)\| \,\mathbb{E}[\mbox{dist}(\RRme{\bar x^R},X^*_h)] \leq \mathbb{E}[f(\RRme{\bar x^R})] -f^* \leq {\eta^{-1}}\texttt{Err}_{\eta}$.
~\\
 (i-3) \ $ \mathbb{E}[\|\RRme{\bar x^R}-x^*\|^2] \leq ({2}/{\mu_f})\left(\|\nabla f(x^*)\| \,\mathbb{E}[\mbox{dist}(\RRme{\bar x^R},X^*_h)] + {\eta^{-1}}\texttt{Err}_{\eta}\right)$.

\noindent [Case ii] Suppose $X^*_h$ is $\alpha$-weak sharp with order $\kappa > 1$. Then, {for some $M>0$}:
~\\
 (ii-1) \ $0\le \mathbb{E}[h(\RRme{\bar x^R})]-h^*\le  \texttt{Err}_{\eta}+\eta M$.
~\\
 (ii-2) \ $ \tfrac{\mu_f}{2}\mathbb{E}[\|\RRme{\bar x^R}-x^*\|^2 ]-\|\nabla f(x^*)\| \,\sqrt[\kappa]{\tfrac{1}{\alpha}\left(\texttt{Err}_{\eta}+\eta M\right)} \leq \mathbb{E}[f(\RRme{\bar x^R})] -f^* \leq \tfrac{1}{\eta}\texttt{Err}_{\eta}$.
~\\(ii-3)  \ \ $ \mathbb{E}[\|\RRme{\bar x^R}-x^*\|^2]\leq \tfrac{2}{\mu_f}\left(\|\nabla f(x^*)\| \,\sqrt[\kappa]{\tfrac{1}{\alpha}\left(\texttt{Err}_{\eta}+\eta M\right)} +  {\eta^{-1}}\texttt{Err}_{\eta}\right)$.

\noindent [Case iii] If $X^*_h$ is $\alpha$-weak sharp with order $\kappa=1$ and  $\eta\le{{\alpha}/{2\|\nabla f(x^*)\|}}$, then:
~\\ (iii-1) \ $0\le \mathbb{E}[h(\RRme{\bar x^R})]-h^*\le 2\, \texttt{Err}_{\eta} .$
~\\ (iii-2) \ $ {0.5\mu_f}{}\mathbb{E}[\|\RRme{\bar x^R}-x^*\|^2] -({2\|\nabla f(x^*)\|}/{\alpha}) \,\texttt{Err}_{\eta}\leq \mathbb{E}[f(\RRme{\bar x^R})] -f^* \leq {\eta^{-1}}\texttt{Err}_{\eta}$.
~\\ (iii-3)   \  $\mathbb{E}[\|\RRme{\bar x^R}-x^*\|^2] \leq ({2}/{\eta \mu_f})\texttt{Err}_{\eta}. $
\end{proposition}}
\RRme{{\begin{proof}

\noindent (i-1)   {In view of} the definition of $\texttt{Err}_{\eta}$ and $f_\eta$, and setting $f^*_\eta\triangleq f_\eta(x_\eta^*)$, we have $\mathbb{E}[h(\RRme{\bar x^R})  + \eta f(\RRme{\bar x^R})]  -f^*_\eta \leq \texttt{Err}_{\eta}.$ Also, from the definition of $f^*_\eta$, we have $f^*_\eta \leq f_\eta(x^*) = h^* +\eta f^*$. From the preceding two relations, we obtain 
\begin{align}\label{eqn:URS-i-1}
\mathbb{E}[h(\RRme{\bar x^R})]  -h^* +\eta ( \mathbb{E}[ f(\RRme{\bar x^R})]-f^*) \leq \texttt{Err}_{\eta}.
\end{align}
By invoking Assumption~\ref{assump:URS}, we have $ \mathbb{E}[ f(\RRme{\bar x^R})]>-\infty$. Therefore, there exists some $M>0$ such that $f^*- \mathbb{E}[ f(\RRme{\bar x^R})]<M$. As a result, we obtain $\mathbb{E}[h(\RRme{\bar x^R})]-h^*\leq  \texttt{Err}_{\eta}+\eta M$. From $\RRme{\bar x^R} \in X$, we also have $\mathbb{E}{[h(\RRme{\bar x^R})]}-h^* \geq 0$. This completes the proof of (i-1). 

\noindent (i-2) The upper bound holds in view of \eqref{eqn:URS-i-1} and that $\mathbb{E}[h(\RRme{\bar x^R})]  -h^*\geq 0$. To show the lower bound, from {the} \RRme{$\mu_f$-strong} convexity of $f$, we may write 
$$ \tfrac{\mu_f}{2}\mathbb{E}[\|\RRme{\bar x^R}-x^*\|^2]+\nabla f(x^*)^\top\mathbb{E}[(\RRme{\bar x^R}-x^*)]\leq \mathbb{E}[f(\RRme{\bar x^R})] - f^* .$$
Note that $\nabla f(x^*)^\top\mathbb{E}[(\RRme{\bar x^R}-x^*)]$ is not necessarily nonnegative.  We have
$$ \tfrac{\mu_f}{2}\mathbb{E}[\|\RRme{\bar x^R}-x^*\|^2]+\nabla f(x^*)^\top\mathbb{E}[(\RRme{\bar x^R}- \Pi_{X^*_h}[\RRme{\bar x^R}]+\Pi_{X^*_h}[\RRme{\bar x^R}]-x^*)]\leq \mathbb{E}[f(\RRme{\bar x^R})] - f^* .$$
Since $\Pi_{X^*_h}[\RRme{\bar x^R}] \in X^*_h$, we have $\nabla f(x^*)^\top\mathbb{E}\left[\left(\Pi_{X^*_h}[\RRme{\bar x^R}]-x^*\right)\right]\geq 0$. Therefore, we obtain 
$$ \tfrac{\mu_f}{2}\mathbb{E}[\|\RRme{\bar x^R}-x^*\|^2]+\nabla f(x^*)^\top\mathbb{E}\left[\left(\RRme{\bar x^R}- \Pi_{X^*_h}[\RRme{\bar x^R}]\right)\right]\leq \mathbb{E}[f(\RRme{\bar x^R})] - f^* .$$
Invoking the Cauchy-Schwarz {inequality}, we obtain 
\vspace{-3pt}
$$ \tfrac{\mu_f}{2}\mathbb{E}[\|\RRme{\bar x^R}-x^*\|^2]-\|\nabla f(x^*)\|\,{\left\|\mathbb{E}\left[\RRme{\bar x^R}- \Pi_{X^*_h}[{\RRme{\bar x^R}}]\right]\right\|}\leq \mathbb{E}[f(\RRme{\bar x^R}) ]- f^* .$$
{Invoking Jensen's inequality, we obtain}
\[{ \tfrac{\mu_f}{2}\mathbb{E}[\|\RRme{\bar x^R}-x^*\|^2]-\|\nabla f(x^*)\|\,\mathbb{E}\left[\left\|\RRme{\bar x^R}- \Pi_{X^*_h}[{\RRme{\bar x^R}}]\right\|\right]\leq \mathbb{E}[f(\RRme{\bar x^R}) ]- f^* .}\]
\noindent Noting that $\mathbb{E}[\mbox{dist}(\RRme{\bar x^R},X^*_h)]=\mathbb{E}\left[\left\|\RRme{\bar x^R}- \Pi_{X^*_h}[{\RRme{\bar x^R}}]\right\|\right]$, we obtain the lower bound in (i-2). 

\noindent (i-3) This result follows directly from (i-2). 

\noindent (ii-1) This result is identical to (i-1). 

\noindent (ii-2) From the weak sharp property, we have {$ \mbox{dist}^\kappa(\RRme{\bar x^R},X^*_h)  \leq \tfrac{1}{\alpha}( h(\RRme{\bar x^R}) - h^*)$. Taking expectations on both sides, we obtain 
$\mathbb{E}\left[\mbox{dist}^\kappa(\RRme{\bar x^R},X^*_h)\right]  \leq \tfrac{1}{\alpha}( \mathbb{E}[h(\RRme{\bar x^R})] - h^*).$ Moreover, by Jensen's inequality and the convexity of the polynomial function $x^\kappa$ for $x \geq 0$, we have $\left(\mathbb{E}\left[\mbox{dist}(\RRme{\bar x^R},X^*_h)\right]\right)^\kappa \leq \mathbb{E}\left[\mbox{dist}^\kappa(\RRme{\bar x^R},X^*_h)\right]$. Combining the preceding two inequalities, we obtain} $ {\mathbb{E}\left[\mbox{dist}(\RRme{\bar x^R},X^*_h)\right]  \leq \sqrt[\kappa]{\tfrac{1}{\alpha}( \mathbb{E}[h(\RRme{\bar x^R})] - h^*)}.}$ Combining this relation with (ii-1), we have $\mathbb{E}[\mbox{dist}(\RRme{\bar x^R},X^*_h)] \leq \sqrt[\kappa]{\tfrac{1}{\alpha}(\texttt{Err}_{\eta}+\eta M)}$. Substituting this bound into the result of (i-2), we arrive at (ii-2). 

\noindent (ii-3) This result follows directly from (ii-2). 

\noindent (iii-1) Consider the relation in (i-2). From the weak sharp property, we obtain 
\begin{align}\label{eqn:URS-i-2} \tfrac{\mu_f}{2}\mathbb{E}[\|\RRme{\bar x^R}-x^*\|^2] -\tfrac{\|\nabla f(x^*)\|}{\alpha}(\mathbb{E}[h(\RRme{\bar x^R})]-h^*) \leq\mathbb{E}[ f(\RRme{\bar x^R})]-f^*.
\end{align}
Multiplying both sides by $\eta$ and substituting the resulting lower bound into \eqref{eqn:URS-i-1}, we obtain 
$$\mathbb{E}[h(\RRme{\bar x^R})]  -h^* + \tfrac{\eta\mu_f}{2}\mathbb{E}[\|\RRme{\bar x^R}-x^*\|^2] -\tfrac{\eta\|\nabla f(x^*)\|}{\alpha}(\mathbb{E}[h(\RRme{\bar x^R})]-h^*) \leq \texttt{Err}_{\eta}.$$
By the choice of $\eta\le\tfrac{\alpha}{2\|\nabla f(x^*)\|}$, we obtain 
\begin{align}\label{eqn:URS-i-3}
\tfrac{1}{2}(\mathbb{E}[h(\RRme{\bar x^R}) ] -h^*) + \tfrac{\eta\mu_f}{2}\mathbb{E}[\|\RRme{\bar x^R}-x^*\|^2 ] \leq \texttt{Err}_{\eta}. 
\end{align}
Dropping the nonnegative term $\tfrac{\eta\mu_f}{2}\mathbb{E}[\|\RRme{\bar x^R}-x^*\|^2]$, we arrive at (iii-1).

\noindent (iii-2) The upper bound follows from (i-2). To show the lower bound, consider \eqref{eqn:URS-i-2}. Invoking {the result in} (iii-1), we obtain 
$\tfrac{\mu_f}{2}\mathbb{E}[\|\RRme{\bar x^R}-x^*\|^2] -\tfrac{\|\nabla f(x^*)\|}{\alpha}(2\texttt{Err}_{\eta}) \leq \mathbb{E}[f(\RRme{\bar x^R})] -f^*.
$
This completes the proof of (iii-2). 

\noindent (iii-3) Consider \eqref{eqn:URS-i-3}. Note that $\mathbb{E}[h(\RRme{\bar x^R}) ] -h^* \geq 0$. Dropping this nonnegative term and rearranging the terms, we arrive at (iii-3). 
\end{proof}}}
\Rme{In the following proposition, we establish bounds for the outer- and inner-level objective functions in problem~\eqref{problem: main problem}. These bounds are instrumental in proving the main communication complexity guarantees stated in Theorem~\ref{thm:thm2}.}

\begin{proposition}\em\label{Prop:Prop2}Let Assumptions~\ref{assump:URS} and~\ref{assumption:main3} hold \RRme{with $\mu_f>0$}. Then, for all $\gamma_l \le \tfrac{1}{16(L_h+\eta L_f)^2(1+B^2)K\gamma_g}$, $\gamma_g\ge1$ and $R\ge \tfrac{16(L_h+\eta L_f)^2(1+B^2)}{\mu_f}$ we obtain

\noindent [Case i] The following error bounds hold true. 
\vspace{-10pt}
\begin{align*}
\text{ (i-1) }0\le \mathbb{E}\left[h(\bar x^R)-h(x^*)\right]\le&\tfrac{1}{ \tilde \gamma\sum_{r=1}^{R}\theta^{r}}\mathbb{E}\left[\|\RRme{\bar x}^{0}-x_\eta^*\|^2\right]+{\tilde \gamma}\left(\tfrac{(2\eta^2\sigma_f^2+2\sigma_h^2)}{KS}+\tfrac{4G^2(N-S)}{NS}\right)\\
&+{\tilde \gamma^2}\left(\tfrac{9(L_h+\eta L_f)(2\eta^2\sigma_f^2+2\sigma_h^2)}{K\gamma_g^2}+\tfrac{18(L_h+\eta L_f)G^2}{{\gamma_g^2}}\right)+2\eta M.
\end{align*}
\vspace{-25pt}
\begin{align*}
\text{ (i-2) }&\tfrac{\mu_f}{2}\mathbb{E}\left[\|\bar x^R-x^*\|^2\right] -\|\nabla f(x^*)\| \,\mathbb{E}\left[\mbox{dist}(\bar x^R,X^*_h)\right]\\\le& \mathbb{E}\left[f(\bar x^R)-f(x^*)\right]\le\tfrac{1}{ \eta\tilde \gamma\sum_{r=1}^{R}\theta^{r}}\mathbb{E}\left[\|\RRme{\bar x}^{0}-x_\eta^*\|^2\right]+\tfrac{\tilde \gamma}{\eta}\left(\tfrac{(2\eta^2\sigma_f^2+2\sigma_h^2)}{KS}+\left(1-\tfrac{S}{N}\right)\tfrac{4}{S}G^2\right)\\
&+\tfrac{\tilde \gamma^2}{\eta}\left(\tfrac{9(L_h+\eta L_f)(2\eta^2\sigma_f^2+2\sigma_h^2)}{K\gamma_g^2}+\tfrac{18(L_h+\eta L_f)G^2}{{\gamma_g^2}}\right).
\end{align*}
\vspace{-25pt}
\begin{align*} \text{ (i-3) }&\mathbb{E}\left[\|\bar{x}^R-x^*\|^2\right] \leq  \tfrac{2}{\mu_f}\left(\|\nabla f(x^*)\| \,\mathbb{E}\left[\mbox{dist}(\bar{x}^R,X^*_h)\right] + \tfrac{1}{ \eta\tilde \gamma\sum_{r=1}^{R}\theta^{r}}\mathbb{E}\left[\|\RRme{\bar x}^{0}-x_\eta^*\|^2\right]\right.\\&\left.+\tfrac{\tilde \gamma}{\eta}\left(\tfrac{(2\eta^2\sigma_f^2+2\sigma_h^2)}{KS}+\left(1-\tfrac{S}{N}\right)\tfrac{4}{S}G^2\right)+\tfrac{\tilde \gamma^2}{\eta}\left(\tfrac{9(L_h+\eta L_f)(2\eta^2\sigma_f^2+2\sigma_h^2)}{K\gamma_g^2}+\tfrac{18(L_h+\eta L_f)G^2}{{\gamma_g^2}}\right)\right).
\end{align*}

\noindent [Case ii] Suppose $X^*_h$ is $\alpha$-weak sharp with order $\kappa > 1$. Then, the following holds. 
\vspace{-10pt}
\begin{align*}
\text{(ii-1) }0\le \mathbb{E}\left[h(\bar x^R)-h(x^*)\right]\le&\tfrac{1}{ \tilde \gamma\sum_{r=1}^{R}\theta^{r}}\mathbb{E}\left[\|\RRme{\bar x}^{0}-x_\eta^*\|^2\right]+{\tilde \gamma}\left(\tfrac{(2\eta^2\sigma_f^2+2\sigma_h^2)}{KS}+\tfrac{4G^2(N-S)}{NS}\right)\\
&+{\tilde \gamma^2}\left(\tfrac{9(L_h+\eta L_f)(2\eta^2\sigma_f^2+2\sigma_h^2)}{K\gamma_g^2}+\tfrac{18(L_h+\eta L_f)G^2}{{\gamma_g^2}}\right)+2\eta M.
\end{align*}
\vspace{-25pt}
\begin{align*}
\text{(ii-2) }&-\|\nabla f(x^*)\|\left(\tfrac{1}{\alpha \tilde \gamma\sum_{r=1}^{R}\theta^{r}}\mathbb{E}\left[\|\RRme{\bar x}^{0}-x_\eta^*\|^2\right]+\tfrac{\tilde \gamma}{\alpha}\left(\tfrac{(2\eta^2\sigma_f^2+2\sigma_h^2)}{KS}+\tfrac{4G^2(N-S)}{NS}\right)\right.\\
&\left.+\tfrac{\tilde \gamma^2}{\alpha}\left(\tfrac{9(L_h+\eta L_f)(2\eta^2\sigma_f^2+2\sigma_h^2)}{K\gamma_g^2}+\tfrac{18(L_h+\eta L_f)G^2}{{\gamma_g^2}}\right)+\tfrac{2\eta M}{\alpha}\right)^\frac{1}{\kappa}\\&\le\mathbb{E}\left[f(\bar x^R)-f(x^*)\right]\le\tfrac{1}{ \eta\tilde \gamma\sum_{r=1}^{R}\theta^{r}}\mathbb{E}\left[\|\RRme{\bar x}^{0}-x_\eta^*\|^2\right]+\tfrac{\tilde \gamma}{\eta}\left(\tfrac{(2\eta^2\sigma_f^2+2\sigma_h^2)}{KS}+\tfrac{4G^2(N-S)}{NS}\right)\\
&+\tfrac{\tilde \gamma^2}{\eta}\left(\tfrac{9(L_h+\eta L_f)(2\eta^2\sigma_f^2+2\sigma_h^2)}{K\gamma_g^2}+\tfrac{18(L_h+\eta L_f)G^2}{{\gamma_g^2}}\right).
\end{align*}
\vspace{-25pt}
\begin{align*}
\text{(ii-3) }&\mathbb{E}\left[\|\bar{x}^R-x^*\|^2\right]\leq \tfrac{2}{\mu_f \eta\tilde \gamma\sum_{r=1}^{R}\theta^{r}}\mathbb{E}\left[\|\RRme{\bar x}^{0}-x_\eta^*\|^2\right]+\tfrac{2\tilde \gamma}{\mu_f\eta}\left(\tfrac{(2\eta^2\sigma_f^2+2\sigma_h^2)}{KS}+\tfrac{4G^2(N-S)}{NS}\right)\\
&+\tfrac{2\tilde \gamma^2}{\mu_f\eta}\left(\tfrac{9(L_h+\eta L_f)(2\eta^2\sigma_f^2+2\sigma_h^2)}{K\gamma_g^2}+\tfrac{18(L_h+\eta L_f)G^2}{{\gamma_g^2}}\right)\\
&+\tfrac{2\|\nabla f(x^*)\|}{\mu_f}\left(\tfrac{1}{\alpha}\Big( \tfrac{1}{ \tilde \gamma\sum_{r=1}^{R}\theta^{r}}\mathbb{E}\left[\|\RRme{\bar x}^{0}-x_\eta^*\|^2\right]\right.+{\tilde \gamma}\left(\tfrac{(2\eta^2\sigma_f^2+2\sigma_h^2)}{KS}+\tfrac{(N-S)4G^2}{NS}\right)\\
&\left.+{\tilde \gamma^2}\left(\tfrac{9(L_h+\eta L_f)(2\eta^2\sigma_f^2+2\sigma_h^2)+18K(L_h+\eta L_f)G^2}{K\gamma_g^2}\right)+2\eta M\Big)\right)^\frac{1}{\kappa}. 
\end{align*}
\noindent [Case iii] If $X^*_h$ is $\alpha$-weak sharp with order $\kappa=1$ and  $\eta\le\tfrac{\alpha}{2\|\nabla f(x^*)\|}$, then:
\vspace{-10pt}
\begin{align*}
\text{(iii-1) }0\le\mathbb{E}\left[h(\bar x^R)-h(x^*)\right]\le&\tfrac{2}{ \tilde \gamma\sum_{r=1}^{R}\theta^{r}}\mathbb{E}\left[\|\RRme{\bar x}^{0}-x_\eta^*\|^2\right]+{2\tilde \gamma}\left(\tfrac{(2\eta^2\sigma_f^2+2\sigma_h^2)}{KS}+\tfrac{4G^2(N-S)}{NS}\right)\\
&+{2\tilde \gamma^2}\left(\tfrac{9(L_h+\eta L_f)(2\eta^2\sigma_f^2+2\sigma_h^2)}{K\gamma_g^2}+\tfrac{18(L_h+\eta L_f)G^2}{{\gamma_g^2}}\right).
\end{align*}
\vspace{-25pt}
\begin{align*}
\text{(iii-2) }&-\tfrac{2\|\nabla f(x^*)\|}{\alpha \tilde \gamma\sum_{r=1}^{R}\theta^{r}}\mathbb{E}\left[\|\RRme{\bar x}^{0}-x_\eta^*\|^2\right]-\tfrac{2\tilde \gamma\|\nabla f(x^*)\|}{\alpha}{}\left(\tfrac{(2\eta^2\sigma_f^2+2\sigma_h^2)}{KS}+\left(1-\tfrac{S}{N}\right)\tfrac{4}{S}G^2\right)\\
&-\tfrac{2\tilde \gamma^2\|\nabla f(x^*)\|}{\alpha}{}\left(\tfrac{9(L_h+\eta L_f)(2\eta^2\sigma_f^2+2\sigma_h^2)}{K\gamma_g^2}+\tfrac{18(L_h+\eta L_f)G^2}{{\gamma_g^2}}\right)\\&\le\mathbb{E}\left[f(\bar x^R)-f(x^*)\right]\le\tfrac{1}{ \eta\tilde \gamma\sum_{r=1}^{R}\theta^{r}}\mathbb{E}\left[\|\RRme{\bar x}^{0}-x_\eta^*\|^2\right]+\tfrac{\tilde \gamma}{\eta}\left(\tfrac{(2\eta^2\sigma_f^2+2\sigma_h^2)}{KS}+\tfrac{4G^2(N-S)}{NS}\right)\\
&+\tfrac{\tilde \gamma^2}{\eta}\left(\tfrac{9(L_h+\eta L_f)(2\eta^2\sigma_f^2+2\sigma_h^2)}{K\gamma_g^2}+\tfrac{18(L_h+\eta L_f)G^2}{{\gamma_g^2}}\right).
\end{align*} 
\vspace{-25pt}
\begin{align*}
\text{(iii-3) }\mathbb{E}\left[\|\bar{x}^R-x^*\|^2\right]\le&\tfrac{2}{\eta\mu_f \tilde \gamma\sum_{r=1}^{R}\theta^{r}}\mathbb{E}\left[\|\RRme{\bar x}^{0}-x_\eta^*\|^2\right]+\tfrac{2\tilde \gamma}{\eta\mu_f}\left(\tfrac{(2\eta^2\sigma_f^2+2\sigma_h^2)}{KS}+\tfrac{4G^2(N-S)}{NS}\right)\\
&+\tfrac{2\tilde \gamma^2}{\eta\mu_f}\left(\tfrac{9(L_h+\eta L_f)(2\eta^2\sigma_f^2+2\sigma_h^2)}{K\gamma_g^2}+\tfrac{18(L_h+\eta L_f)G^2}{{\gamma_g^2}}\right).
\end{align*}
\end{proposition}
\Rme{We now derive the bounds in Theorem~\ref{thm:thm2} using the previous results.

}
\begin{proof}[Proof of Theorem~\ref{thm:thm2}]

\noindent (i-1)
Based on the definition of $\theta$, we have $\eta \tilde \gamma={2}/{\mu_f}\left(1-{1}/{\theta}\right)$. Invoking Proposition~\ref{Prop:Prop2}, we obtain 
\begin{align*}
&\mathbb{E}\left[f(\bar x^R)-f(x^*)\right]\le\tfrac{1}{{2}/{\mu_f}\left(1-\tfrac{1}{\theta}\right)\tfrac{\theta^{R+1}-\theta}{\theta-1}}\mathbb{E}\left[\|\RRme{\bar x}^{0}-x_\eta^*\|^2\right]+\tfrac{\tilde \gamma}{\eta}\left(\tfrac{(2\eta^2\sigma_f^2+2\sigma_h^2)}{KS}+\left(1-\tfrac{S}{N}\right)\tfrac{4}{S}G^2\right)\\
&+\tfrac{\tilde \gamma^2}{\eta}\left(\tfrac{9(L_h+\eta L_f)(2\eta^2\sigma_f^2+2\sigma_h^2)}{K\gamma_g^2}+\tfrac{18(L_h+\eta L_f)G^2}{{\gamma_g^2}}\right)=\tfrac{1}{{2}/{\mu_f}\left(\theta^{R}-1\right)}\mathbb{E}\left[\|\RRme{\bar x}^{0}-x_\eta^*\|^2\right]\\
&+\tfrac{\tilde \gamma}{\eta}\left(\tfrac{(2\eta^2\sigma_f^2+2\sigma_h^2)}{KS}+\left(1-\tfrac{S}{N}\right)\tfrac{4}{S}G^2\right)+\tfrac{\tilde \gamma^2}{\eta}\left(\tfrac{9(L_h+\eta L_f)(2\eta^2\sigma_f^2+2\sigma_h^2)}{K\gamma_g^2}+\tfrac{18(L_h+\eta L_f)G^2}{{\gamma_g^2}}\right).
\end{align*}
Invoking the definition of $\theta$, we obtain
\begin{align*}
&\mathbb{E}\left[f(\bar x^R)-f(x^*)\right]\le\tfrac{\mu_f}{\tfrac{2}{(1-0.5{\eta \mu_f \tilde \gamma}{})^{R}}-2}\mathbb{E}\left[\|\RRme{\bar x}^{0}-x_\eta^*\|^2\right]\\
&+\tfrac{\tilde \gamma}{\eta}\left(\tfrac{(2\eta^2\sigma_f^2+2\sigma_h^2)}{KS}+\left(1-\tfrac{S}{N}\right)\tfrac{4}{S}G^2\right)
+\tfrac{\tilde \gamma^2}{\eta}\left(\tfrac{9(L_h+\eta L_f)(2\eta^2\sigma_f^2+2\sigma_h^2)}{K\gamma_g^2}+\tfrac{18(L_h+\eta L_f)G^2}{{\gamma_g^2}}\right).
\end{align*}
By choosing $\tilde \gamma\triangleq {1}/{\mu_f^aR^a}$, $\eta\triangleq {p\RRme{\ln}(R)}/{\mu_f^bR^b}$, \RRme{and $\hat\eta\triangleq \max_R\eta=\frac{p}{\exp(1) b \mu_f^b}$}, we obtain
\begin{align}
&\mathbb{E}\left[f(\bar x^R)-f(x^*)\right]\le\tfrac{\mu_f}{\tfrac{2}{(1-{p \RRme{\ln}(R)}/{2R})^{R}}-2}\mathbb{E}\left[\|\RRme{\bar x}^{0}-x_\eta^*\|^2\right]+\tfrac{1}{p\RRme{\ln}(R)\mu_f^{a-b}R^{a-b}}\notag\\
&\times\left(\left(\tfrac{(2\RRme{\hat\eta}^2\sigma_f^2+2\sigma_h^2)}{KS}+\tfrac{(N-S)4G^2}{NS}\right)+\left(\tfrac{9(L_h+\RRme{\hat\eta} L_f)(2\RRme{\hat\eta}^2\sigma_f^2+2\sigma_h^2)}{\mu_f^aR^aK\gamma_g^2}+\tfrac{18(L_h+\RRme{\hat\eta} L_f)G^2}{{\mu_f^aR^a\gamma_g^2}}\right)\right).\label{eqqq:coeficient}
\end{align}
For the coefficient of the first term, we have
\begin{align*}
&\tfrac{\mu_f}{\tfrac{2}{(1-{p \RRme{\ln}(R)}/{2R})^{R}}-2}=\tfrac{\mu_f}{\tfrac{2}{\left((1-{p \RRme{\ln}(R)}/{2R})^{{2R}/{p\RRme{\ln}(R)}}\right)^{{p\RRme{\ln}(R)}/{2}}}-2}\\
&\le\tfrac{\mu_f}{\tfrac{2}{\lim_{R\rightarrow \infty}\left((1-{p \RRme{\ln}(R)}/{2R})^{{2R}/{p\RRme{\ln}(R)}}\right)^\frac{p\RRme{\ln}(R)}{2}}-2}=\tfrac{\mu_f}{\tfrac{2}{\left(\exp(-1)\right)^\frac{p\RRme{\ln}(R)}{2}}-2}=\tfrac{\mu_f}{2\left(R^\frac{p}{2}-1\right)}.
\end{align*}
Substituting the preceding bound in~\eqref{eqqq:coeficient}, we obtain
\begin{align*}
&\mathbb{E}\left[f(\bar x^R)-f(x^*)\right]\le\tfrac{\mu_f}{2\left(R^{{p}/{2}}-1\right)}\mathbb{E}\left[\|\RRme{\bar x}^{0}-x_\eta^*\|^2\right]+\tfrac{1}{p\RRme{\ln}(R)\mu_f^{a-b}R^{a-b}}\\
&\times\left(\left(\tfrac{(2\RRme{\hat\eta}^2\sigma_f^2+2\sigma_h^2)}{KS}+\tfrac{(N-S)4G^2}{NS}\right)+\left(\tfrac{9(L_h+\RRme{\hat\eta} L_f)(2\RRme{\hat\eta}^2\sigma_f^2+2\sigma_h^2)}{\mu_f^aR^aK\gamma_g^2}+\tfrac{18(L_h+\RRme{\hat\eta} L_f)G^2}{{\mu_f^aR^a\gamma_g^2}}\right)\right).
\end{align*}
Based on the definitions of $Q$, $W$, $Y$, and $E$, we obtain the final result.

\noindent (i-2) Based on the previous part and Proposition~\ref{Prop:Prop2}, we obtain
\begin{align*}
\mathbb{E}\left[h(\bar x^R)-h(x^*)\right]\le&\tfrac{p\RRme{\ln}(R)\mu_f^{1-b}}{2R^b\left(R^{p/2}-1\right)}\mathbb{E}\left[\|\RRme{\bar x}^{0}-x_\eta^*\|^2\right]+\tfrac{1}{\mu_f^{a}R^{a}}\left(\tfrac{(2\RRme{\hat\eta}^2\sigma_f^2+2\sigma_h^2)}{KS}+\left(1-\tfrac{S}{N}\right)\tfrac{4}{S}G^2\right)\\
&+\tfrac{1}{\mu_f^{2a}R^{2a}}\left(\tfrac{9(L_h+\RRme{\hat\eta} L_f)(2\RRme{\hat\eta}^2\sigma_f^2+2\sigma_h^2)}{K\gamma_g^2}+\tfrac{18(L_h+\RRme{\hat\eta} L_f)G^2}{{\gamma_g^2}}\right)+\tfrac{p\RRme{\ln}(R)}{\mu_f^{b}R^{b}}2 M.
\end{align*}
\far{From} definitions of $Q$, $W$, $Y$, and $E$, the final result is obtained. Using Proposition~\ref{Prop:Prop2}, we \far{obtain the} results in parts \noindent (ii) and \noindent (iii).
\end{proof}  

\Rme{Following this, we employ a similar approach to derive the communication complexity bounds for StR-FedAvg when the {outer loss} is convex.}
\begin{lemma}\em\label{lemma:first bound for x-x*-f convex}Let $\RRme{\bar x}^r$ be generated by Algorithm~\ref{Alg:FEDAVG}. \RRme{Let Assumptions~\ref{assump:URS} and~\ref{assumption:main3} hold with $\mu_f=0$. Then,} \far{for} $\gamma_l\le\tfrac{1}{16K \gamma_g(L_h+\eta L_f)^2(1+ B^2)}$ and $\tilde \gamma \triangleq \gamma_l \gamma_g K$, 
\begin{align*}
\mathbb{E}\left[\|\RRme{\bar x}^r-x_\eta^*\|^2\right]\le&\mathbb{E}\left[\|\RRme{\bar x}^{r-1}-x_\eta^*\|^2\right]-\tilde \gamma\mathbb{E}\left[(f_\eta(\RRme{\bar x}^{r-1})-f_\eta(x_\eta^*))\right]\\
&+3\tilde \gamma(L_h+\eta L_f)\varepsilon_r+\tfrac{\tilde \gamma^2(2\eta^2\sigma_f^2+2\sigma_h^2)}{KS}+\left(1-\tfrac{S}{N}\right)\tfrac{4\tilde \gamma^2}{S}G^2.
\end{align*}
where $\varepsilon_r$ is the drift caused by the local updates on the clients defined to be
$\varepsilon_r\triangleq \tfrac{1}{KS}\textstyle{\sum_{k=1}^K\sum_{i=1}^N} \mathbb{E}\left[\|y_{i,k}^r-\RRme{\bar x}^{r-1}\|^2|\mathcal{F}_r\right].$
\end{lemma}
\Rme{In the following proposition, we derive an upper bound for the regularized function $f_\eta(x)$.}

\begin{proposition}\em\label{Proposition:Proposition3} Let $\bar x^R=\sum_{r=1}^{R}\tfrac{\theta^{r}}{\sum_{\hat r=1}^{R-1}\theta^{\hat r}} \RRme{\bar x}^{r-1}$ be generated by \RRRme{StR-FedAvg} in Algorithm~\ref{Alg:FEDAVG}, where $\theta:=\tfrac{1}{1-0.5{\eta \mu_f \tilde \gamma}}$. Let Assumptions~\ref{assump:URS} and~\ref{assumption:main3} hold \RRme{with} $\mu_f=0$. Then for all $\gamma_l\le\min\{\tfrac{1}{27K B^2(L_h+\eta L_f)},\tfrac{1}{16K \gamma_g(L_h+\eta L_f)^2(1+ B^2)}\}$ we have
\begin{align*}
& \mathbb{E}\left[(f_\eta(\bar{x}^{R})-f_\eta(x_\eta^*))\right]\le\tfrac{3}{\tilde \gamma R}\left(\mathbb{E}\left[\|\RRme{\bar x}^{0}-x_\eta^*\|^2\right]-\mathbb{E}\left[\|\RRme{\bar x}^R-x_\eta^*\|^2\right]\right)\\
&+3\tilde \gamma\left(\tfrac{(2\eta^2\sigma_f^2+2\sigma_h^2)}{KS}+\left(1-\tfrac{S}{N}\right)\tfrac{4}{S}G^2\right)+3\tilde \gamma^2\left(\tfrac{9(L_h+\eta L_f)(2\eta^2\sigma_f^2+2\sigma_h^2)}{K\gamma_g^2}+\tfrac{18(L_h+\eta L_f)G^2}{{\gamma_g^2}}\right).
\end{align*}

\end{proposition}

\Rme{In Proposition~\ref{Proposition:Proposition3}, we derived a bound for the regularized function $f_\eta(x)$. Analogous to the strongly convex case, the following proposition establishes a connection between the bound on the regularized function and bounds on the outer- and inner-level objective functions in problem~\eqref{problem: main problem}. This result plays a crucial role in deriving our communication complexity guarantees.

\begin{proposition}\em\label{thm:URS-convex}Let $\texttt{Err}_{\eta}$ be the upper bound we derived in Proposition~\ref{Proposition:Proposition3} on the optimality error metric $\mathbb{E}\left[f_\eta(\bar x^R)-f_\eta(x_\eta^*)\right]$, where $f_\eta(x) := h(x)+\eta f(x)$ and $f_\eta(x_\eta^*):= \inf_{x \in \mathbb{R}^n} f_\eta(x)$. Let Assumption~\ref{assump:URS} hold \RRme{with} $\mu_f=0$. Let $x^*$ denote an arbitrary optimal solution to the bilevel problem~\eqref{problem: main problem}. Let $f^* := \inf_{x\in X^*_h} \ f(x)$ denote the optimal objective value of the bilevel problem and $h^* := \inf_{x \in \mathbb{R}^n} h(x)$. 

\noindent [Case i] The following error bounds hold true {for some $M>0$}.
~\\
 (i-1) \ $0\le \mathbb{E}[ h(\RRme{\bar x^R})]-h^*\le  \texttt{Err}_{\eta}+\eta M$.
~\\
 (i-2) \ $ -\|\nabla f(x^*)\| \,\mathbb{E}[\mbox{dist}(\RRme{\bar x^R},X^*_h)] \leq \mathbb{E}[f(\RRme{\bar x^R})] -f^* \leq {\eta^{-1}}\texttt{Err}_{\eta}$.

\noindent [Case ii] Suppose $X^*_h$ is $\alpha$-weak sharp with order $\kappa > 1$. Then, {for some $M>0$}:
~\\
 (ii-1) \ $0\le \mathbb{E}[h(\RRme{\bar x^R})]-h^*\le  \texttt{Err}_{\eta}+\eta M$.
~\\
 (ii-2) \ $ -\|\nabla f(x^*)\| \,\sqrt[\kappa]{\tfrac{1}{\alpha}\left(\texttt{Err}_{\eta}+\eta M\right)} \leq \mathbb{E}[f(\RRme{\bar x^R})] -f^* \leq \tfrac{1}{\eta}\texttt{Err}_{\eta}$.

\noindent [Case iii] If $X^*_h$ is $\alpha$-weak sharp with order $\kappa=1$ and  $\eta\le{{\alpha}/{2\|\nabla f(x^*)\|}}$, then:
~\\ (iii-1) \ $0\le \mathbb{E}[h(\RRme{\bar x^R})]-h^*\le 2\, \texttt{Err}_{\eta} .$
~\\ (iii-2) \ $  -({2\|\nabla f(x^*)\|}/{\alpha}) \,\texttt{Err}_{\eta}\leq \mathbb{E}[f(\RRme{\bar x^R})] -f^* \leq {\eta^{-1}}\texttt{Err}_{\eta}$.

\end{proposition}
\RRme{\begin{proof}

\noindent (i-1) The proof is similar to that of part (i-1) in Proposition~\ref{thm:URS}.

\noindent (i-2) The upper bound is derived using the same approach as in part (i-2) of Proposition~\ref{thm:URS}. To establish the lower bound, using the convexity of $f$, we may write
$$ \nabla f(x^*)^\top\mathbb{E}[(\RRme{\bar x^R}-x^*)]\leq \mathbb{E}[f(\RRme{\bar x^R})] - f^* ,$$
Note that $\nabla f(x^*)^\top\mathbb{E}[(\RRme{\bar x^R}-x^*)]$ is not necessarily nonnegative. By adding and subtracting $\Pi_{X^*_h}[\RRme{\bar x^R}]$, we obtain
$$\nabla f(x^*)^\top\mathbb{E}[(\RRme{\bar x^R}- \Pi_{X^*_h}[\RRme{\bar x^R}]+\Pi_{X^*_h}[\RRme{\bar x^R}]-x^*)]\leq \mathbb{E}[f(\RRme{\bar x^R})] - f^* .$$
In view of $\Pi_{X^*_h}[\RRme{\bar x^R}] \in X^*_h$, we have $\nabla f(x^*)^\top\mathbb{E}\left[\left(\Pi_{X^*_h}[\RRme{\bar x^R}]-x^*\right)\right]\geq 0$. Thus, we obtain 
$$\nabla f(x^*)^\top\mathbb{E}\left[\left(\RRme{\bar x^R}- \Pi_{X^*_h}[\RRme{\bar x^R}]\right)\right]\leq \mathbb{E}[f(\RRme{\bar x^R})] - f^* .$$
Applying the Cauchy-Schwarz inequality, we obtain
\vspace{-3pt}
$$ -\|\nabla f(x^*)\|\,{\left\|\mathbb{E}\left[\RRme{\bar x^R}- \Pi_{X^*_h}[{\RRme{\bar x^R}}]\right]\right\|}\leq \mathbb{E}[f(\RRme{\bar x^R}) ]- f^* .$$
{Invoking Jensen's inequality, we obtain}
\[{ -\|\nabla f(x^*)\|\,\mathbb{E}\left[\left\|\RRme{\bar x^R}- \Pi_{X^*_h}[\fyy{\RRme{\bar x^R}}]\right\|\right]\leq \mathbb{E}[f(\RRme{\bar x^R}) ]- f^* .}\]
\noindent Since $\mathbb{E}[\mbox{dist}(\RRme{\bar x^R},X^*_h)]=\mathbb{E}\left[\left\|\RRme{\bar x^R}- \Pi_{X^*_h}[{\RRme{\bar x^R}}]\right\|\right]$, we obtain the lower bound in (i-2).

\noindent (ii-1) This result is identical to (i-1). 

\noindent (ii-2) The proof follows the same arguments as those in part (ii-2) of Proposition~\ref{thm:URS}.

\noindent (iii-1) Consider the relation in (i-2). By the weak sharp property, we obtain 
\begin{align}\label{eqn:URS-i-2-cvx} -\tfrac{\|\nabla f(x^*)\|}{\alpha}(\mathbb{E}[h(\RRme{\bar x^R})]-h^*) \leq\mathbb{E}[ f(\RRme{\bar x^R})]-f^*.
\end{align}
We also have
\begin{align}\label{eqn:URS-i-1-cnx}
\mathbb{E}[h(\RRme{\bar x^R})]  -h^* +\eta\left ( \mathbb{E}[ f(\RRme{\bar x^R})]-f^*\right) \leq \texttt{Err}_{\eta}.
\end{align}
Multiplying both sides of~\eqref{eqn:URS-i-2-cvx} by $\eta$ and substituting the resulting lower bound for $\eta\left(\mathbb{E}[f(\RRme{\bar x^R})]-f^*\right)$ into~\eqref{eqn:URS-i-1-cnx}, we obtain
$$\mathbb{E}[h(\RRme{\bar x^R})]  -h^*  -\tfrac{\eta\|\nabla f(x^*)\|}{\alpha}(\mathbb{E}[h(\RRme{\bar x^R})]-h^*) \leq \texttt{Err}_{\eta}.$$
Under the condition $\eta\le\tfrac{\alpha}{2\|\nabla f(x^*)\|}$, we obtain 
\begin{align}\label{eqn:URS-i-3}
\tfrac{1}{2}(\mathbb{E}[h(\RRme{\bar x^R}) ] -h^*)  \leq \texttt{Err}_{\eta}. 
\end{align}

\noindent (iii-2) The upper bound follows directly from (i-2). For the lower bound, consider \eqref{eqn:URS-i-2-cvx}. Applying the result of (iii-1), we obtain $-\tfrac{\|\nabla f(x^*)\|}{\alpha}(2\texttt{Err}_{\eta}) \leq \mathbb{E}[f(\RRme{\bar x^R})] -f^*.$
This completes the proof of (iii-2).

\end{proof}}
In the following proposition, as in the strongly convex case, we derive bounds for the outer- and inner-level objective functions in problem~\eqref{problem: main problem}. These bounds are key to establishing the communication complexity guarantees of Theorem~\ref{thm:thm3}.}
\begin{proposition}\em\label{Pro:Prop4}Let Assumptions~\ref{assump:URS} and~\ref{assumption:main3} hold \RRme{with $\mu_f=0$}. \far{For all $\gamma_l \le \tfrac{1}{16(L_h+\eta L_f)^2(1+ B^2)K\gamma_g}$, and $\gamma_g\ge1$ we have}

\noindent [Case i] The following error bounds hold true. 
\vspace{-10pt}
\begin{align*}
\text{(i-1) }&0 \le \mathbb{E}\left[(h(\bar{x}^{R})-h(x^*))\right]\le\tfrac{3}{\tilde \gamma R}\left(\mathbb{E}\left[\|\RRme{\bar x}^{0}-x_\eta^*\|^2\right]-\mathbb{E}\left[\|\RRme{\bar x}^R-x_\eta^*\|^2\right]\right)+2\eta M\\
&+3\tilde \gamma\left(\tfrac{(2\eta^2\sigma_f^2+2\sigma_h^2)}{KS}+\tfrac{(N-S)4G^2}{NS}\right)+3\tilde \gamma^2\left(\tfrac{9(L_h+\eta L_f)(2\eta^2\sigma_f^2+2\sigma_h^2)}{K\gamma_g^2}+\tfrac{18(L_h+\eta L_f)G^2}{{\gamma_g^2}}\right).
\end{align*}
\vspace{-25pt}
\begin{align*}
\text{(i-2) } & -\|\nabla f(x^*)\| \,\mbox{dist}(\bar{x}^{R},X^*_h)\\ &\le\mathbb{E}\left[(f(\bar{x}^{R})-f(x^*))\right]\le\tfrac{3}{\eta\tilde \gamma R}\left(\mathbb{E}\left[\|\RRme{\bar x}^{0}-x_\eta^*\|^2\right]-\mathbb{E}\left[\|\RRme{\bar x}^R-x_\eta^*\|^2\right]\right)\\
&+\tfrac{3\tilde \gamma}{\eta}\left(\tfrac{(2\eta^2\sigma_f^2+2\sigma_h^2)}{KS}+\tfrac{(N-S)4G^2}{NS}\right)+\tfrac{3\tilde \gamma^2}{\eta}\left(\tfrac{9(L_h+\eta L_f)(2\eta^2\sigma_f^2+2\sigma_h^2)}{K\gamma_g^2}+\tfrac{18(L_h+\eta L_f)G^2}{{\gamma_g^2}}\right).
\end{align*}
\noindent [Case ii] Suppose $X^*_h$ is $\alpha$-weak sharp with order $\kappa > 1$. Then, the following holds. 
\vspace{-10pt}
\begin{align*}
\text{(ii-1) }&0 \le \mathbb{E}\left[(h(\bar{x}^{R})-h(x^*))\right]\le\tfrac{3}{\tilde \gamma R}\left(\mathbb{E}\left[\|\RRme{\bar x}^{0}-x_\eta^*\|^2\right]-\mathbb{E}\left[\|\RRme{\bar x}^R-x_\eta^*\|^2\right]\right)+2\eta M\\
&+3\tilde \gamma\left(\tfrac{(2\eta^2\sigma_f^2+2\sigma_h^2)}{KS}+\tfrac{(N-S)4G^2}{NS}\right)+3\tilde \gamma^2\left(\tfrac{9(L_h+\eta L_f)(2\eta^2\sigma_f^2+2\sigma_h^2)}{K\gamma_g^2}+\tfrac{18(L_h+\eta L_f)G^2}{{\gamma_g^2}}\right).
\end{align*}
\vspace{-25pt}
\begin{align*}
\text{(ii-2) }&-\|\nabla f(x^*)\|\left(\tfrac{3}{\alpha \tilde \gamma R}\left(\mathbb{E}\left[\|\RRme{\bar x}^{0}-x_\eta^*\|^2\right]-\mathbb{E}\left[\|\RRme{\bar x}^R-x_\eta^*\|^2\right]\right)+\tfrac{2\eta M}{\alpha}\right.\\
&\left.+\tfrac{3\tilde \gamma}{\alpha}\left(\tfrac{(2\eta^2\sigma_f^2+2\sigma_h^2)}{KS}+\tfrac{(N-S)4G^2}{NS}\right)+\tfrac{3\tilde \gamma^2}{\alpha}\left(\tfrac{9(L_h+\eta L_f)(2\eta^2\sigma_f^2+2\sigma_h^2)}{K\gamma_g^2}+\tfrac{18(L_h+\eta L_f)G^2}{{\gamma_g^2}}\right)\right)^\frac{1}{\kappa}\\&\le\mathbb{E}\left[f(\bar x^R)-f(x^*)\right]\le\tfrac{3}{\eta\tilde \gamma R}\left(\mathbb{E}\left[\|\RRme{\bar x}^{0}-x_\eta^*\|^2\right]-\mathbb{E}\left[\|\RRme{\bar x}^R-x_\eta^*\|^2\right]\right)\\
&+\tfrac{3\tilde \gamma}{\eta}\left(\tfrac{(2\eta^2\sigma_f^2+2\sigma_h^2)}{KS}+\left(1-\tfrac{S}{N}\right)\tfrac{4}{S}G^2\right)+\tfrac{3\tilde \gamma^2}{\eta}\left(\tfrac{9(L_h+\eta L_f)(2\eta^2\sigma_f^2+2\sigma_h^2)}{K\gamma_g^2}+\tfrac{18(L_h+\eta L_f)G^2}{{\gamma_g^2}}\right).
\end{align*}
\noindent [Case iii] If $X^*_h$ is $\alpha$-weak sharp with order $\kappa=1$ and  $\eta\le{\alpha}/{2\|\nabla f(x^*)\|}$, then 
\vspace{-10pt}
\begin{align*}
\text{(iii-1) }0&\le \mathbb{E}\left[(h(\bar{x}^{R})-h(x^*))\right]\le\tfrac{6}{\tilde \gamma R}\left(\mathbb{E}\left[\|\RRme{\bar x}^{0}-x_\eta^*\|^2\right]-\mathbb{E}\left[\|\RRme{\bar x}^R-x_\eta^*\|^2\right]\right)\\
&+6\tilde \gamma\left(\tfrac{(2\eta^2\sigma_f^2+2\sigma_h^2)}{KS}+\tfrac{(N-S)4G^2}{NS}\right)+6\tilde \gamma^2\left(\tfrac{9(L_h+\eta L_f)(2\eta^2\sigma_f^2+2\sigma_h^2)}{K\gamma_g^2}+\tfrac{18(L_h+\eta L_f)G^2}{{\gamma_g^2}}\right).
\end{align*}
\vspace{-25pt}
\begin{align*}
\text{(iii-2) }&-\tfrac{6\|\nabla f(x^*)\|}{\alpha\tilde \gamma R}\left(\mathbb{E}\left[\|\RRme{\bar x}^{0}-x_\eta^*\|^2\right]-\mathbb{E}\left[\|\RRme{\bar x}^R-x_\eta^*\|^2\right]\right)-\tfrac{6\tilde \gamma \|\nabla f(x^*)\|}{\alpha}\\
&\times\left(\left(\tfrac{(2\eta^2\sigma_f^2+2\sigma_h^2)}{KS}+\tfrac{(N-S)4G^2}{NS}\right)+\tilde \gamma\left(\tfrac{9(L_h+\eta L_f)(2\eta^2\sigma_f^2+2\sigma_h^2)}{K\gamma_g^2}+\tfrac{18(L_h+\eta L_f)G^2}{{\gamma_g^2}}\right)\right)\\ &\le\mathbb{E}\left[(f(\bar{x}^{R})-f(x^*))\right]\le\tfrac{3}{\eta\tilde \gamma R}\left(\mathbb{E}\left[\|\RRme{\bar x}^{0}-x_\eta^*\|^2\right]-\mathbb{E}\left[\|\RRme{\bar x}^R-x_\eta^*\|^2\right]\right)\\
&+\tfrac{3\tilde \gamma}{\eta}\left(\tfrac{(2\eta^2\sigma_f^2+2\sigma_h^2)}{KS}+\tfrac{(N-S)4G^2}{NS}\right)+\tfrac{3\tilde \gamma^2}{\eta}\left(\tfrac{9(L_h+\eta L_f)(2\eta^2\sigma_f^2+2\sigma_h^2)}{K\gamma_g^2}+\tfrac{18(L_h+\eta L_f)G^2}{{\gamma_g^2}}\right).
\end{align*}

\end{proposition}

\Rme{\begin{proof}[Proof of Theorem~\ref{thm:thm3}]Invoking Proposition~\ref{Pro:Prop4}, the proof is similar to the proof of Theorem~\ref{thm:thm2}.
\end{proof}}

\Rme{\begin{remark}\em As we mentioned earlier, an interesting finding of this work is that our analysis leads to suitable rules for choosing $\eta$. Such a choice is algorithm-specific and depends on the underlying assumptions.
\end{remark}}


\subsection{Nonconvex Case}
\Rme{{In this subsection, we analyze IPIR-FedAvg method for solving problem~\eqref{problem: main problem-ncvx}, covering the nonconvex case of~$f$. We provide a detailed analysis that establishes Theorem~\ref{thm:thm ncvx}.}}

{\bf{\Rme{History of the method.}}} \far{We} define the history of Algorithm~\ref{algorithm:ncvx} as $\mathcal{F}_t\triangleq \mathcal{F}_{t-1}\cup\left(\cup_{i=1}^N \{\zeta_{i,0}^{t-1},\zeta_{i,0}^{t-1},\ldots,\zeta_{i,KR_{t-1}}^{t-1}\}\right)$, for all $k \geq 1$, and $\mathcal{F}_0 \triangleq \{x_0\}$. We also let $\mathbb{E}[\bullet\mid \mathcal{F}_t]$ denote the conditional expectation with respect to the filtration $\mathcal{F}_t$.

We should note that \Rme{$x^t_{\eta_t}$} is an inexact solution to the projection problem. Let us denote the error for this approximation at iteration $k$ of Algorithm~\ref{algorithm:ncvx} as \Rme{$e^t_{\eta_t}$}, and define \Rme{$e^t_{\eta_t}\triangleq x^t_{\eta_t}-\Pi_{X^*_h}[y^t]$}. Therefore, we can rewrite the update rule of Algorithm~\ref{algorithm:ncvx} as \Rme{$y^{t+1}:=y^t - \gamma (\nabla F(y^t) + \frac{1}{\lambda}e^t_{\eta_t})$}. In the inner-loop of Algorithm~\ref{algorithm:ncvx} we employ the \RRRme{StR-FedAvg} for solving the projection problem \eqref{problem:peojection problem}. So, based on Proposition~\ref{Prop:Prop2}, we establish the upper bound for \Rme{$\mathbb{E}[\|e^t_{\eta_t}\|^2\mid \mathcal{F}_t]=\mathbb{E}[\|{x}^{t}_{\eta_t}-x^{{t}*}\|^2\mid\mathcal{F}_t]$} for different cases in the following proposition, where $x^{{t}*}$ is the optimal solution to the projection problem at iteration $k$ of the outer-loop of Algorithm~\ref{algorithm:ncvx} ($\Pi_{X^*_h}[y^t]=x^{{t}*}$). Notably, the following result follows from Proposition~\ref{Prop:Prop2} in that the key differences lie in replacing $\mu_f$ and $L_f$ with $1$, applying Lemma~\ref{remark:bound in terms of the obj-ncvx}, and substituting $\bar{x}^R$ with \RRme{${\bar x}_{\eta_t}^{t}$}.
\begin{proposition}\em\label{Prop:Prop9}Consider problem~\eqref{problem:peojection problem}. Let $g$ be $1$-strongly convex and $h_i$ be convex for all $i\in[N]$. Let \Rme{$\theta:=\tfrac{1}{1-{0.5\eta \tilde \gamma_t}}$}. Invoking Theorem~\ref{thm:URS} and Proposition~\ref{Prop:Prop2}, and using \RRRme{StR-FedAvg} in Algorithm~\ref{algorithm:R-FedAvg and R-SCAFFOLD} as the inner-loop of Algorithm~\ref{algorithm:ncvx}, for all \Rme{$\gamma_{l_t} \le \tfrac{1}{16(L_h+\eta )^2(1+B_{ncvx}^2)K\gamma_{g_t}}$, $\gamma_{g_t}\ge1$} and $R_t\ge {16(L_h+\RRRme{\hat\eta})^2(1+B_{ncvx}^2)}$, \RRRme{where $\hat\eta=11$}, the following results hold.

\noindent [Case i] Suppose $X^*_h$ is $\alpha$-weak sharp with order $\kappa > 1$. \Rme{Then, for $\tilde \gamma_t:= \tfrac{1}{R_t^a}$ and $\eta_t:= \tfrac{p\RRme{\ln}(R_t)}{R_t^b}$, where $0< b<a\le1$ and $p\ge 1$, the following results hold. }
  \begin{align*} &\Rme{\mathbb{E}\left[\|{\RRme{\bar x}}^{t}_{\eta_t}-x^{{t}*}\|^2\mid\mathcal{F}_t\right] }\leq \tfrac{1}{ R_t^{{p}/{2}}-1}\mathbb{E}\left[\|\RRme{\bar x}^{0}-x^{{t}*}\|^2\mid\mathcal{F}_t\right]\\&+\tfrac{2}{p\RRme{\ln}(R_t)R_t^{a-b}}\left(\left(\tfrac{(2\eta^2\sigma_f^2+2\sigma_h^2)}{KS}+\tfrac{(N-S)4G_{ncvx}^2}{NS}\right)+\left(\tfrac{9(L_h+\eta )(2\eta^2\sigma_f^2+2\sigma_h^2)}{K\Rme{\gamma_{g_t}^2}\RRRme{R_t^{a}}}+\tfrac{18(L_h+\eta )G_{ncvx}^2}{{\Rme{\gamma_{g_t}^2}R_t^{a}}}\right)\right)\\
&+\tfrac{\|\nabla f(x^{t*})\|}{\alpha^{{1}/{\kappa}}}\left(\tfrac{p\RRme{\ln}(R_t)}{ R_t^b(R_t^{{p}/{2}}-1)}\mathbb{E}\left[\|\RRme{\bar x}^{0}-x^{{t}*}\|^2\mid\mathcal{F}_t\right]+\tfrac{2}{R_t^{a}}\left(\tfrac{(2\eta^2\sigma_f^2+2\sigma_h^2)}{KS}+\left(1-\tfrac{S}{N}\right)\tfrac{4}{S}G_{ncvx}^2\right)\right.\\
&\left.+\tfrac{2}{R_t^{2a}}\left(\tfrac{9(L_h+\eta )(2\eta^2\sigma_f^2+2\sigma_h^2)}{K\Rme{\gamma_{g_t}^2}}+\tfrac{18(L_h+\eta )G_{ncvx}^2}{{\Rme{\gamma_{g_t}^2}}}\right)+\tfrac{2p\RRme{\ln}(R_t)}{R_t^{b}} M\right)^\frac{1}{\kappa}.
\end{align*}
\RRRme{
\noindent [Case ii] If $X^*_h$ is $\alpha$-weak sharp with order $\kappa=1$, $\tilde \gamma_t:= \tfrac{p\RRme{\ln}(R_t)}{R_t}$, for some $p\ge 1$, and $\eta_t:=\eta\le{\alpha}/{2\|\nabla F_\lambda(x^{t*})\|}$, then 
 \begin{align*} &\mathbb{E}\left[\|\bar{x}^{t}_{\eta}-x^{{t}*}\|^2\mid\mathcal{F}_t\right] \leq  \tfrac{1}{ \eta(R_t^{{p}/{2}}-1)}\mathbb{E}\left[\|\bar{x}^{0}-x^{{t}*}\|^2\mid\mathcal{F}_t\right]\\&+\tfrac{2p\RRme{\ln}(R_t)}{\eta R_t}\left(\left(\tfrac{(2\eta^2\sigma_f^2+2\sigma_h^2)}{KS}+\tfrac{(N-S)4G_{ncvx}^2}{NS}\right)+\left(\tfrac{p\ln(R_t)9(L_h+\eta )(2\eta^2\sigma_f^2+2\sigma_h^2)}{K\gamma_g^2R_t}+\tfrac{p\ln(R_t)18(L_h+\eta )G_{ncvx}^2}{{\gamma_g^2R_t}}\right)\right).
\end{align*}}
\end{proposition}

In the following result, we establish the convergence rates for Algorithm~\ref{algorithm:ncvx} in addressing problem~\eqref{problem: main problem-ncvx}.

 \begin{proof}[Proof of Theorem~\ref{thm:thm ncvx}]
 \hspace{0.1 \in}\noindent [Case i] \noindent (i-1) 
From the $L_f+\frac{2}{\lambda}$-smoothness of $\RRme{F_\lambda}$, we have $\RRme{F_\lambda}(y^{t+1}) \leq \RRme{F_\lambda}(y^t) +\nabla \RRme{F_\lambda}(y^\mje{t})^{\fyy{\top}}(y^{t+1}-y^t) + \tfrac{L_f\lambda+2}{2\lambda}\|y^{t+1}-y^t\|^2$. We should note that \Rme{$x^{t}_{\eta_t}$} is an inexact solution to the projection problem. Based on the approximation error at iteration $t$ of Algorithm~\ref{algorithm:ncvx} denoted by \Rme{$e^{t}_{\eta_t}$}, and invoking the update rule of this algorithm, we obtain
\begin{align*}
 \RRme{F_\lambda}(y^{t+1}) \le &\RRme{F_\lambda}(y^t) +\nabla \RRme{F_\lambda}(y^t)^{\fyy{\top}}( -\gamma (\nabla \RRme{F_\lambda}(y^t) + \RRme{\tfrac{1}{\lambda}} \Rme{e^{t}_{\eta_t}})) + \tfrac{L_f\lambda+2}{2\lambda}\|- \gamma (\nabla \RRme{F_\lambda}(y^t) +\RRme{ \tfrac{1}{\lambda}}\Rme{ e^{t}_{\eta_t}})\|^2\\
  \le& \RRme{F_\lambda}(y^t) -\gamma \|\nabla \RRme{F_\lambda}(y^t)\|^2+\tfrac{\gamma}{2\RRme{\nu}} \|\nabla \RRme{F_\lambda}(y^t)\|^2+\tfrac{\RRme{\nu}\gamma}{2\RRme{\lambda^2}}\|\Rme{e^{t}_{\eta_t}}\|^2+ \tfrac{\gamma^2(L_f\lambda+2)}{\lambda}\|\nabla \RRme{F_\lambda}(y^t)\|^2\\
 &+\tfrac{\gamma^2(L_f\lambda+2)}{\lambda^3}\|\Rme{e^{t}_{\eta_t}}\|^2,
 \end{align*}
where $\RRme{\nu}>0$ is a positive constant. Setting $\RRme{\nu}=2$, we obtain
$\RRme{F_\lambda}(y^{t+1})  \le \RRme{F_\lambda}(y^t) -{\gamma} \left(\tfrac{5}{4}- \tfrac{\gamma(L_f\lambda+2)}{\lambda}\right)\|\nabla \RRme{F_\lambda}(y^t)\|^2+\left(\tfrac{\gamma}{\RRme{\lambda^2}}+\tfrac{\gamma^2(L_f\lambda+2)}{\lambda^3}\right)\|\Rme{e^{t}_{\eta_t}}\|^2.$
By setting $\gamma \le \tfrac{3\lambda}{4(L_f\lambda+2)}$, we have $\left(\tfrac{5}{4}- \tfrac{\gamma(L_f\lambda+2)}{\lambda}\right)\le \tfrac{1}{2}$. Therefore, we obtain
$\RRme{F_\lambda}(y^{t+1})   \leq \RRme{F_\lambda}(y^t) -\tfrac{\gamma}{2}\|\nabla \RRme{F_\lambda}(y^t)\|^2+\left(\tfrac{\gamma}{\lambda}+\tfrac{\gamma^2(L_f\lambda+2)}{\lambda^3}\right)\|\Rme{e^{t}_{\eta_t}}\|^2.$ Rearranging the terms, and summing both sides over $t=1,2,\ldots, T$ where $T\geq 1$, we obtain
$$
\textstyle{\sum_{t=1}^{T}}\|\nabla \RRme{F_\lambda}(y^t)\|^2 \leq \tfrac{2}{\gamma}\left(\RRme{F_\lambda}(y^0)- \RRme{F_\lambda}(y^{T})\right)+\left(\tfrac{2}{\RRme{\lambda^2}}+\tfrac{2\gamma(L_f\lambda+2)}{\lambda^3}\right)\sum_{t=1}^{T}\|\Rme{e^{t}_{\eta_t}}\|^2.$$
Taking conditional expectations on both sides \far{and} invoking Jensen's inequality, we obtain
\begin{align}
\tfrac{1}{T}\textstyle{\sum_{t=1}^{T}}\mathbb{E}[\|\nabla \RRme{F_\lambda}(y^t)\|^2\mid \mathcal{F}_t] \!\leq\!  \frac{2\gamma^{-1}\left(\RRme{F_\lambda}(\RRme{y}^0)- \RRme{F^*_\lambda}\right)}{T}+\left(\tfrac{2}{\RRme{\lambda^2}}+\tfrac{2\gamma(L_f\lambda+2)}{\lambda^3}\right)\frac{\sum_{t=1}^{T}\mathbb{E}[\|\Rme{e^{t}_{\eta_t}}\|^2\mid \mathcal{F}_t]}{T}.\label{eq:ncnvx}
\end{align}
In the inner loop of Algorithm~\ref{algorithm:ncvx} we employ the \RRRme{StR-FedAvg} for solving the projection problem. By applying Proposition~\ref{Prop:Prop9} and substituting the upper bound for \Rme{$\mathbb{E}[\|{\RRme{\bar x}}^t_{\eta_t}-x^{t*}\|^2\mid \mathcal{F}_t]=\mathbb{E}\left[\|e^{t}_{\eta_t}\|^2\mid \mathcal{F}_t\right] $} into the preceding inequality, for \Rme{$\tilde \gamma_t:= \tfrac{1}{R_t^a}$ and $\eta_t:= \tfrac{p\RRme{\ln}(R_t)}{R_t^b}$, for $p=2,\RRRme{a=\tfrac{7}{12},b=\tfrac{5}{12}}$, and $R_t=t$,} \RRme{and defining $\hat\eta\triangleq\max_t\eta_t=11$, we then take another expectation to obtain}
\begin{align*}
&\mathbb{E}[\|\nabla \RRme{F_\lambda}(y^{T^*})\|^2] \leq \tfrac{2\gamma^{-1}\left(\RRme{F_\lambda}(y^0)- \RRme{F^*_\lambda}\right)}{T}+\left(\tfrac{2}{\RRme{\lambda^2}}+\tfrac{2\gamma(L_f\lambda+2)}{\lambda^3}\right)\mathbb{E}\left[\|\RRme{\bar x}^{0}-x^{{t}*}\|^2\right]\RRRme{{\textstyle\sum_{t=2}^{T+1}\tfrac{1}{ t-1}}/{T}}\\
&+\left(\tfrac{2}{\RRme{\lambda^2}}+\tfrac{2\gamma(L_f\lambda+2)}{\lambda^3}\right)\left(\tfrac{(2\RRme{\hat\eta}^2\sigma_f^2+2\sigma_h^2)}{KS}+\left(1-\tfrac{S}{N}\right)\tfrac{4}{S}G_{ncvx}^2\right)\RRRme{{\textstyle\sum_{t=2}^{T+1}\tfrac{1}{\RRme{\ln}(t)t^{\RRRme{{1}/{6}}}}}/{T}}\\
&+\left(\tfrac{2}{\RRme{\lambda^2}}+\tfrac{2\gamma(L_f\lambda+2)}{\lambda^3}\right)\left(\tfrac{9(L_h+\RRme{\hat\eta} )(2\RRme{\hat\eta}^2\sigma_f^2+2\sigma_h^2)}{K\Rme{\gamma_{g_t}^2}}+\tfrac{18(L_h+\RRme{\hat\eta} )G_{ncvx}^2}{{\Rme{\gamma_{g_t}^2}}}\right)\RRRme{{\textstyle\sum_{t=2}^{T+1}\tfrac{1}{\RRme{\ln}(t)t^{\RRRme{3/4}}}}/{T}}\\
&+\left(\tfrac{2}{\RRme{\lambda^2}}+\tfrac{2\gamma(L_f\lambda+2)}{\lambda^3}\right)\tfrac{\|\nabla f(x^{t*})\|}{\alpha^\frac{1}{\kappa}}\left(\mathbb{E}\left[\|\RRme{\bar x}^{0}-x^{{t}*}\|^2\right]\right)^{\tfrac{1}{\kappa}}\RRRme{{\textstyle\sum_{t=2}^{T+1}\left(\tfrac{2\RRme{\ln}(t)}{ t^\RRRme{{5}/{12}}(t-1)}\right)^\frac{1}{\kappa}}/{T}}\\
&+\left(\tfrac{2}{\RRme{\lambda^2}}+\tfrac{2\gamma(L_f\lambda+2)}{\lambda^3}\right)\tfrac{\|\nabla f(x^{t*})\|}{\alpha^\frac{1}{\kappa}}\left(\tfrac{(2\RRme{\hat\eta}^2\sigma_f^2+2\sigma_h^2)}{KS}+\left(1-\tfrac{S}{N}\right)\tfrac{4}{S}G_{ncvx}^2\right)^\frac{1}{\kappa}\RRRme{{\textstyle\sum_{t=1}^{T}\left({2}/{t^\RRRme{{7}/{12}}}\right)^\frac{1}{\kappa}}/{T}}\\
&+\left(\tfrac{2}{\RRme{\lambda^2}}+\tfrac{2\gamma(L_f\lambda+2)}{\lambda^3}\right)\tfrac{\|\nabla f(x^{t*})\|}{\alpha^\frac{1}{\kappa}}\left(\tfrac{9(L_h+\RRme{\hat\eta} )(2\RRme{\hat\eta}^2\sigma_f^2+2\sigma_h^2)}{K\Rme{\gamma_{g_t}^2}}+\tfrac{18(L_h+\RRme{\hat\eta} )G_{ncvx}^2}{{\Rme{\gamma_{g_t}^2}}}\right)^\frac{1}{\kappa}\RRRme{{\textstyle\sum_{t=1}^{T}\left({2}/{t^\RRRme{{7}/{6}}}\right)^\frac{1}{\kappa}}/{T}}\\
&+M^\frac{1}{\kappa}\RRRme{{\textstyle\sum_{t=1}^{T}\left(\tfrac{4\RRme{\ln}(t)}{t^\RRRme{{5}/{12}}}\right)^\frac{1}{\kappa}}/{T}}.
\end{align*}
By substituting the summations with their upper bounds, we obtain
\begin{align*}
&\mathbb{E}[\|\nabla \RRme{F_\lambda}(y^{T^*})\|^2] \leq  \tfrac{2\gamma^{-1}\left(\RRme{F_\lambda}(y^0)- \RRme{F^*_\lambda}\right)}{T}+\left(\tfrac{2}{\RRme{\lambda^2}}+\tfrac{2\gamma(L_f\lambda+2)}{\lambda^3}\right)\mathbb{E}\left[\|\RRme{\bar x}^{0}-x^{{t}*}\|^2\right]\tfrac{\RRme{1+\ln}(T)}{T}\\
&+\left(\tfrac{2}{\RRme{\lambda^2}}+\tfrac{2\gamma(L_f\lambda+2)}{\lambda^3}\right)\left(\tfrac{(2\RRme{\hat\eta}^2\sigma_f^2+2\sigma_h^2)}{KS}+\left(1-\tfrac{S}{N}\right)\tfrac{4}{S}G_{ncvx}^2\right)\RRRme{\tfrac{6}{5T^\frac{1}{6}\ln(T)}}\\
&+\left(\tfrac{2}{\RRme{\lambda^2}}+\tfrac{2\gamma(L_f\lambda+2)}{\lambda^3}\right)\left(\tfrac{9(L_h+\RRme{\hat\eta} )(2\RRme{\hat\eta}^2\sigma_f^2+2\sigma_h^2)}{K\Rme{\gamma_{g_t}^2}}+\tfrac{18(L_h+\RRme{\hat\eta} )G_{ncvx}^2}{{\Rme{\gamma_{g_t}^2}}}\right)\RRRme{\tfrac{3}{2T^{3/4}\RRme{\ln}(T)}}+\RRRme{\tfrac{(4M)^{1/\kappa}(24\kappa-5)}{12\kappa-5}\tfrac{(\ln T)^{1/\kappa}}{T^{5/(12\kappa)}}}\\
&+\left(\tfrac{2}{\RRme{\lambda^2}}+\tfrac{2\gamma(L_f\lambda+2)}{\lambda^3}\right)\tfrac{\|\nabla f(x^{t*})\|}{\alpha^\frac{1}{\kappa}}\left(\mathbb{E}\left[\|\RRme{\bar x}^{0}-x^{{t}*}\|^2\right]\right)^{\tfrac{1}{\kappa}}\RRRme{
\tfrac{2(4)^{1/\kappa}(\ln T+1)^{1/\kappa}}{T^{17/12\kappa}}}\\
&+\left(\tfrac{2}{\RRme{\lambda^2}}+\tfrac{2\gamma(L_f\lambda+2)}{\lambda^3}\right)\tfrac{\|\nabla f(x^{t*})\|}{\alpha^\frac{1}{\kappa}}\left(\tfrac{2(2\RRme{\hat\eta}^2\sigma_f^2+2\sigma_h^2)}{KS}+2\left(1-\tfrac{S}{N}\right)\tfrac{4}{S}G_{ncvx}^2\right)^\frac{1}{\kappa}\RRRme{\tfrac{2^{1/\kappa}}{1-\tfrac{7}{12\kappa}}\tfrac{1}{T^{7/(12\kappa)}}}\\
&+\left(\tfrac{2}{\RRme{\lambda^2}}+\tfrac{2\gamma(L_f\lambda+2)}{\lambda^3}\right)\tfrac{\|\nabla f(x^{t*})\|}{\alpha^\frac{1}{\kappa}}\left(\tfrac{18(L_h+\RRme{\hat\eta})(2\RRme{\hat\eta}^2\sigma_f^2+2\sigma_h^2)}{K\Rme{\gamma_{g_t}^2}}+\tfrac{36(L_h+\RRme{\hat\eta} )G_{ncvx}^2}{{\Rme{\gamma_{g_t}^2}}}\right)^\frac{1}{\kappa}\RRRme{\tfrac{1}{T^{1/\kappa}}}.
\end{align*}

\noindent (i-2) For the feasibility bound, invoking the $\alpha$-weak sharpness property of $X^*_h$ with order $\kappa \geq 1$ we have
\begin{align} \Rme{\mathbb{E}[\hbox{dist}({\RRme{\bar x}}^t_{\eta_t}, X^*_h)\mid \mathcal{F}_t]\le \tfrac{1}{\alpha}(\mathbb{E}[h({\RRme{\bar x}}^{t}_{\eta_t})\mid \mathcal{F}_t ]- h^*)^{1/\kappa}}.\label{eq:feasibility}
\end{align}
By applying the bound for \Rme{$(\mathbb{E}[h({\RRme{\bar x}}^{t}_{\eta_t}) ]- h^*)$} from Proposition~\ref{Prop:Prop2} and adapting the notation to align with the nonconvex case, we obtain
\begin{align*}
\Rme{\mathbb{E}\left[h({\RRme{\bar x}}^t_{\eta_t})-h^*\mid \mathcal{F}_t\right]}\le&\tfrac{1}{\Rme{ \tilde \gamma_{t}}\sum_{r=1}^{R_t}\theta^{r}}\mathbb{E}\left[\|\RRme{\bar x}^{0}-x^{t*}\|^2\mid \mathcal{F}_t\right]+{\Rme{\tilde \gamma_t}}\left(\tfrac{(2\eta^2\sigma_f^2+2\sigma_h^2)}{KS}+\left(1-\tfrac{S}{N}\right)\tfrac{4G_{ncvx}^2}{S}\right)\\
&+\Rme{{\tilde \gamma_t^2}}\left(\tfrac{9(L_h+\eta )(2\eta^2\sigma_f^2+2\sigma_h^2)}{K\Rme{\gamma_{g_t}^2}}+\tfrac{18(L_h+\eta )G_{ncvx}^2}{{\Rme{\gamma_{g_t}^2}}}\right)+2\Rme{\eta_t} M.
\end{align*}
Based on Theorem~\ref{thm:thm2}, for \Rme{$\tilde \gamma_t:= \tfrac{1}{R_t^a}$ and $\eta_t:= \tfrac{p\ln(R_t)}{R_t^b}$}, we obtain
\begin{align*}
&\Rme{\mathbb{E}\left[h({\RRme{\bar x}}^t_{\eta_t})-h^*\mid \mathcal{F}_t\right]}\le\tfrac{p\RRme{\ln}(R_t)}{2R_t^b\left(R_t^{{p}/{2}}-1\right)}\mathbb{E}\left[\|\RRme{\bar x}^{0}-x^{t*}\|^2\mid \mathcal{F}_t\right]+\tfrac{p\RRme{\ln}(R_t)}{R_t^{b}}2 M\\
&+\tfrac{1}{R_t^{a}}\left(\tfrac{(2\eta^2\sigma_f^2+2\sigma_h^2)}{KS}+\left(1-\tfrac{S}{N}\right)\tfrac{4G_{ncvx}^2}{S}\right)+\tfrac{1}{R_t^{2a}}\left(\tfrac{9(L_h+\eta )(2\eta^2\sigma_f^2+2\sigma_h^2)}{K\Rme{\gamma_{g_t}^2}}+\tfrac{18(L_h+\eta )G_{ncvx}^2}{\Rme{{\gamma_{g_t}^2}}}\right).
\end{align*}
\RRRme{By substituting this upper bound into \eqref{eq:feasibility}, \RRme{setting} $p=2,a=\tfrac{7}{12},b=\tfrac{5}{12}$, and $R_t=t$, \RRme{defining $\hat\eta\triangleq\max_t\eta_t=\frac{6}{\mathrm{e}  \mu_f^b}$,} and then taking another expectation, we obtain
\begin{align*}
\Rme{\mathbb{E}[\hbox{dist}({\RRme{\bar x}}^t_{\eta}, X^*_h)]}\le&\tfrac{1}{\alpha}\left(\tfrac{\RRme{\ln}(t)}{t^{{5}/{12}}\left(t-1\right)}\mathbb{E}\left[\|\RRme{\bar x}^{0}-x^{t*}\|^2\right]+\tfrac{1}{t^{{7}/{12}}}\left(\tfrac{(2\RRme{\hat\eta}^2\sigma_f^2+2\sigma_h^2)}{KS}+\left(1-\tfrac{S}{N}\right)\tfrac{4G_{ncvx}^2}{S}\right)\right.\\
&\left.+\tfrac{1}{t^{{7}/{6}}}\left(\tfrac{9(L_h+\RRme{\hat\eta} )(2\RRme{\hat\eta}^2\sigma_f^2+2\sigma_h^2)}{K\Rme{\gamma_{g_t}^2}}+\tfrac{18(L_h+\RRme{\hat\eta} )G_{ncvx}^2}{{\Rme{\gamma_{g_t}^2}}}\right)+\tfrac{\RRme{\ln}(t)}{t^{{5}/{12}}}4 M\right)^{1/\kappa}.
\end{align*}
\noindent [Case ii] \noindent (i-1) By invoking~\eqref{eq:ncnvx} and applying [Case ii] of Proposition~\ref{Prop:Prop9} with $p=2$ and $R_t=t$, \RRme{defining $\hat\eta=11$,} and then substituting the resulting bound into~\eqref{eq:ncnvx} and taking expectation once more, we obtain

\begin{align*}
\mathbb{E}[\|\nabla \RRme{F_\lambda}(y^{T^*})\|^2]\leq\!  &\frac{2\gamma^{-1}\left(\RRme{F_\lambda}(\RRme{y}^0)- \RRme{F^*_\lambda}\right)}{T}+\left(\tfrac{2}{\RRme{ \eta\lambda^2}}+\tfrac{2\gamma(L_f\lambda+2)}{ \eta\lambda^3}\right)\mathbb{E}\left[\|\bar{x}^{0}-x^{{t}*}\|^2\right]\tfrac{\sum_{t=2}^{T+1}\tfrac{1}{{t}-1}}{T}\\
&+\left(\tfrac{8}{\RRme{ \eta\lambda^2}}+\tfrac{8\gamma(L_f\lambda+2)}{ \eta\lambda^3}\right)\left(\tfrac{(2\hat\eta^2\sigma_f^2+2\sigma_h^2)}{KS}+\tfrac{(N-S)4G_{ncvx}^2}{NS}\right)\tfrac{\sum_{t=2}^{T+1}\tfrac{\ln(t)}{{t}}}{T}\\
&+\left(\tfrac{16}{\RRme{ \eta\lambda^2}}+\tfrac{16\gamma(L_f\lambda+2)}{ \eta\lambda^3}\right)\left(\tfrac{9(L_h+\hat\eta )(2\hat\eta^2\sigma_f^2+2\sigma_h^2)}{K\gamma_g^2}+\tfrac{18(L_h+\hat\eta )G_{ncvx}^2}{{\gamma_g^2}}\right)\tfrac{\sum_{t=2}^{T+1}\tfrac{\left(\ln(t)\right)^2}{{t^2}}}{T}.
\end{align*}
Substituting the upper bounds of the summations yields
\begin{align*}
\mathbb{E}[\|\nabla \RRme{F_\lambda}(y^{T^*})\|^2]\leq\!  &\frac{2\gamma^{-1}\left(\RRme{F_\lambda}(\RRme{y}^0)- \RRme{F^*_\lambda}\right)}{T}+\left(\tfrac{2}{\RRme{ \eta\lambda^2}}+\tfrac{2\gamma(L_f\lambda+2)}{ \eta\lambda^3}\right)\mathbb{E}\left[\|\bar{x}^{0}-x^{{t}*}\|^2\right]\tfrac{2\ln(T)}{T}\\
&+\left(\tfrac{8}{\RRme{ \eta\lambda^2}}+\tfrac{8\gamma(L_f\lambda+2)}{ \eta\lambda^3}\right)\left(\tfrac{(2\hat\eta^2\sigma_f^2+2\sigma_h^2)}{KS}+\tfrac{(N-S)4G_{ncvx}^2}{NS}\right)\tfrac{\left(\ln(T+2)\right)^2}{T}\\
&+\left(\tfrac{16}{\RRme{ \eta\lambda^2}}+\tfrac{16\gamma(L_f\lambda+2)}{ \eta\lambda^3}\right)\left(\tfrac{9(L_h+\hat\eta )(2\hat\eta^2\sigma_f^2+2\sigma_h^2)}{K\gamma_g^2}+\tfrac{18(L_h+\hat\eta )G_{ncvx}^2}{{\gamma_g^2}}\right)\tfrac{2}{T}.
\end{align*}

\noindent (ii-2) To derive the feasibility bound, we invoke the $\alpha$-weak sharpness property of $X^*_h$ with order $\kappa = 1$, which yields
\begin{align} \Rme{\mathbb{E}[\hbox{dist}({\RRme{\bar x}}^t_{\eta_t}, X^*_h)\mid \mathcal{F}_t]\le \tfrac{1}{\alpha}(\mathbb{E}[h({\RRme{\bar x}}^{t}_{\eta_t})\mid \mathcal{F}_t ]- h^*)}.\label{eq:new we discussed}
\end{align}
Applying the bound for \Rme{$(\mathbb{E}[h({\RRme{\bar x}}^{t}_{\eta_t}) ]- h^*$} from Proposition~\ref{Prop:Prop2} and adapting the notation to the nonconvex case yields
\begin{align*}
\mathbb{E}[h({\RRme{\bar x}}^{t}_{\eta_t})\mid \mathcal{F}_t ]- h^*)\le&\tfrac{2}{ \tilde \gamma_t\sum_{r=1}^{R}\theta^{r}}\mathbb{E}\left[\|\RRme{\bar x}^{0}-x^{t*}\|^2\mid \mathcal{F}_t \right]+{2\tilde \gamma_t}\left(\tfrac{(2\eta^2\sigma_f^2+2\sigma_h^2)}{KS}+\tfrac{4G^2(N-S)}{NS}\right)\\
&+{2\tilde \gamma_t^2}\left(\tfrac{9(L_h+\eta L_f)(2\eta^2\sigma_f^2+2\sigma_h^2)}{K\gamma_{g_t}^2}+\tfrac{18(L_h+\eta L_f)G^2}{{\gamma_{g_t}^2}}\right).
\end{align*}
Based on Theorem~\ref{thm:thm2}, let $\tilde{\gamma}_t := \tfrac{p\RRme{\ln}(R_t)}{R_t}$ with $p=2$ and $R_t=t$, and choose $\eta_t := \eta \le \alpha/2\|\nabla F_\lambda(x^{t*})\|$. \RRme{Define $\hat{\eta}=11$.} Substituting this bound into~\eqref{eq:new we discussed} and taking another expectation yields
\begin{align*}
\mathbb{E}[\hbox{dist}({\RRme{\bar x}}^t_{\eta_t}, X^*_h)]\le&\tfrac{\eta}{ \alpha t}\mathbb{E}\left[\|\RRme{\bar x}^{0}-x^{t*}\|^2 \right]+\tfrac{4\RRme{\ln}(t)}{t}\left(\tfrac{(2\hat\eta^2\sigma_f^2+2\sigma_h^2)}{KS}+\tfrac{4G^2(N-S)}{NS}\right)\\
&+\tfrac{8\left(\RRme{\ln}(t)\right)^2}{t^2}\left(\tfrac{9(L_h+\hat\eta L_f)(2\hat\eta^2\sigma_f^2+2\sigma_h^2)}{K\gamma_{g_t}^2}+\tfrac{18(L_h+\hat\eta L_f)G^2}{{\gamma_{g_t}^2}}\right).
\end{align*}

}
\end{proof}
\begin{remark}\em
Although Case~[i] of Theorem~\ref{thm:thm ncvx} assumes that $X_h^*$ is $\alpha$-weak sharp of order $\kappa > 1$, the exact value of $\kappa$ is not required. In particular, Algorithm~\ref{algorithm:ncvx} does not require knowledge of $\kappa$.
\end{remark}
 \section{Numerical \far{e}xperiments}
In this section, we present numerical to validate our theoretical findings and to demonstrate the efficiency of the proposed algorithms. 
\subsection{Over-parameterized learning with a convex regularizer}
We present experiments on a variant of sparse over-parameterized regression, where the goal is to compute the sparsest solution among multiple optimal solutions, we define the problem as
\begin{align}\label{prob:l1}
&\min_x \quad \|x\|_1 \quad\text{s.t.} \quad x\in  \text{arg}\min_{y} h(y) = \tfrac{1}{2N}\textstyle{\sum_{i=1}^N}\textstyle{\sum_{\tilde{\ell}\in\tilde{\mathcal{D}}_i}}\|U_{i,\tilde{\ell}}y-v_{i,\tilde{\ell}}\|^2 ,
\end{align}
where $N$ denotes the number of clients, $x,y \in\mathbb{R}^{n}$, $U_{i,\tilde{\ell}}$ and $v_{i,\tilde{\ell}}$ are the $\tilde{\ell}$th input and output data of client $i$'s local training dataset $\tilde{\mathcal{D}}_i$, respectively. We use {the Wiki Math Essential, \mje{MNIST, and CIFAR-10} dataset\mje{s}} and consider $N=10$ agents.

\noindent {\bf Moreau envelope of $\|x\|_1$.} Note that the $\ell_1$ norm is nondifferentiable, so we consider the Moreau smoothing of $\|x\|_1$, defined as $M^\mu_{\ell_1}(x) \triangleq \min_y \{\|x\|_1 + \frac{1}{2\mu}\|x-y\|^2\}$. By invoking the results from \cite{beck2017first}, we obtain an explicit expression: $M^\mu_{\ell_1}(x) = \sum_{i=1}^nH_\mu(x_i)$, where $H_\mu(\bullet)$ denotes the Huber function. \me{Moreover}, the gradient of $M^\mu_{\ell_1}(x)$ \me{is given by} $\nabla M^\mu_{\ell_1}(x) = \frac{1}{\mu}(x-\hbox{prox}^\mu_{\ell_1}(x))$, where $\hbox{prox}^\mu_{\ell_1}(x) = \text{sign}(x) \odot\max\{{\bf 0}_n, |x|-\mu\mathbf{1}_n\} \hbox{ for all }x \in \mathbb{R}^n$, \me{and} $\odot$ denotes component-wise product of two vectors. The \me{outer-level} of problem \eqref{prob:l1} now becomes $f(x)=M^\mu_{\ell_1}(x)$.

\noindent {\bf Experiments on different local steps $K$.}  Follow our theory for convex \me{outer-level}, we set local stepsize $\gamma_l = \tfrac{1}{(R+\Gamma)^a}$, regularization parameter $\eta = \tfrac{1}{(R+\Gamma)^b}$ and global stepsize $\gamma_g=\sqrt{N}$, where $R$ is the total communication rounds, $a=\tfrac{1}{2}$, $b=\tfrac{1}{4}$ and $\Gamma>0$. We show the impact of local steps $K$ on the convergence, where we choose $K\in\{1,5,10,20\}$. We use three metrics to demonstrate our results, $\mathbb{E}[\|f(\bar x_r)- f(\bar x_{r-1})\|]$, $\mathbb{E}[f(\bar x_r)]$ and $\mathbb{E}[h(\bar x_r)]$, where $\bar x_r$ denotes the average iterate at communication round $r$. We also define $\bar x_{-1}=:0$ in order to have $\mathbb{E}[\|f(\bar x_r)- f(\bar x_{r-1})\|]$ well-defined for all $r\geq 0$. Note that $\mathbb{E}[f(\bar x_r)]$ and $\mathbb{E}[h(\bar x_r)]$ captures the performance of \me{outer} and \fyy{inner-level} \me{problems}, respectively. According to our theory, $\mathbb{E}[f(\bar x_r)]$ could be decreasing or increasing, therefore we also employ the metric $\mathbb{E}[\|f(\bar x_r)- f(\bar x_{r-1})\|]$ to demonstrate the convergence for the \fyy{outer-level} problem.

\noindent {\bf Results and insights.} The results of this experiment are shown in Figure \ref{fig:IRFedAvg:l1}. As we can see from both plots, given the same communication budget, our algorithms perform better under larger $K$, which shows the benefits of federated methods in solving distributed over-parameterized regression problems.

\begin{table}[htb]
\setlength{\tabcolsep}{0pt}
\centering{
\begin{tabular}{c || c  c  c}
{\footnotesize Setting\ \ }& {WikiMath} & {MNIST} & {CIFAR-10}\\
\hline\\

\rotatebox[origin=c]{90}{{\footnotesize {$\mathbb{E}[f(\bar x_r)]$}}}
&
\begin{minipage}{.3\columnwidth}
\includegraphics[width=1.0\columnwidth]{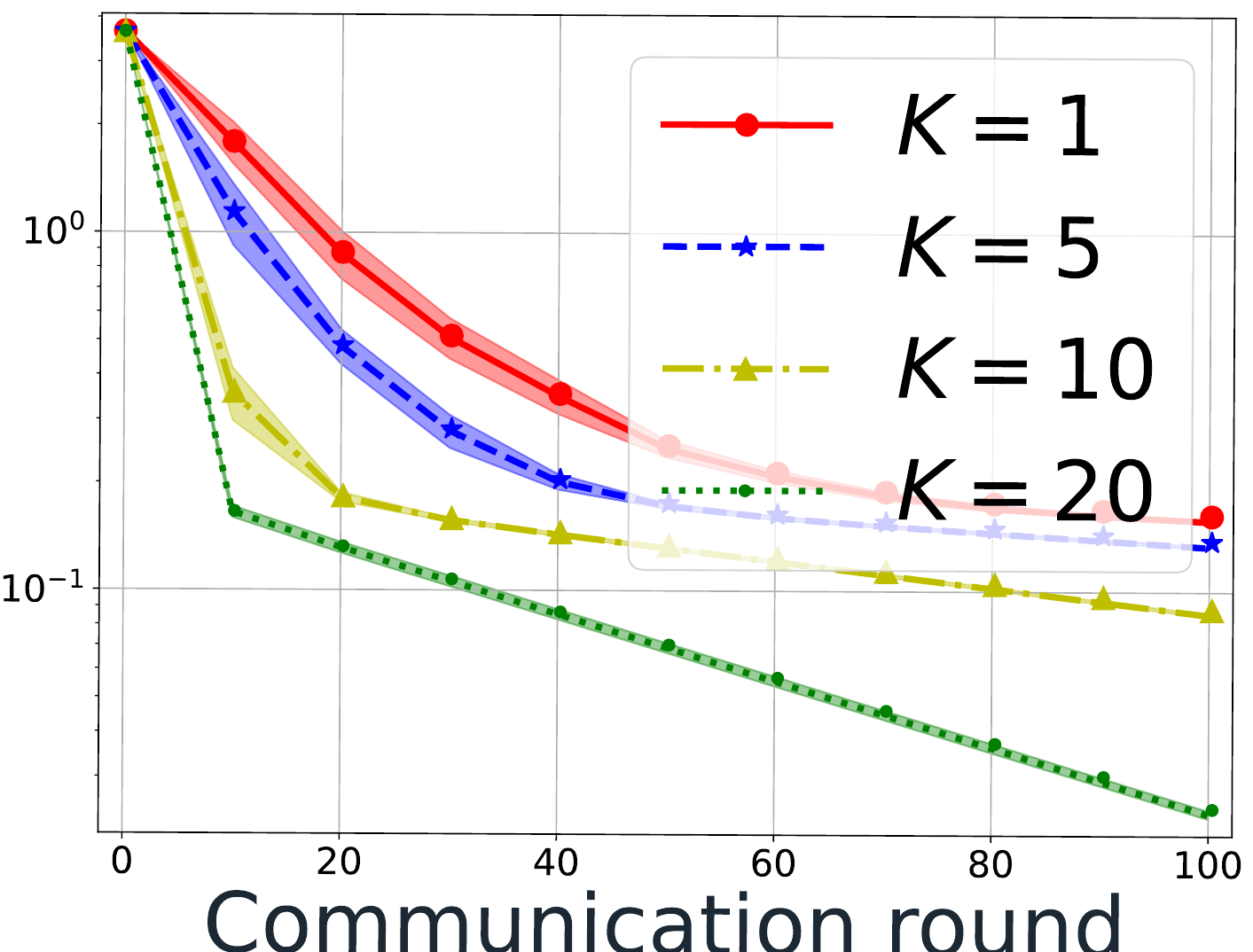}
\end{minipage}
&
\begin{minipage}{.3\columnwidth}
\includegraphics[width=1.0\columnwidth]{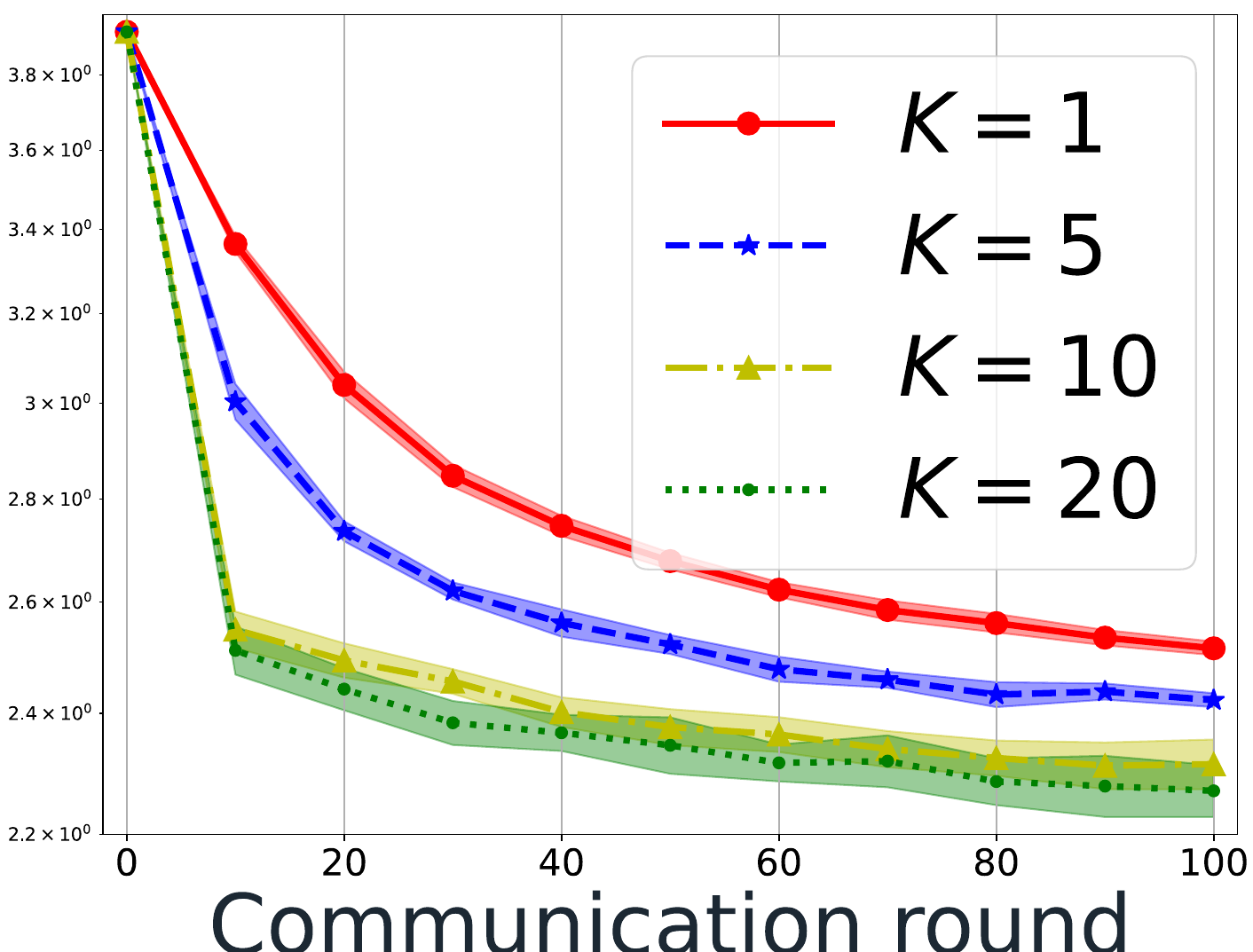}
\end{minipage}
&
\begin{minipage}{.3\columnwidth}
\includegraphics[width=1.0\columnwidth]{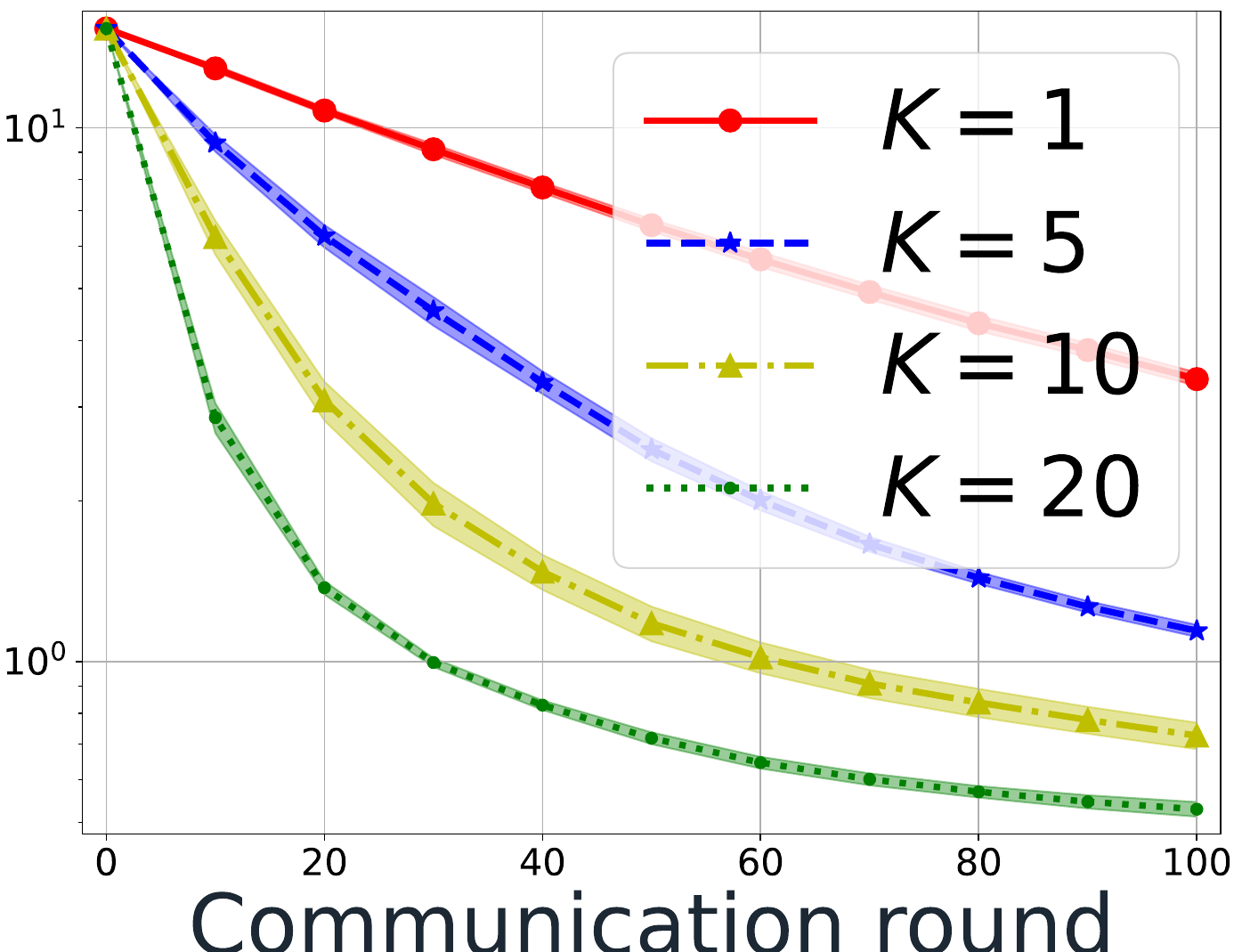}
\end{minipage}
\\

\hbox{}& & &\\
\hline\\

\rotatebox[origin=c]{90}{{\footnotesize {$\mathbb{E}[h(\bar x_r)]$}}}
&
\begin{minipage}{.3\columnwidth}
\includegraphics[width=1.0\columnwidth]{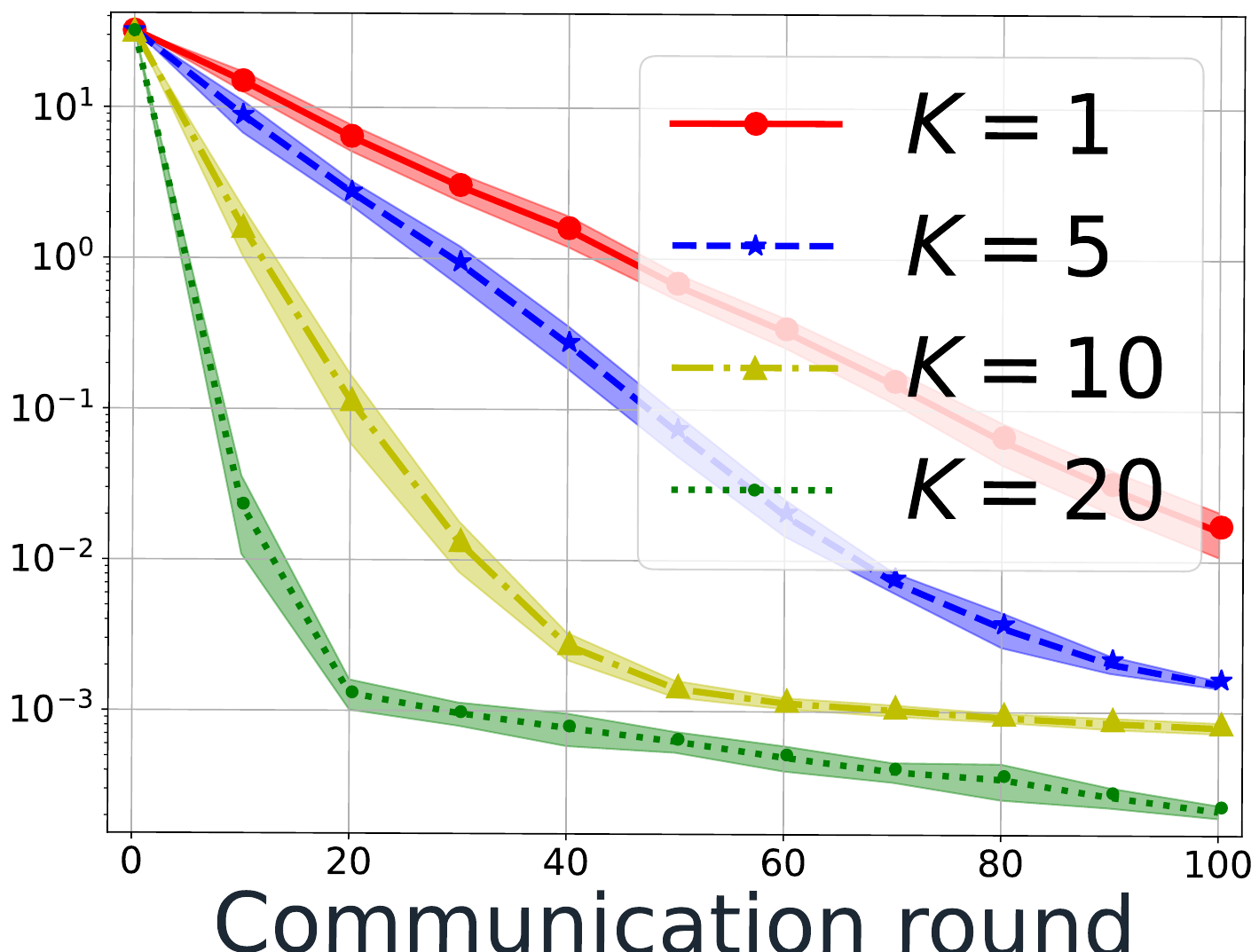}
\end{minipage}
&
\begin{minipage}{.3\columnwidth}
\includegraphics[width=1.0\columnwidth]{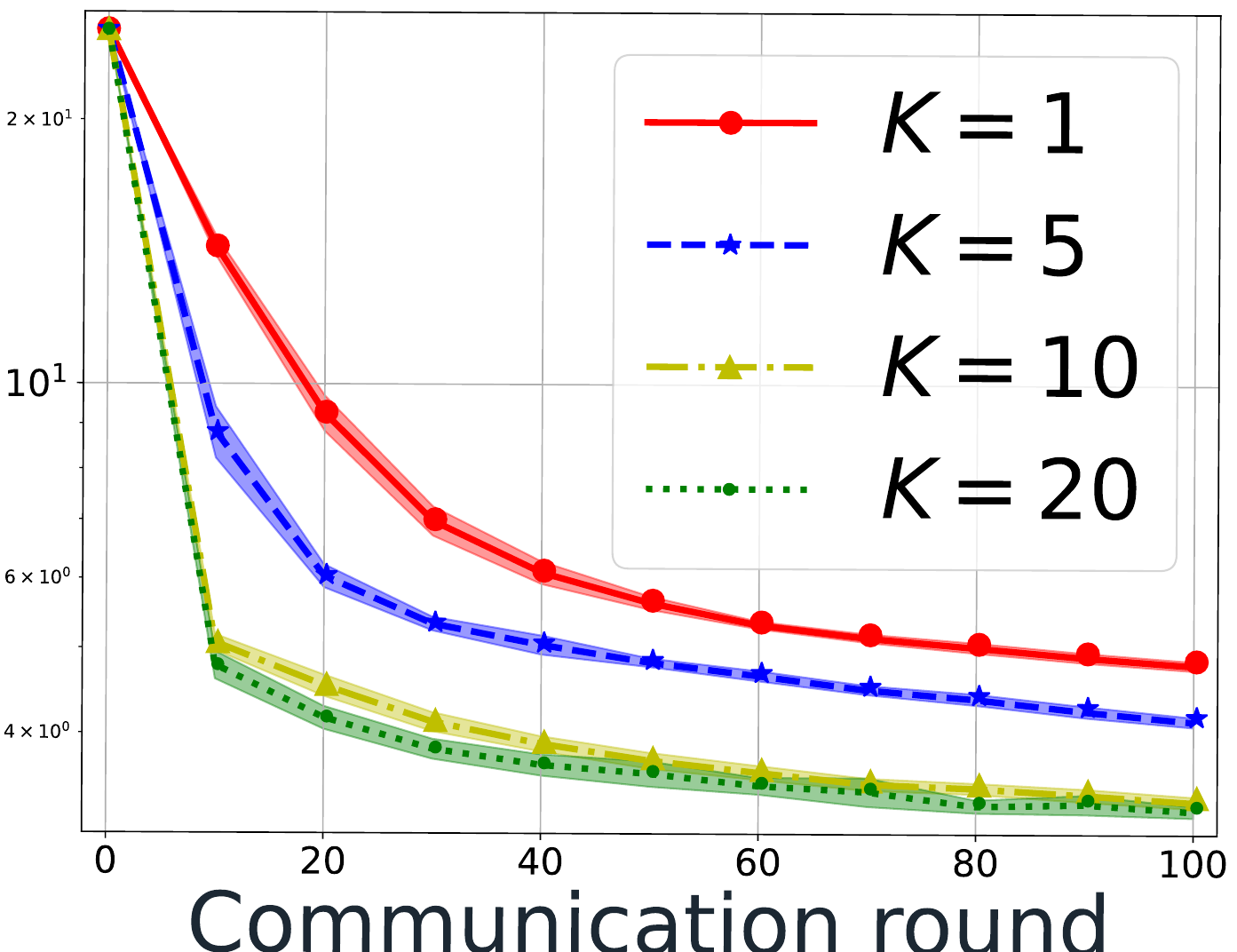}
\end{minipage}
&
\begin{minipage}{.3\columnwidth}
\includegraphics[width=1.0\columnwidth]{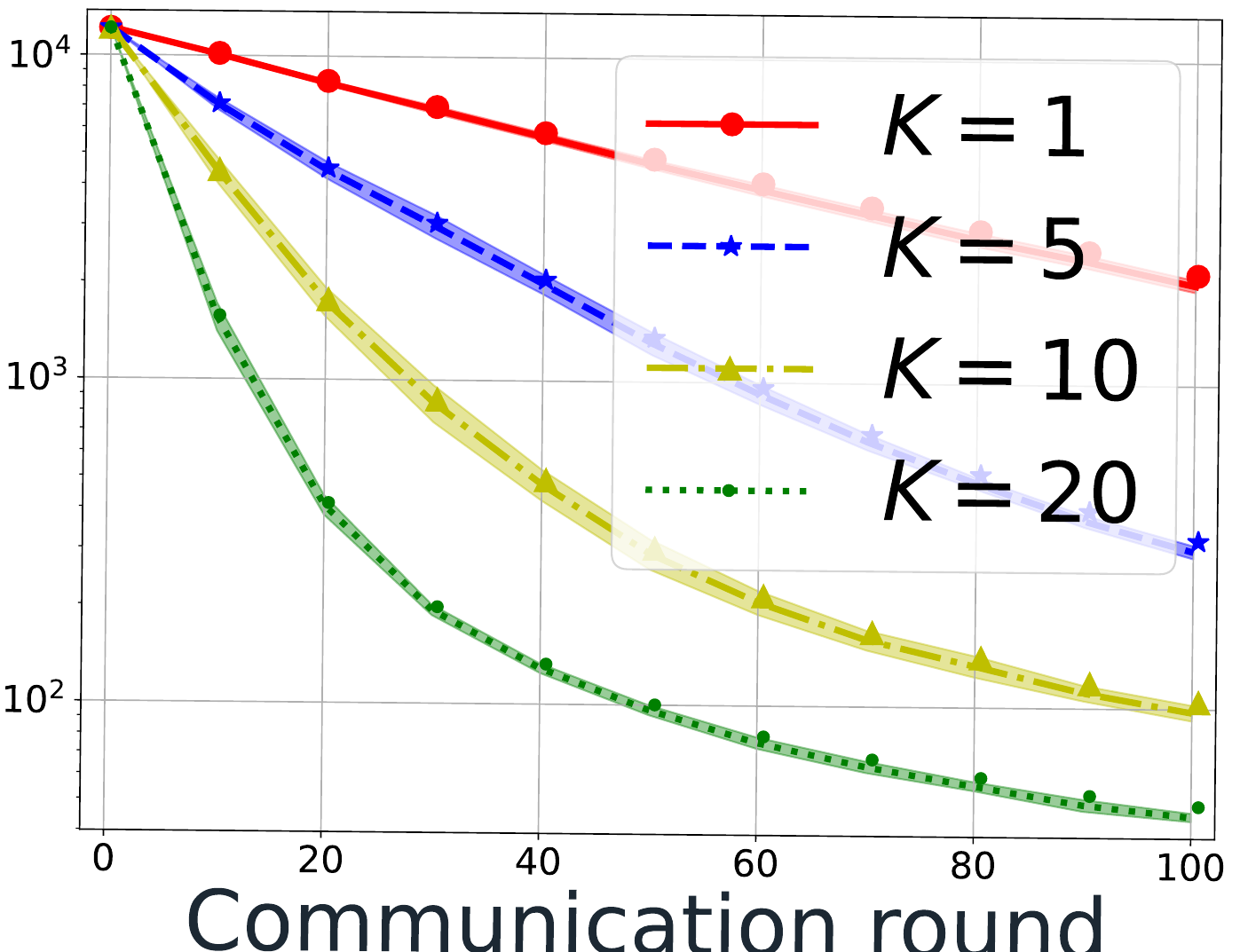}
\end{minipage}

\end{tabular}}

\vspace{.05in}

\captionof{figure}{Performance of Algorithm~\ref{Alg:FEDAVG} for different local steps..}

\label{fig:IRFedAvg:l1}

\vspace{-.1in}
\end{table}

\subsection{Experiments with nonconvex \fyy{outer-level}}
To demonstrate the performance of our method in the nonconvex case, we consider a problem with a sparsity promoted nonconvex regularizer, called log-sum penalty (LSP) function defined as $\text{LSP}(x) = \sum_{i=1}^n\text{log}\left(1+\frac{|x_i|}{\epsilon}\right)$, where $\epsilon>0$. This function is commonly used as a middle-ground between the $\ell_0$ and $\ell_1$ norms to enhance sparsity\me{.  I}t was shown that the log-sum penalty function \me{can promote sparsity more effectively} than the $\ell_1$ norm \cite{candes2008enhancing}. 

\noindent {\bf Moreau envelope of LSP.} Notably\me{,} LSP is prox-friendly and has an explicit expression \me{for its proximal operator} $\text{prox}_{\text{LSP}}(x)$ \cite{prater2022proximity}. The Moreau envelope of LSP can be written as $M^\mu_{\text{LSP}}(x) = f(\text{prox}_{\text{LSP}}(x)) + \frac{1}{2\mu}\|x-\text{prox}_{\text{LSP}}(x)\|^2$, \me{and its gradient is given by} $\nabla M^\mu_{\text{LSP}}(x) = \frac{1}{\mu}(x-\hbox{prox}_{\text{LSP}}(x))$\me{. N}ote that $\text{prox}_{\text{LSP}}(x) = \left(\text{prox}_{\text{LSP}}(x_i)\right)_{i=1}^n$. By invoking the results from \cite{prater2022proximity}, in the case where $\sqrt{\mu}\leq \epsilon$, we have
$$\text{prox}_{\text{LSP}}(x_i)= 
\left\{
\begin{aligned}
& 0, \quad \text{if } |x_i|\leq \tfrac{\mu}{\epsilon}; \\
& \text{sign}(x_i) \tfrac{|x_i|-\epsilon+\sqrt{(|x_i|+\epsilon)^2-4\mu}}{2}, \quad \text{otherwise}.
\end{aligned}
\right.$$
Therefore\me{,} we can replace the \me{outer-level} problem in \eqref{prob:l1} with $M^\mu_{\text{LSP}}(x)$, yielding 
\begin{align}\label{prob:lsp}
&\min_x \quad M^\mu_{\text{LSP}}(x) \quad\text{s.t.} \quad x\in  \text{arg}\min_{y}  h(y) = \tfrac{1}{2N}\textstyle{\sum_{i=1}^N\sum_{\tilde{\ell}\in\tilde{\mathcal{D}}_i}}\|U_{i,\tilde{\ell}}y-v_{i,\tilde{\ell}}\|^2 .
\end{align}
\noindent {\bf Experiments and results.} The experimental results for the nonconvex \me{outer-level} are shown in Figure \ref{exp2:ncvx}. We test the performance of the two-loop scheme proposed for the nonconvex setting\RRRme{,} where \RRRme{StR-FedAvg} is employed in the inner loop. Three choices of local steps $K=\{1, 5, 10, 20\}$ \RRRme{are} used. We observe improvement of both \me{outer} and \fyy{inner-level} \me{problems} with larger local steps, which is indeed very promising for resolving nonconvex simple bilevel problems in FL.

\begin{table}[htb]
\setlength{\tabcolsep}{0pt}
\centering{
\begin{tabular}{c || c  c  c}
{\footnotesize Setting\ \ }& {WikiMath} & {MNIST} & {CIFAR-10}\\
\hline\\

\rotatebox[origin=c]{90}{{\footnotesize {$\mathbb{E}[f(\bar x_r)]$}}}
&
\begin{minipage}{.3\columnwidth}
\includegraphics[width=1.0\columnwidth]{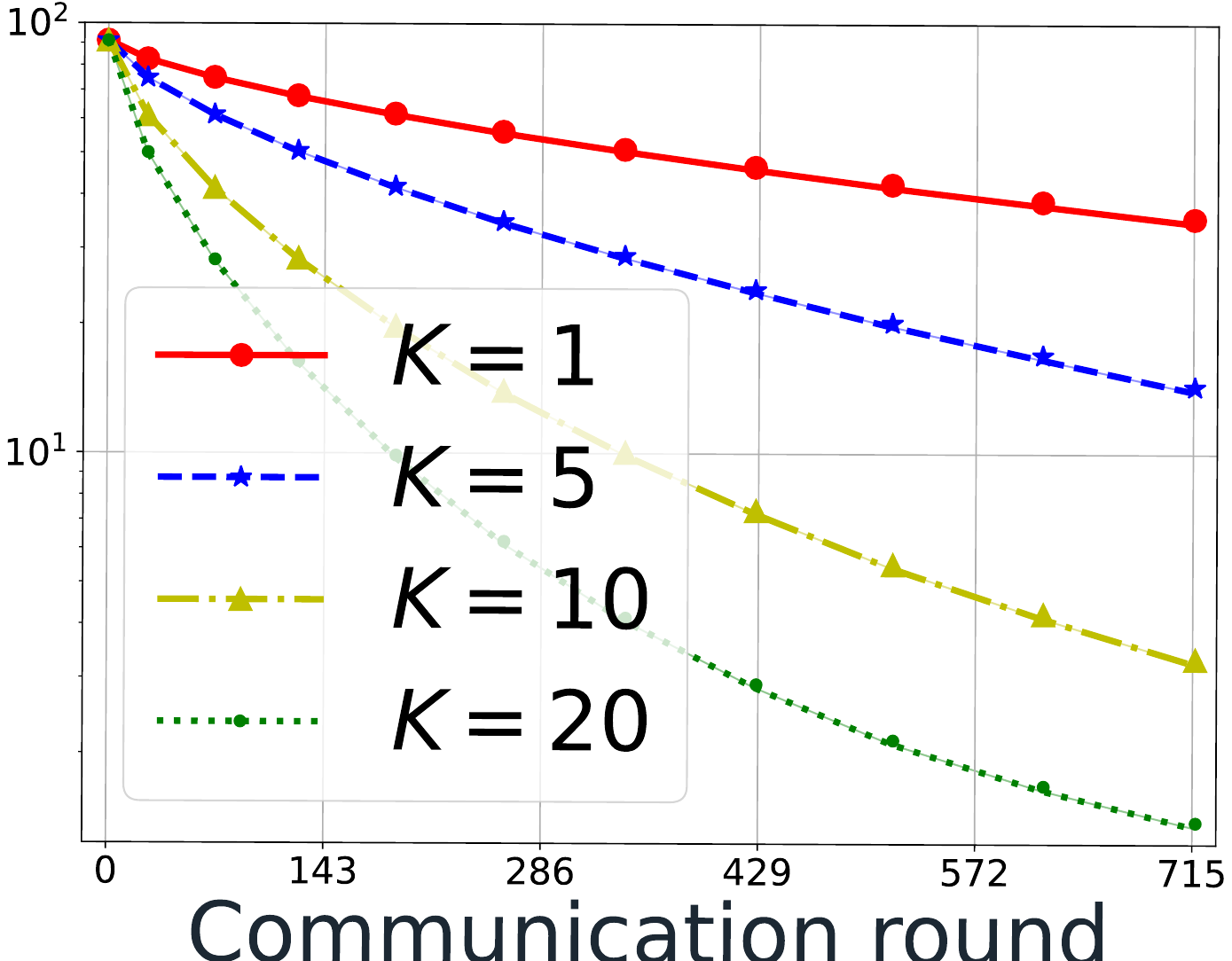}
\end{minipage}
&
\begin{minipage}{.3\columnwidth}
\includegraphics[width=1.0\columnwidth]{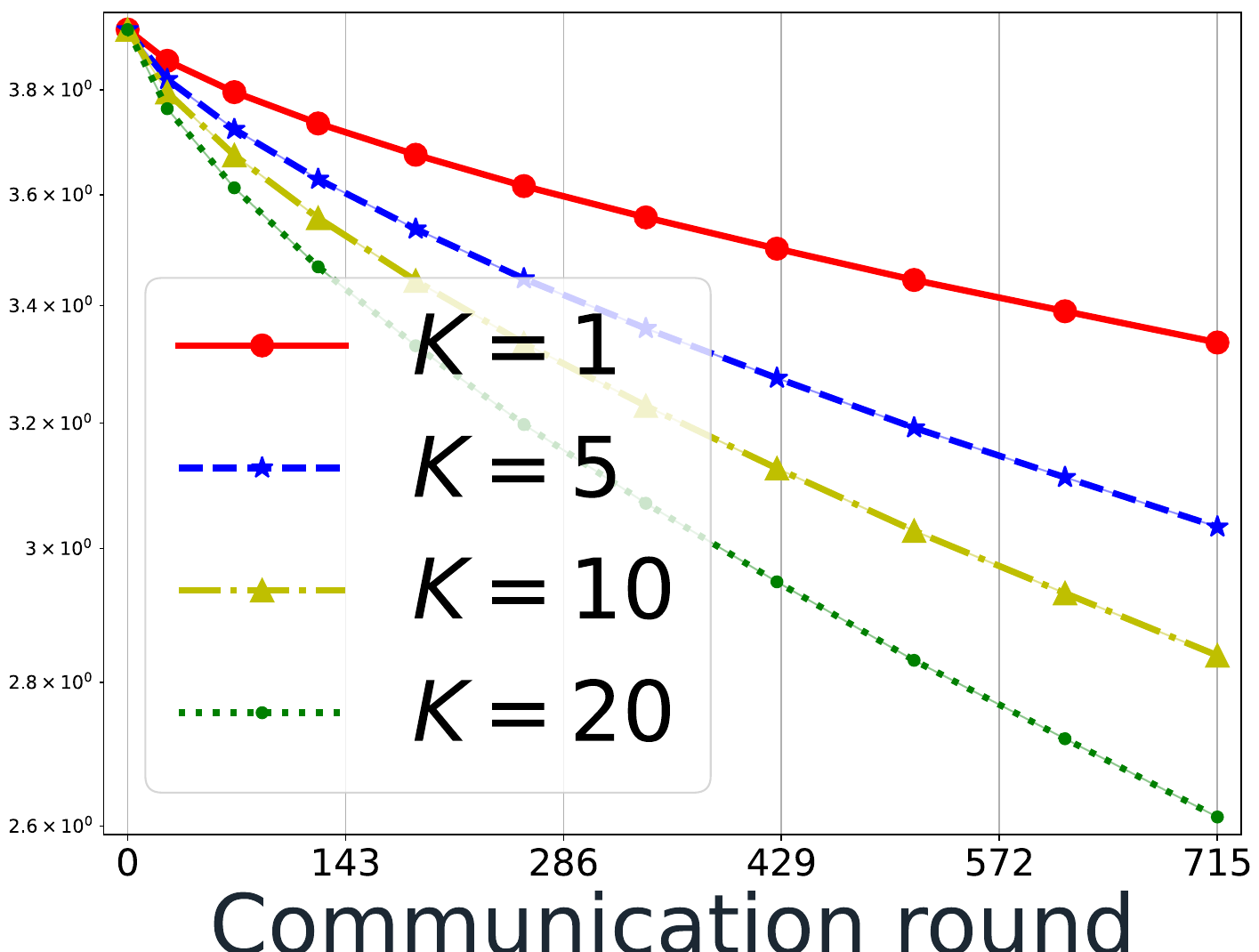}
\end{minipage}
&
\begin{minipage}{.3\columnwidth}
\includegraphics[width=1.0\columnwidth]{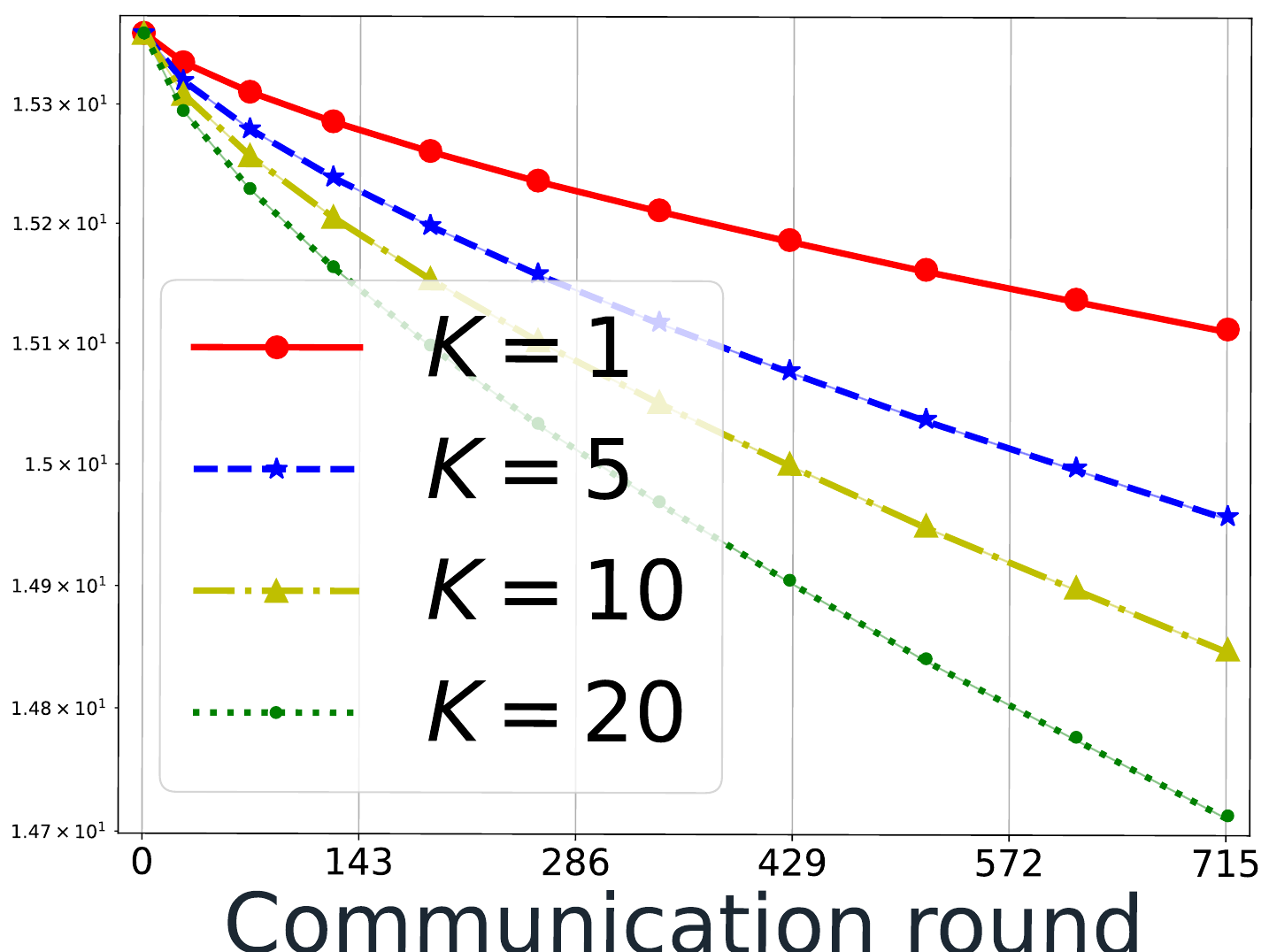}
\end{minipage}
\\

\hbox{}& & &\\
\hline\\

\rotatebox[origin=c]{90}{{\footnotesize {$\mathbb{E}[h(\bar x_r)]$}}}
&
\begin{minipage}{.3\columnwidth}
\includegraphics[width=1.0\columnwidth]{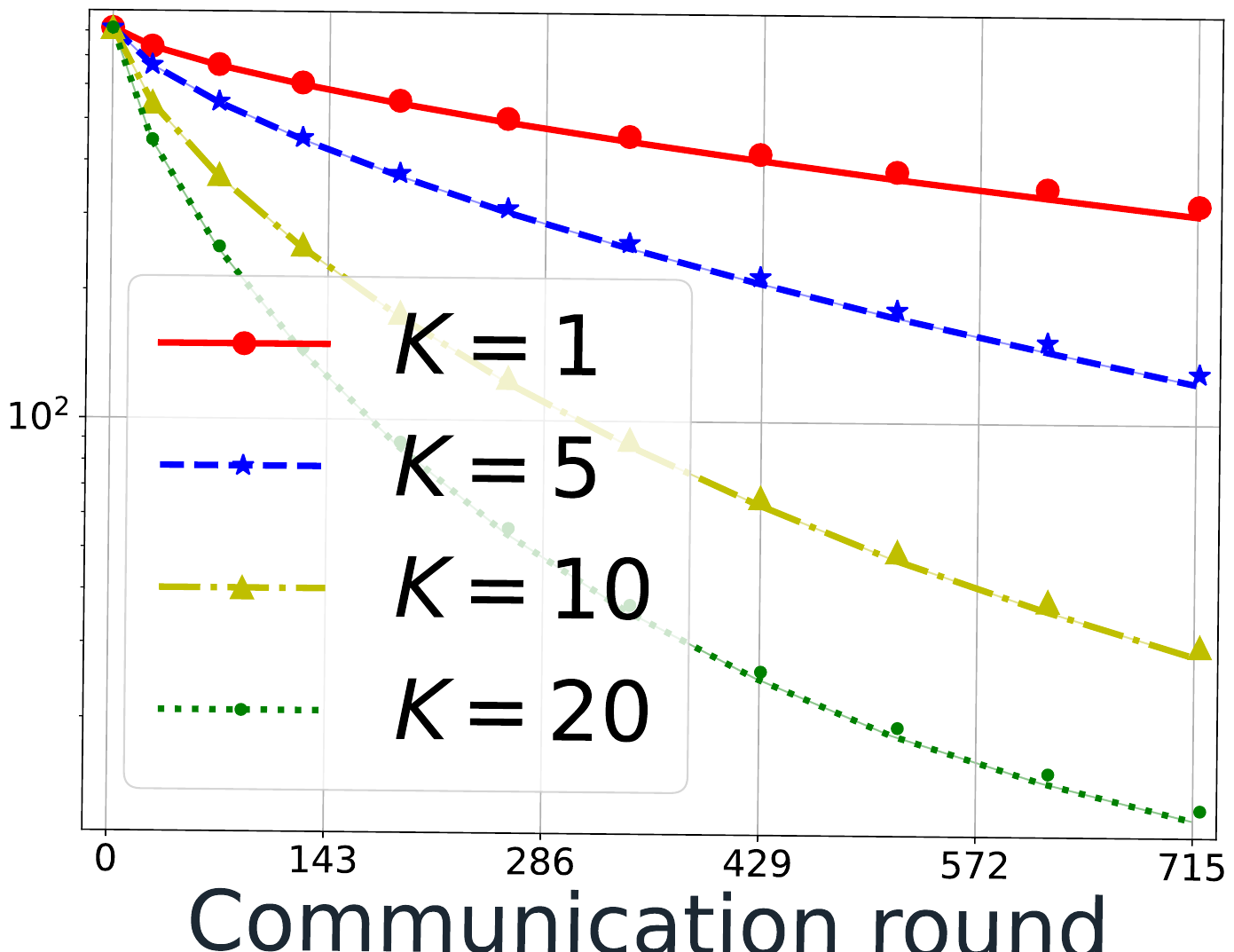}
\end{minipage}
&
\begin{minipage}{.3\columnwidth}
\includegraphics[width=1.0\columnwidth]{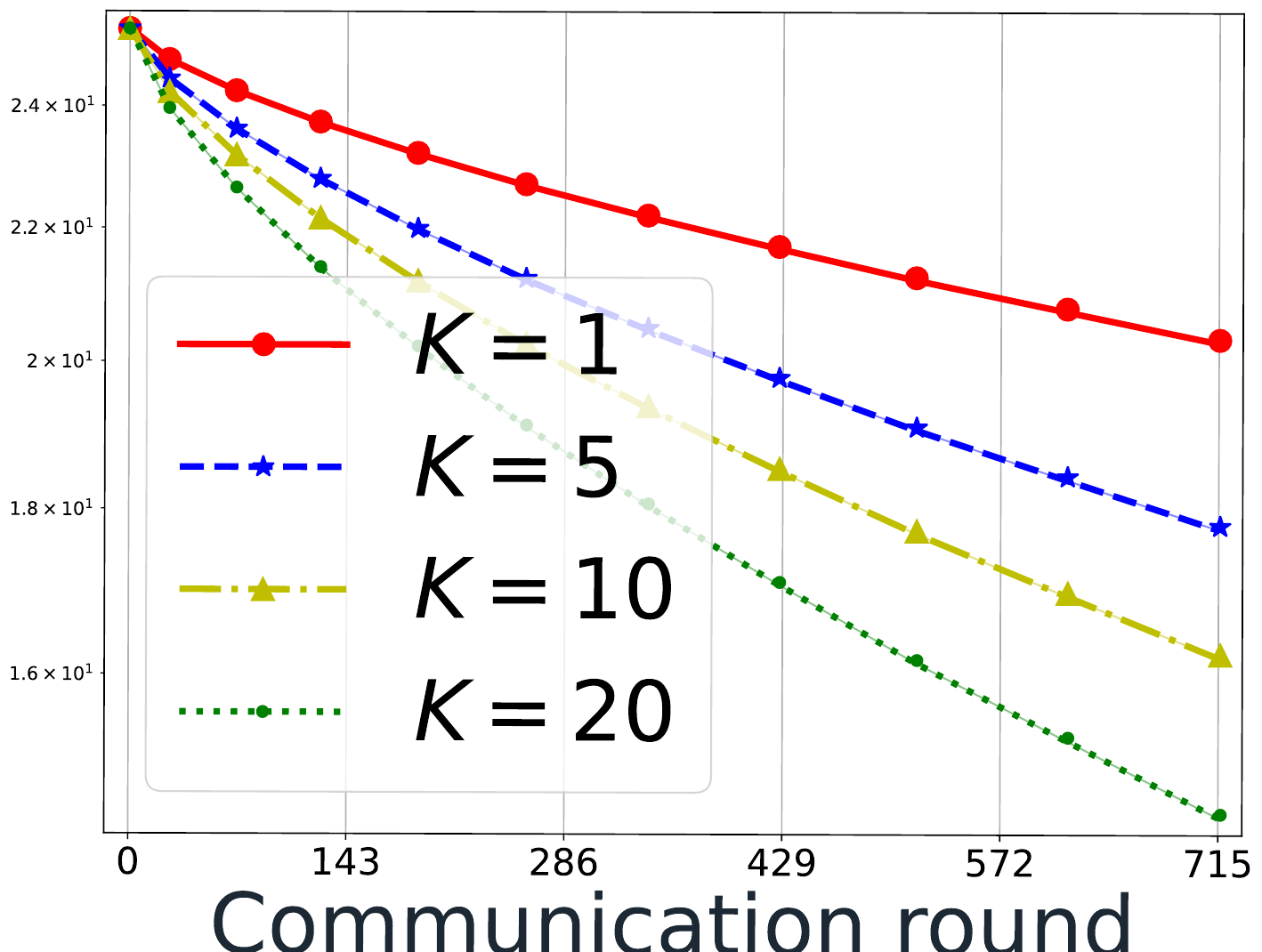}
\end{minipage}
&
\begin{minipage}{.3\columnwidth}
\includegraphics[width=1.0\columnwidth]{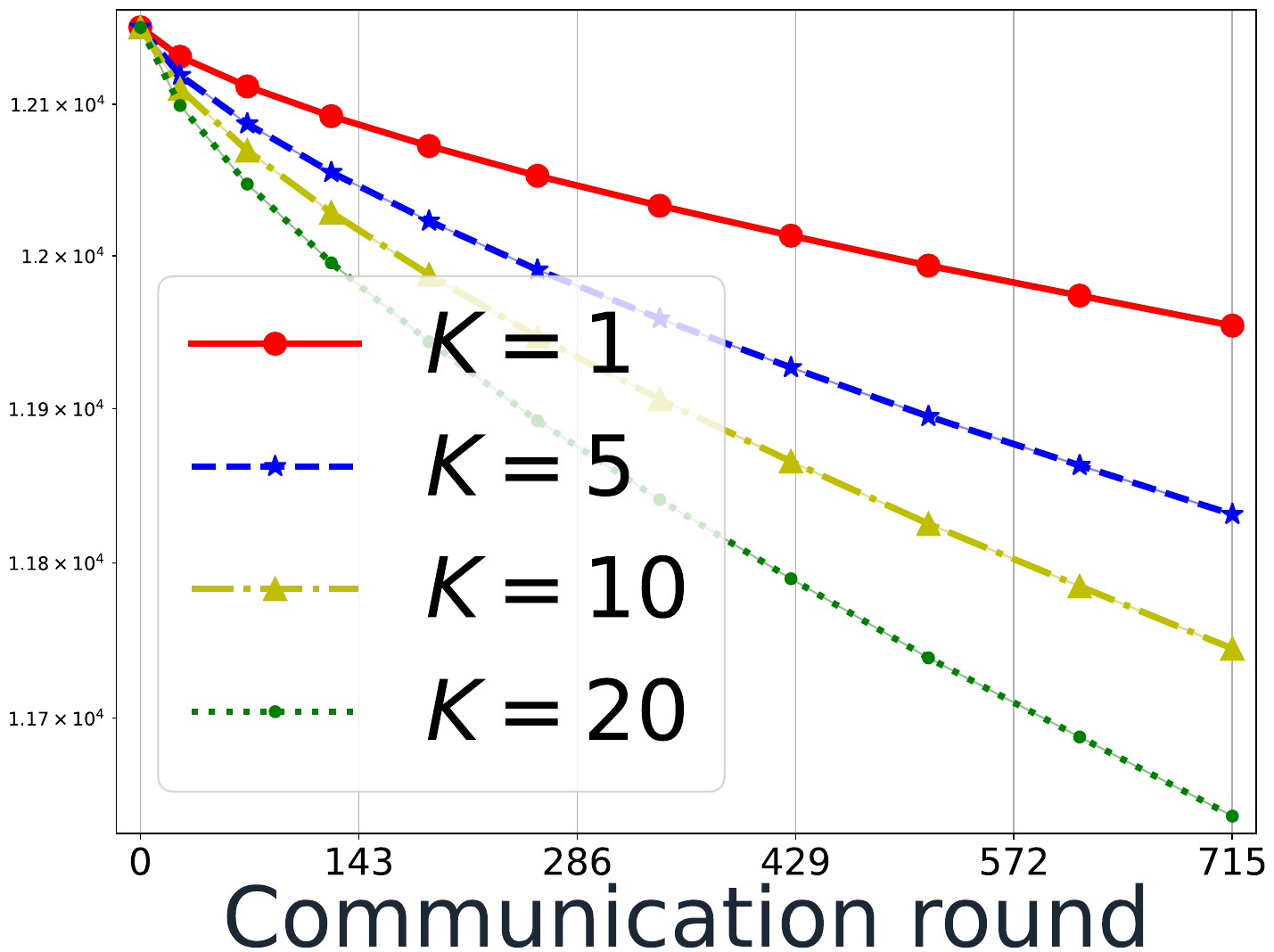}
\end{minipage}

\end{tabular}}

\vspace{.05in}

\captionof{figure}{Performance of Algorithm~\ref{algorithm:ncvx} for different local steps..}

\label{exp2:ncvx}

\vspace{-.1in}
\end{table}

\fyr{\subsection{Comparison with other FL and centralized schemes}}
{\bf{FL.}} \mje{We compare StR-FedAvg with the federated incremental subgradient method (FISM)~\cite{boontawee2025federated}, which applies FL only to the inner-level problem while keeping the outer-level centralized. We solve problem~\eqref{prob:l1} with outer objective $f(x):=\|x\|_1+0.5\|x\|^2$.} \yq{The results are shown in Figure~\ref{fig:FISM:COMP}.}

\noindent{\bf Insights.} \mje{StR-FedAvg outperforms FISM across all metrics. Unlike FISM, it exploits the strong convexity parameter $\mu_f$ in the stepsize and regularization updates, and updates the regularized outer-level gradient at every local iteration rather than once per communication round.}

\begin{figure}[t]
\centering

\setlength{\tabcolsep}{2pt}
\renewcommand{\arraystretch}{1.4}

\newcommand{\rowlabel}[1]{%
\raisebox{-.5\height}{\rotatebox[origin=c]{90}{\scriptsize #1}}}

\begin{tabular}{c|ccc}

\scriptsize Setting
&
\scriptsize $\mathbb{E}[f(\bar x_r)-f(\bar x_{r-1})]$
&
\scriptsize $\mathbb{E}[f(\bar x_r)]$
&
\scriptsize $\mathbb{E}[h(\bar x_r)]$
\\
\hline

\rowlabel{WikiMath}
&
\raisebox{-.5\height}{%
\includegraphics[
width=.3\columnwidth
]{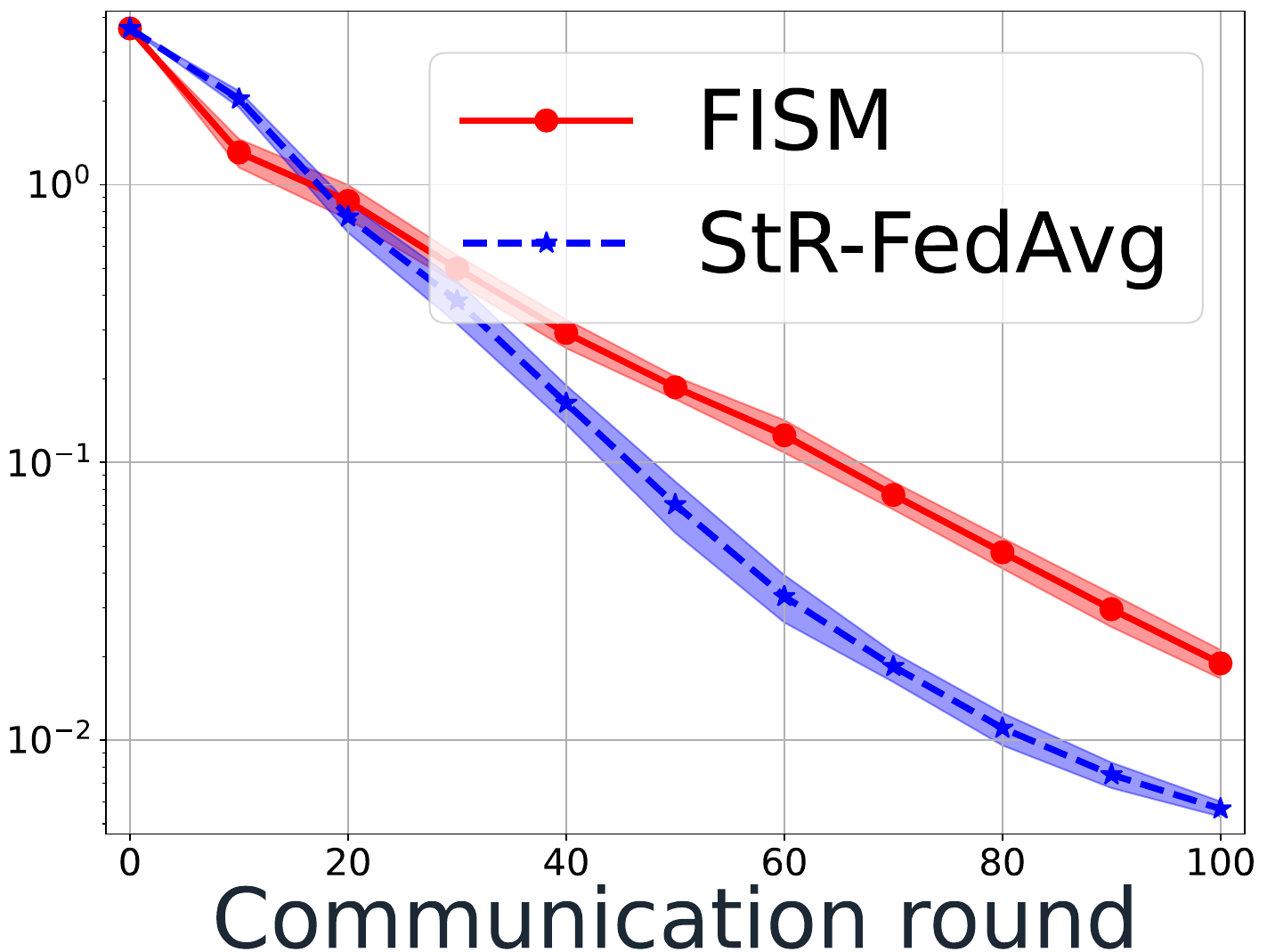}}
&
\raisebox{-.5\height}{%
\includegraphics[
width=.3\columnwidth
]{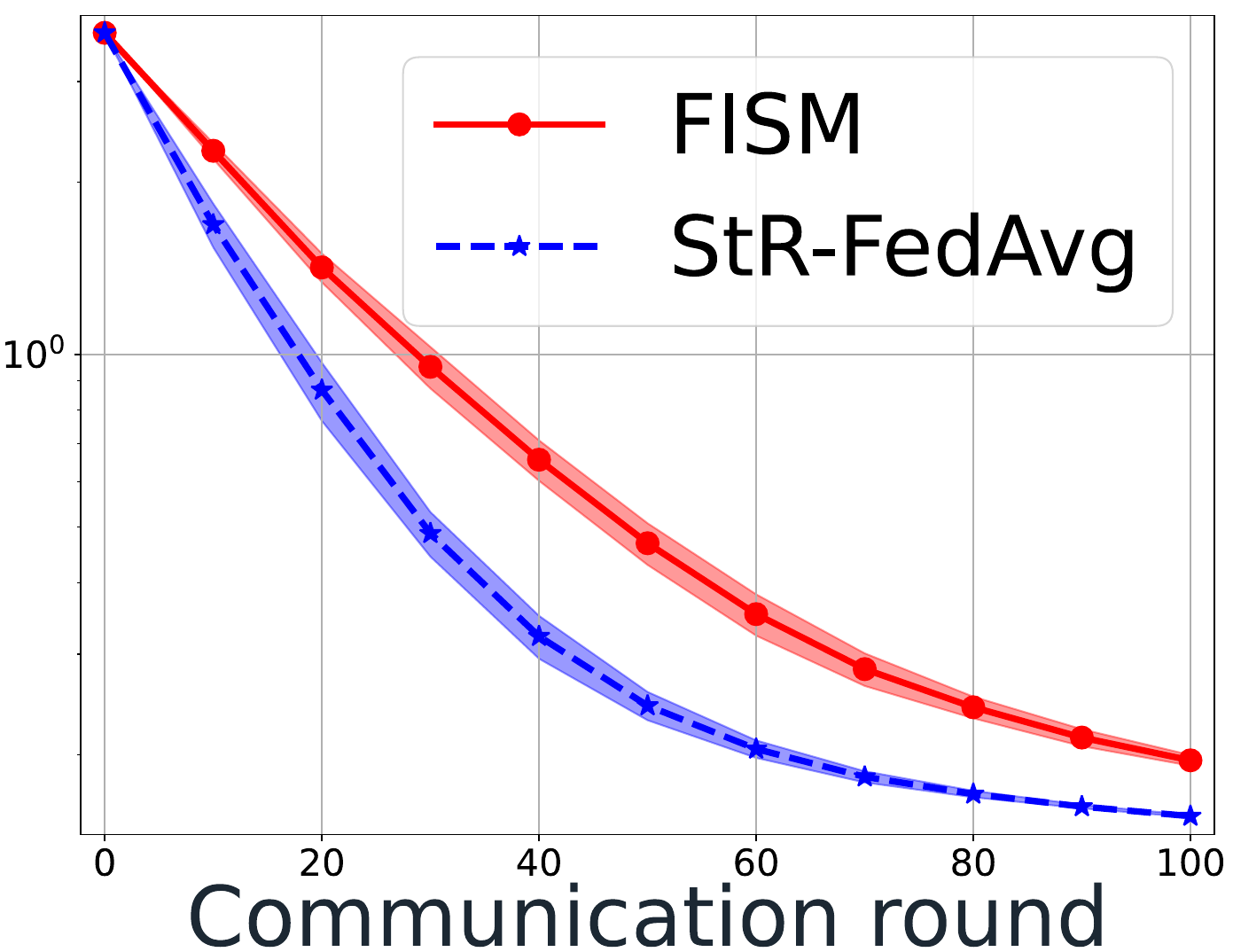}}
&
\raisebox{-.5\height}{%
\includegraphics[
width=.3\columnwidth
]{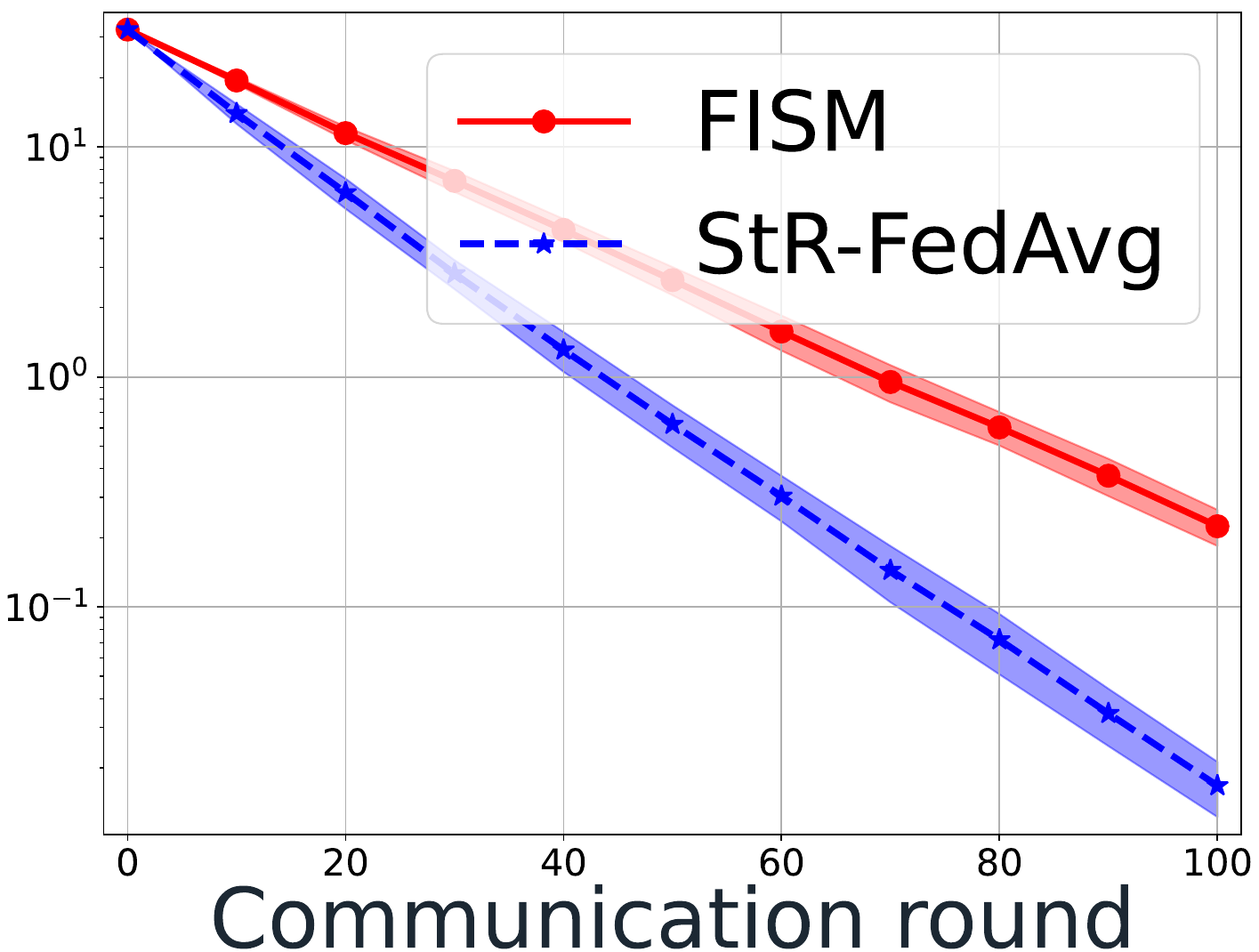}}
\\[3mm]

\end{tabular}

\vspace{-2mm}

\caption{
\yq{Comparison of Algorithm~\ref{algorithm:R-FedAvg and R-SCAFFOLD} and FISM~\cite{boontawee2025federated}.}
}

\vspace{-4mm}

\label{fig:FISM:COMP}

\end{figure}

\vspace{3mm}

\mje{
{\bf{\fyr{Centralized schemes.}}} Since no existing work studies federated bilevel optimization with distributed data at both levels, we compare the centralized version of our method with MNG~\cite{beck2014first}, BiG-SAM~\cite{sabach2017first}, CG-BiO~\cite{jiang2023conditional}, and a-IRG~\cite{kaushik2021method} for ill-posed inverse problems. In particular, we consider
\begin{align}
\min_x \quad & f(x) :=0.5\|x\|^2 + \|x\|_1 \\
\text{s.t.} \quad & x \in \arg\min_{\|y\|_1\le1} \ h(y) := 0.5\|{\bf A}y-{ b}\|^2,\notag
\end{align}
where ${\bf A}\in\mathbb{R}^{n\times n}$ and ${b}\in\mathbb{R}^n$ are generated from the ``Baart'' inverse problem~\cite{beck2014first} in the Regularization Toolbox.

\noindent{\bf Insight.} Although some methods progress faster initially, as the simulation budget increases, the proposed self-tuned regularization approach adapts and yields more stable improvement. As shown in Figure~\ref{fig: centralized comparison}, it eventually achieves the best upper-level objective value while attaining nearly the same inner-level performance as the competing methods.

\begin{table}[t]
\setlength{\tabcolsep}{0pt}
\centering{
\begin{tabular}{c || c  c  c}

{\footnotesize Setting\ \ } 
&
\multicolumn{3}{c}{
$\xrightarrow{\hspace*{2.6cm}\text{Run time}\hspace*{2.6cm}}$
}
\\
\hline\\

\rotatebox[origin=c]{90}{{\footnotesize {$f(x_k)$}}}
&
\begin{minipage}{.3\columnwidth}
\includegraphics[width=1.0\columnwidth]{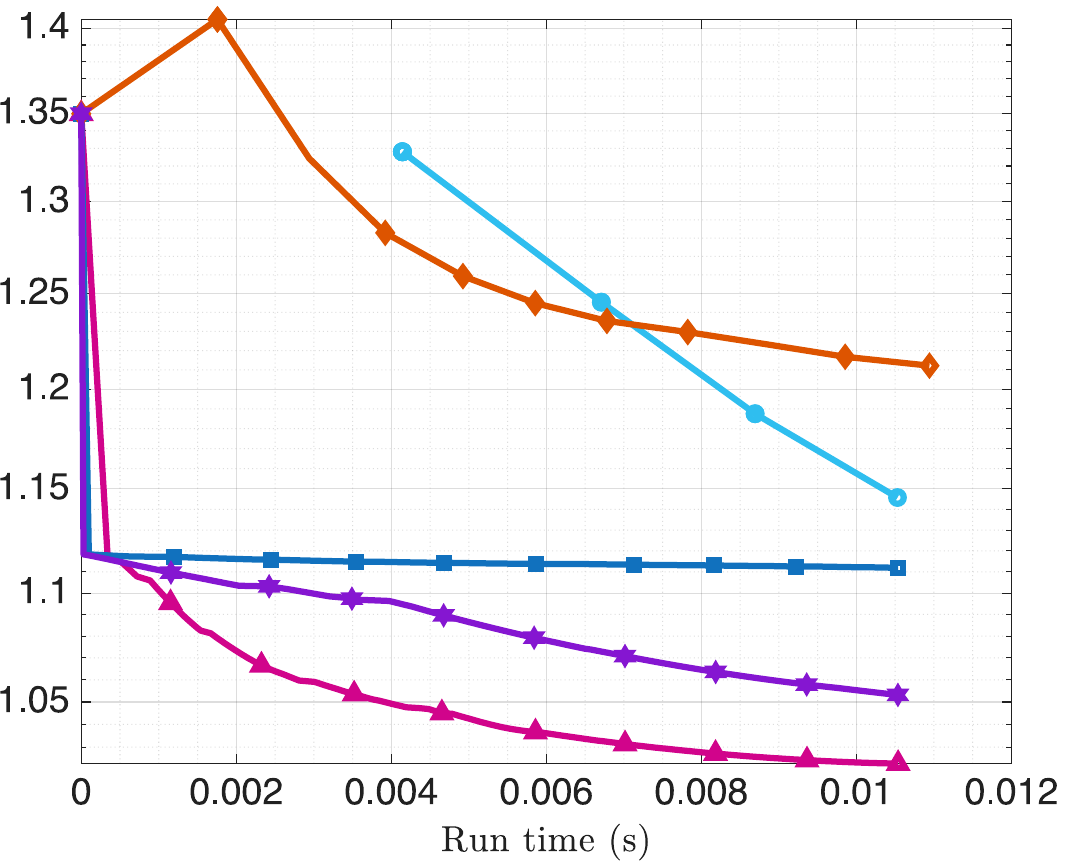}
\end{minipage}
&
\begin{minipage}{.3\columnwidth}
\includegraphics[width=1.0\columnwidth]{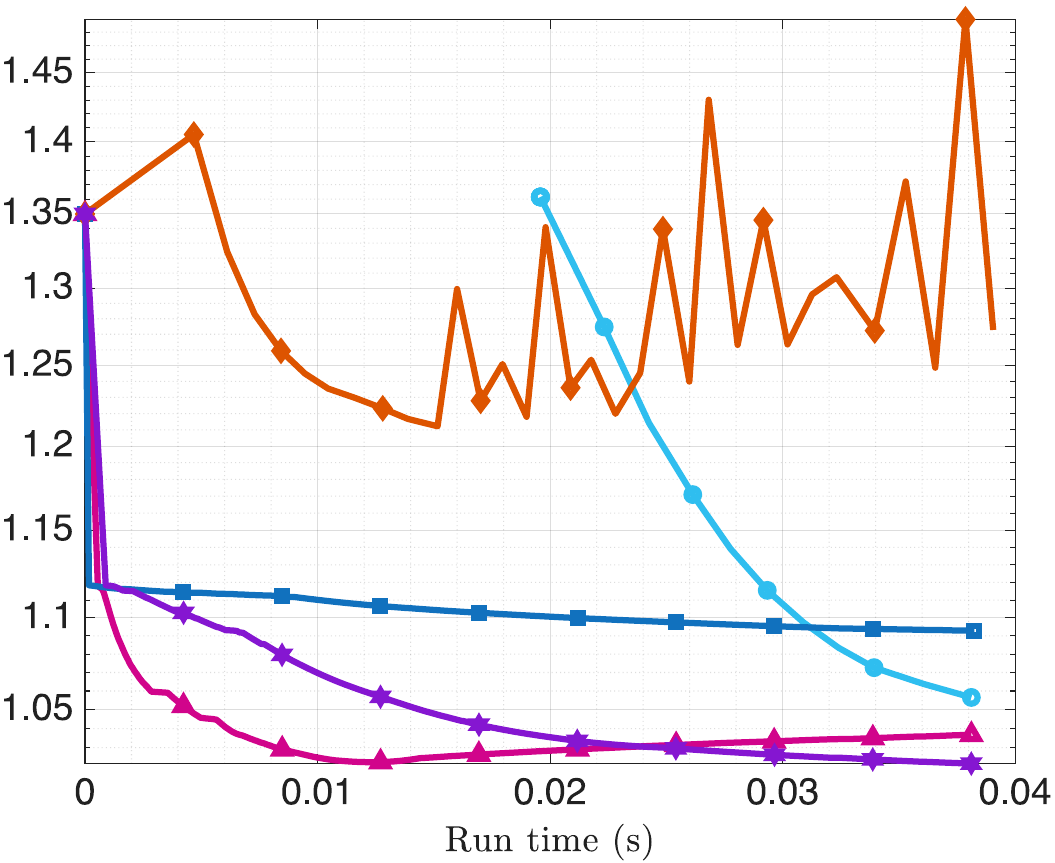}
\end{minipage}
&
\begin{minipage}{.3\columnwidth}
\includegraphics[width=1.0\columnwidth]{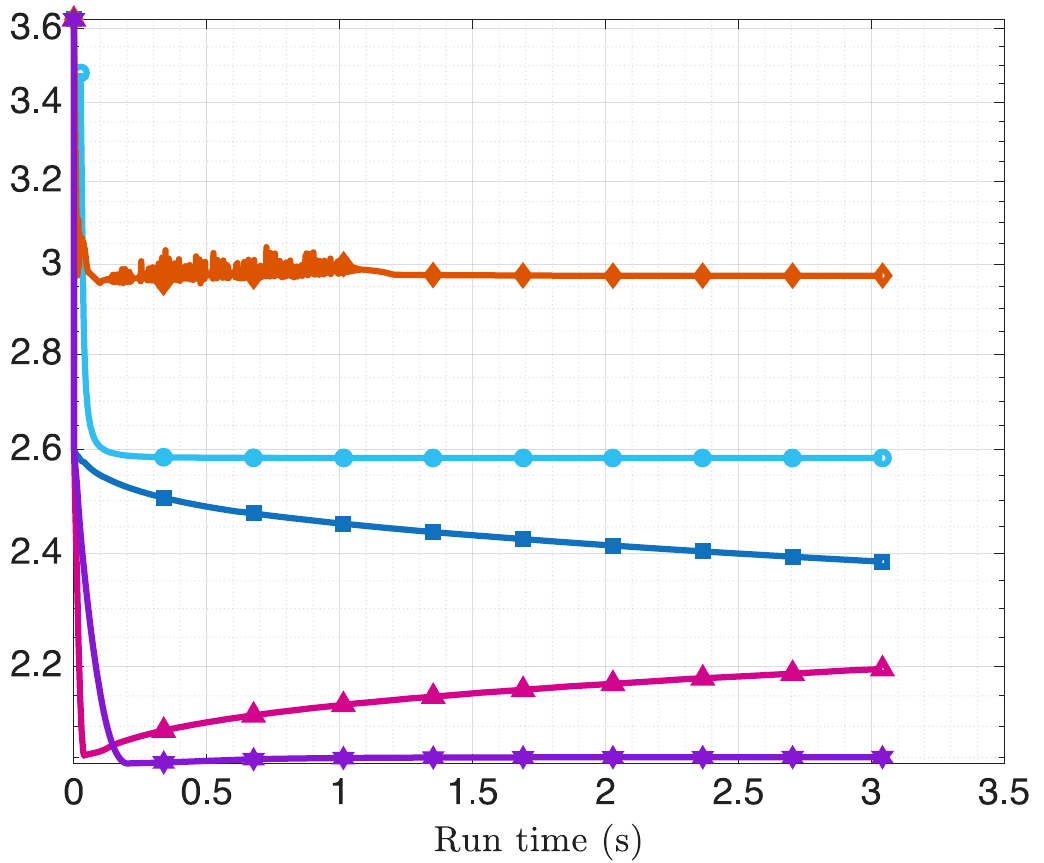}
\end{minipage}
\\

\hbox{}& & &\\
\hline\\

\rotatebox[origin=c]{90}{{\footnotesize {$h(x_k)$}}}
&
\begin{minipage}{.3\columnwidth}
\includegraphics[width=1.0\columnwidth]{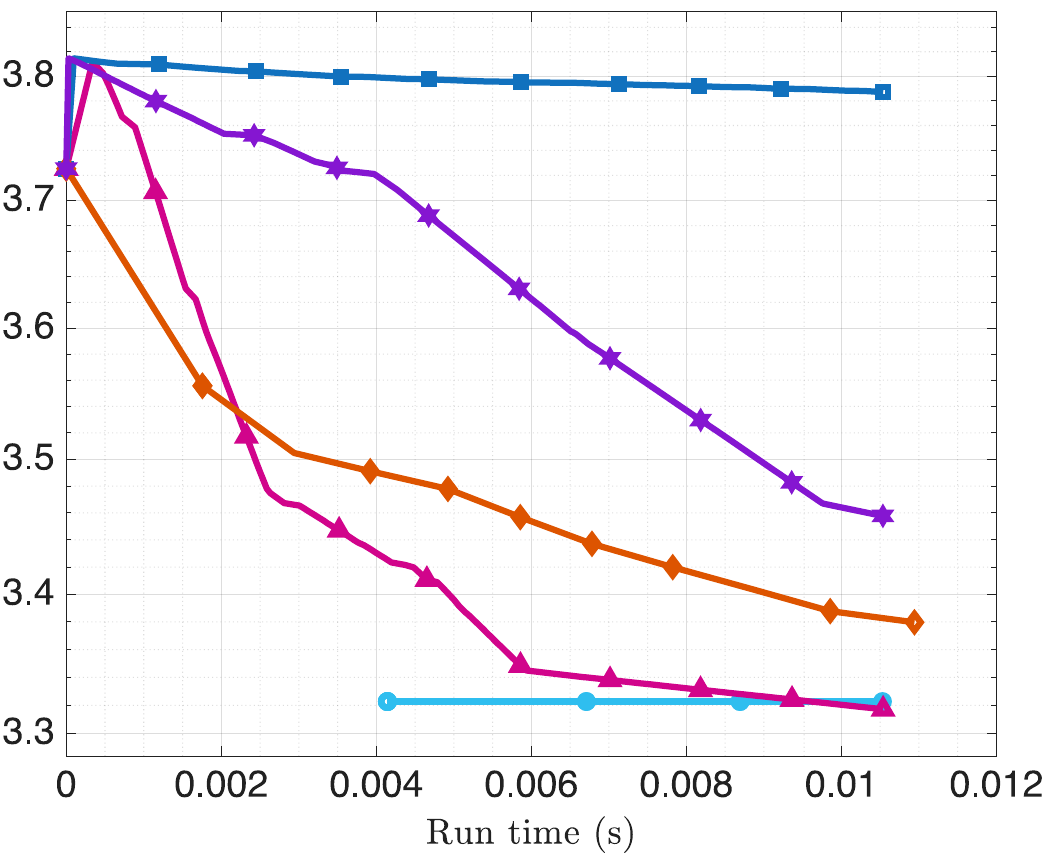}
\end{minipage}
&
\begin{minipage}{.3\columnwidth}
\includegraphics[width=1.0\columnwidth]{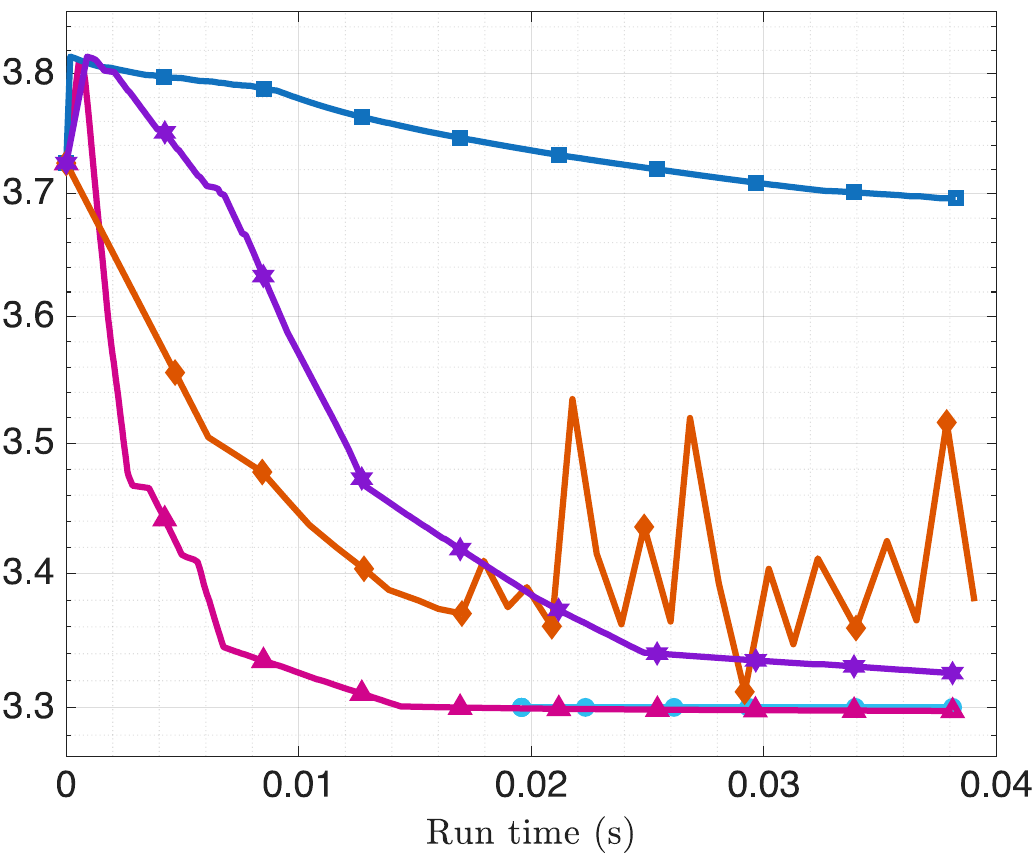}
\end{minipage}
&
\begin{minipage}{.3\columnwidth}
\includegraphics[width=1.0\columnwidth]{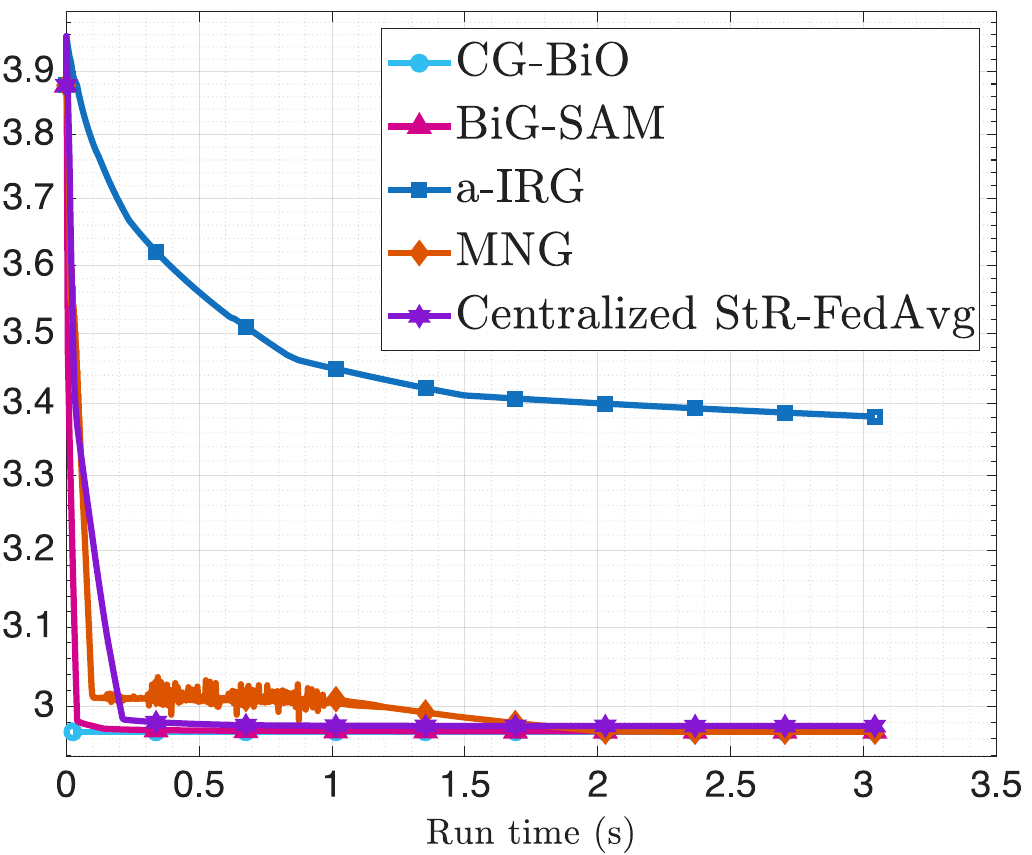}
\end{minipage}

\end{tabular}}

\vspace{.05in}

\captionof{figure}{\fyr{Comparison with centralized schemes.}}

\label{fig: centralized comparison}

\vspace{-.1in}
\end{table}
}

\section{Concluding Remarks}
In this paper, we propose a \RRRme{self-tuned} regularized scheme for solving \RRRme{the optimal solution selection problem in federated learning}. This problem finds its roots in the study of ill-posed optimization and is motivated by \me{over-parameterized} learning, among others. We enable the federated averaging method to address this class of bilevel problems via a \RRRme{self-tuned} regularized variant, called \RRRme{StR-FedAvg}. To showcase the significance of our contribution, we obtain explicit communication complexity bounds for \far{this method for} addressing the \me{simple} bilevel problem when the inner-level loss function is convex and the outer-level loss function is convex or strongly convex. \me{For the case where the \fyy{outer-level} loss function is nonconvex, we develop a two-loop federated scheme using iterative regularization and provide explicit error bounds.}

\bibliographystyle{siam}
\bibliography{ref}
\section{Appendix}

\Rme{
{\bf{A\quad Proof of Lemma~\ref{assumption:main4}}}
\begin{proof} 
Invoking the convexity and $L_f$-smoothness of each $f_i$, we obtain
\begin{align*}
\|\nabla f_i(x)-\nabla f_i(y)\|^2
~\le~ 2L_f\Big(f_i(x)-f_i(y)-\langle \nabla f_i(y),\,x-y\rangle\Big).
\end{align*}
Setting $y=x_i^*$ for some $x_i^*\in\arg\min_x f_i(x)$ and noting that $\nabla f_i(x_i^*)=0$, we obtain
\begin{align*}
\|\nabla f_i(x)\|^2
~\le~ 2L_f\Big(f_i(x)-f_i(x_i^*)\Big).
\end{align*}
Averaging both sides over $i=1,\ldots,N$ and adding and subtracting $f^*$ on the right-hand side, we obtain
\begin{align*}
\tfrac{1}{N}\textstyle\sum_{i=1}^N\|\nabla f_i(x)\|^2\le 2L_f\left(f^*-\tfrac{1}{N}\textstyle\sum_{i=1}^Nf_i(x_i^*)\right)+2L_f\left(f(x)-f^*\right).
\end{align*}
We note that $f^*\ge \frac{1}{N}\sum_{i=1}^N f_i(x_i^*)$. Thus, setting $G_f\triangleq \sqrt{\frac{2L_f}{N}\left(f^*-\frac{1}{N}\sum_{i=1}^N f_i(x_i^*)\right)}$ and $B_f\triangleq 1$, we obtain the desired result. The proof for the function $h$ follows similarly.
\end{proof}}

{\bf B \quad Proof of Lemma~\ref{remark:bound in terms of the obj}}
\begin{proof}
Let us denote $ f_{\eta,i}(x) : = h_i(x) +\eta f_i(x)$. We have
\vspace{-7pt}
\begin{align*}
\tfrac{1}{N}\textstyle{\sum_{i=1}^N}\|\nabla f_{\eta,i}(x)\|^2&=\tfrac{1}{N}\textstyle{\sum_{i=1}^N}\|\eta \nabla f_i(x)+\nabla h_i(x)\|^2 \le \tfrac{2\eta^2}{N}\textstyle{\sum_{i=1}^N}\| \nabla f_i(x)\|^2\\
&+\tfrac{2}{N}\textstyle{\sum_{i=1}^N}\|\nabla h_i(x)\|^2.
\end{align*}
\vspace{-10pt}
Let $f_{\min} := \inf_{x} f(x)$ and $h_{\min} := \inf_{x} h(x)$. Invoking Lemma~\ref{assumption:main4}, we get
\begin{align*}
&\tfrac{1}{N}\textstyle{\sum_{i=1}^N}\|\nabla f_{\eta,i}(x)\|^2\le 2\eta^2  G_f^2+4\eta^2L_f B_f^2(f(x)-f_{\min})+2 G_h^2+4L_h B_h^2(h(x)-h_{\min})\\
&=2\eta^2  G_f^2+4\eta^2L_f B_f^2(f(x)-f_{\min}+f^*-f^*)+2 G_h^2+4L_h B_h^2(h(x)-h_{\min}+h^*-h^*),
\end{align*}
where $f^*$ and $h^*$ denote the optimal objective values in \eqref{problem: main problem}. We obtain
\begin{align*}
\tfrac{1}{N}\textstyle{\sum_{i=1}^N}\|\nabla f_{\eta,i}(x)\|^2\le& 2\eta^2  G_f^2+4\eta L_f B_f^2(\eta f(x)-\eta f^*)+2 G_h^2+4L_h B_h^2(h(x)-h^*)\\
&+4\eta^2L_f B_f^2(f^*-f_{\min})+4L_h B_h^2(h^*-h_{\min}).
\end{align*}
Recalling the definition of $B$ and $G$,  we obtain
\vspace{-6pt}
\begin{align*}
\tfrac{1}{N}\textstyle{\sum_{i=1}^N}\|\nabla f_{\eta,i}(x)\|^2&\le G^2+B^2\left((\eta f(x)+h(x))-(\eta f(x^*)+h(x^*))\right)\\
&=G^2+B^2\left(f_\eta(x)-(h^*+\eta f^*)\right).
\end{align*}
\vspace{-6pt}
Invoking the fact that $f_\eta^* = f_\eta(x_\eta^*)\le h^*+\eta f^*$, we obtain the result.
\end{proof}
~\\

{\bf{C\quad Proof of Lemma~\ref{remark:bound in terms of the obj-ncvx}}}
\begin{proof}
The proof follows similarly to that of Lemma~\ref{remark:bound in terms of the obj} \far{by replacing  $f$ with $g$ and updating} $G_f$ and $B_f$ to $G_g$ and $B_g$, respectively, while setting \far{$L_f=1=\mu_f=1$.}
\end{proof}
{\bf D \quad Proof of Proposition~\ref{Proposition:Proposition1}}
\begin{proof}

Invoking Lemmas~\ref{lemma:first bound for x-x*} and \ref{lemma: bound for varepsilon-k}, we obtain
\begin{align*}
&\mathbb{E}\left[\|\RRme{\bar x}^r-x_\eta^*\|^2\right]\le(1-0.5{\eta \mu_f \tilde \gamma}{})\mathbb{E}\left[\|\RRme{\bar x}^{r-1}-x_\eta^*\|^2\right]-\tfrac{\tilde \gamma}{3}\mathbb{E}\left[(f_\eta(\RRme{\bar x}^{r-1})-f_\eta(x_\eta^*))\right]\\
&+\tfrac{\tilde \gamma^2(2\eta^2\sigma_f^2+2\sigma_h^2)}{KS}+\left(1-\tfrac{S}{N}\right)\tfrac{4\tilde \gamma^2}{S}G^2+\tfrac{9\tilde \gamma^3(L_h+\eta L_f)(2\eta^2\sigma_f^2+2\sigma_h^2)}{K\gamma_g^2}+\tfrac{18(L_h+\eta L_f)\tilde \gamma^3G^2}{{\gamma_g^2}}\\
&=(1-0.5{\eta \mu_f \tilde \gamma}{})\mathbb{E}\left[\|\RRme{\bar x}^{r-1}-x_\eta^*\|^2\right]-\tfrac{\tilde \gamma}{3}\mathbb{E}\left[(f_\eta(\RRme{\bar x}^{r-1})-f_\eta(x_\eta^*))\right]\\
&+\tilde \gamma^2\left(\tfrac{(2\eta^2\sigma_f^2+2\sigma_h^2)}{KS}+\left(1-\tfrac{S}{N}\right)\tfrac{4}{S}G^2+\tfrac{9\tilde \gamma(L_h+\eta L_f)(2\eta^2\sigma_f^2+2\sigma_h^2)}{K\gamma_g^2}+\tfrac{18(L_h+\eta L_f)\tilde \gamma G^2}{{\gamma_g^2}}\right).
\end{align*}
Dividing both sides by $(1-0.5{\eta \mu_f \tilde \gamma}{})^{r}$, for any $r\geq 1$, we obtain
\begin{align*}
&\tfrac{\mathbb{E}\left[\|\RRme{\bar x}^r-x_\eta^*\|^2\right]}{(1-0.5{\eta \mu_f \tilde \gamma}{})^{r}}\le\tfrac{\mathbb{E}\left[\|\RRme{\bar x}^{r-1}-x_\eta^*\|^2\right]}{(1-0.5{\eta \mu_f \tilde \gamma}{})^{r-1}}-\tfrac{\tilde \gamma}{3(1-0.5{\eta \mu_f \tilde \gamma}{})^{r}}\mathbb{E}\left[(f_\eta(\RRme{\bar x}^{r-1})-f_\eta(x_\eta^*))\right]\\
&+\tfrac{\tilde \gamma^2}{(1-0.5{\eta \mu_f \tilde \gamma}{})^{r}}\left(\tfrac{(2\eta^2\sigma_f^2+2\sigma_h^2)}{KS}+\left(1-\tfrac{S}{N}\right)\tfrac{4}{S}G^2+\tfrac{9\tilde \gamma(L_h+\eta L_f)(2\eta^2\sigma_f^2+2\sigma_h^2)}{K\gamma_g^2}+\tfrac{18(L_h+\eta L_f)\tilde \gamma G^2}{{\gamma_g^2}}\right).
\end{align*}
Summing both sides over $r=1,\ldots,R$, we have
\begin{align*}
&\textstyle{\sum_{r=1}^R}\tfrac{\mathbb{E}\left[\|\RRme{\bar x}^r-x_\eta^*\|^2\right]}{(1-0.5{\eta \mu_f \tilde \gamma}{})^{r}}\le\sum_{r=1}^R\tfrac{\mathbb{E}\left[\|\RRme{\bar x}^{r-1}-x_\eta^*\|^2\right]}{(1-0.5{\eta \mu_f \tilde \gamma}{})^{r-1}}-\sum_{r=1}^R\tfrac{\tilde \gamma}{3(1-0.5{\eta \mu_f \tilde \gamma}{})^{r}}\mathbb{E}\left[(f_\eta(\RRme{\bar x}^{r-1})-f_\eta(x_\eta^*))\right]\\
&+\textstyle{\sum_{r=1}^R}\tfrac{\tilde \gamma^2}{(1-0.5{\eta \mu_f \tilde \gamma})^{r}}\left(\tfrac{(2\eta^2\sigma_f^2+2\sigma_h^2)}{KS}+\tfrac{4G^2\left(N-S\right)}{NS}+\tfrac{9\tilde \gamma(L_h+\eta L_f)(2\eta^2\sigma_f^2+2\sigma_h^2)}{K\gamma_g^2}+\tfrac{18(L_h+\eta L_f)\tilde \gamma G^2}{{\gamma_g^2}}\right).
\end{align*}
Rearranging the terms and multiplying both sides by $(1-0.5{\eta \mu_f \tilde \gamma}{})^{R}$, we obtain
\begin{align*}
&\mathbb{E}\left[\|\RRme{\bar x}^R-x_\eta^*\|^2\right]\le(1-0.5{\eta \mu_f \tilde \gamma}{})^{R}\mathbb{E}\left[\|\RRme{\bar x}^{0}-x_\eta^*\|^2\right]-(1-0.5{\eta \mu_f \tilde \gamma}{})^{R}\textstyle{\sum_{r=1}^R}\tfrac{\tilde \gamma}{3(1-0.5{\eta \mu_f \tilde \gamma}{})^{r}}\\&\times\mathbb{E}\left[(f_\eta(\RRme{\bar x}^{r-1})-f_\eta(x_\eta^*))\right]+(1-0.5{\eta \mu_f \tilde \gamma}{})^{R}\textstyle{\sum_{r=1}^R}\tfrac{\tilde \gamma^2}{(1-0.5{\eta \mu_f \tilde \gamma}{})^{r}}\\
&\times\left(\tfrac{(2\eta^2\sigma_f^2+2\sigma_h^2)}{KS}+\left(1-\tfrac{S}{N}\right)\tfrac{4}{S}G^2+\tfrac{9\tilde \gamma(L_h+\eta L_f)(2\eta^2\sigma_f^2+2\sigma_h^2)}{K\gamma_g^2}+\tfrac{18(L_h+\eta L_f)\tilde \gamma G^2}{{\gamma_g^2}}\right).
\end{align*}
Invoking Jensen's inequality and the definition of $\theta$, we obtain
\begin{align*}
&\mathbb{E}\left[\|\RRme{\bar x}^R-x_\eta^*\|^2\right]\le\left(\tfrac{1}{\theta}\right)^{R}\mathbb{E}\left[\|\RRme{\bar x}^{0}-x_\eta^*\|^2\right]-\tilde \gamma\left(\tfrac{1}{\theta}\right)^{R}\textstyle{\sum_{r=1}^{R}}\theta^{r}\mathbb{E}\left[(f_\eta(\RRme{\bar x}^{r-1})-f_\eta(x_\eta^*))\right]+\\
&\tilde \gamma^2\left(\tfrac{1}{\theta}\right)^{R}\textstyle{\sum_{r=1}^{R}}\theta^{r}\left(\tfrac{(2\eta^2\sigma_f^2+2\sigma_h^2)}{KS}+\left(1-\tfrac{S}{N}\right)\tfrac{4}{S}G^2+\tfrac{9\tilde \gamma(L_h+\eta L_f)(2\eta^2\sigma_f^2+2\sigma_h^2)}{K\gamma_g^2}+\tfrac{18(L_h+\eta L_f)\tilde \gamma G^2}{{\gamma_g^2}}\right)\\
&=\left(\tfrac{1}{\theta}\right)^{R}\mathbb{E}\left[\|\RRme{\bar x}^{0}-x_\eta^*\|^2\right]-\tilde \gamma\left(\tfrac{1}{\theta}\right)^{R}\textstyle{\sum_{\hat r=1}^{R}}\theta^{\hat r}\textstyle{\sum_{r=1}^{R}}\tfrac{\theta^{r}}{\sum_{\hat r=1}^{R-1}\theta^{\hat r}}\mathbb{E}\left[(f_\eta(\RRme{\bar x}^{r-1})-f_\eta(x_\eta^*))\right]+\\
&\tilde \gamma^2\left(\tfrac{1}{\theta}\right)^{R}\textstyle{\sum_{r=1}^{R}}\theta^{r}\left(\tfrac{(2\eta^2\sigma_f^2+2\sigma_h^2)}{KS}+\left(1-\tfrac{S}{N}\right)\tfrac{4}{S}G^2+\tfrac{9\tilde \gamma(L_h+\eta L_f)(2\eta^2\sigma_f^2+2\sigma_h^2)}{K\gamma_g^2}+\tfrac{18(L_h+\eta L_f)\tilde \gamma G^2}{{\gamma_g^2}}\right)\\
&\le\left(\tfrac{1}{\theta}\right)^{R}\mathbb{E}\left[\|\RRme{\bar x}^{0}-x_\eta^*\|^2\right]-\tilde \gamma\left(\tfrac{1}{\theta}\right)^{R}\textstyle{\sum_{\hat r=1}^{R}}\theta^{\hat r}\mathbb{E}\left[ (f_\eta(\bar x^R)-f_\eta(x_\eta^*))\right]+\\
&\tilde \gamma^2\left(\tfrac{1}{\theta}\right)^{R}\textstyle{\sum_{r=1}^{R}}\theta^{r}\left(\tfrac{(2\eta^2\sigma_f^2+2\sigma_h^2)}{KS}+\left(1-\tfrac{S}{N}\right)\tfrac{4}{S}G^2+\tfrac{9\tilde \gamma(L_h+\eta L_f)(2\eta^2\sigma_f^2+2\sigma_h^2)}{K\gamma_g^2}+\tfrac{18(L_h+\eta L_f)\tilde \gamma G^2}{{\gamma_g^2}}\right).
\end{align*}
Rearranging the terms, we obtain 
\begin{align}
&\mathbb{E}\left[\|\RRme{\bar x}^R-x_\eta^*\|^2\right]+\tilde \gamma\left(\tfrac{1}{\theta}\right)^{R}\textstyle{\sum_{\hat r=1}^{R}}\theta^{\hat r}\mathbb{E}\left[ (f_\eta(\bar x^R)-f_\eta(x_\eta^*))\right]\le\left(\tfrac{1}{\theta}\right)^{R}\mathbb{E}\left[\|\RRme{\bar x}^{0}-x_\eta^*\|^2\right]\label{eq:bound for all terms}+\\
&\tilde \gamma^2\left(\tfrac{1}{\theta}\right)^{R}\textstyle{\sum_{r=1}^{R}}\theta^{r}\left(\tfrac{(2\eta^2\sigma_f^2+2\sigma_h^2)}{KS}+\left(1-\tfrac{S}{N}\right)\tfrac{4}{S}G^2\notag+\tfrac{9\tilde \gamma(L_h+\eta L_f)(2\eta^2\sigma_f^2+2\sigma_h^2)}{K\gamma_g^2}+\tfrac{18(L_h+\eta L_f)\tilde \gamma G^2}{{\gamma_g^2}}\right).\notag
\end{align}Dropping \far{the term} $\mathbb{E}\left[\|\RRme{\bar x}^R-x_\eta^*\|^2\right]$ and rearranging the terms, we obtain the result. 
\end{proof}
{\bf E \quad Proof of Proposition~\ref{Proposition:Proposition3}}
\begin{proof}
Invoking Lemmas~\ref{lemma: bound for varepsilon-k} and \ref{lemma:first bound for x-x*-f convex}, we obtain
\begin{align*}
&\mathbb{E}\left[\|\RRme{\bar x}^r-x_\eta^*\|^2\right]\le\mathbb{E}\left[\|\RRme{\bar x}^{r-1}-x_\eta^*\|^2\right]-\tfrac{\tilde \gamma}{3}\mathbb{E}\left[(f_\eta(\RRme{\bar x}^{r-1})-f_\eta(x_\eta^*))\right]\\
&+\tfrac{\tilde \gamma^2(2\eta^2\sigma_f^2+2\sigma_h^2)}{KS}+\left(1-\tfrac{S}{N}\right)\tfrac{4\tilde \gamma^2}{S}G^2+\tfrac{9\tilde \gamma^3(L_h+\eta L_f)(2\eta^2\sigma_f^2+2\sigma_h^2)}{K\gamma_g^2}+\tfrac{18(L_h+\eta L_f)\tilde \gamma^3G^2}{{\gamma_g^2}}\\
&=\mathbb{E}\left[\|\RRme{\bar x}^{r-1}-x_\eta^*\|^2\right]-\tfrac{\tilde \gamma}{3}\mathbb{E}\left[(f_\eta(\RRme{\bar x}^{r-1})-f_\eta(x_\eta^*))\right]+\tilde \gamma^2\left(\tfrac{(2\eta^2\sigma_f^2+2\sigma_h^2)}{KS}+\left(1-\tfrac{S}{N}\right)\tfrac{4}{S}G^2\right.\\
&\left.+\tfrac{9\tilde \gamma(L_h+\eta L_f)(2\eta^2\sigma_f^2+2\sigma_h^2)}{K\gamma_g^2}+\tfrac{18(L_h+\eta L_f)\tilde \gamma G^2}{{\gamma_g^2}}\right).
\end{align*}
Rearranging the terms, we obtain
\begin{align*}
&\mathbb{E}\left[(f_\eta(\RRme{\bar x}^{r-1})-f_\eta(x_\eta^*))\right]\le\tfrac{3}{\tilde \gamma}\mathbb{E}\left[\|\RRme{\bar x}^{r-1}-x_\eta^*\|^2\right]-\tfrac{3}{\tilde \gamma}\mathbb{E}\left[\|\RRme{\bar x}^r-x_\eta^*\|^2\right]\\
&+3\tilde \gamma\left(\tfrac{(2\eta^2\sigma_f^2+2\sigma_h^2)}{KS}+\left(1-\tfrac{S}{N}\right)\tfrac{4}{S}G^2+\tfrac{9\tilde \gamma(L_h+\eta L_f)(2\eta^2\sigma_f^2+2\sigma_h^2)}{K\gamma_g^2}+\tfrac{18(L_h+\eta L_f)\tilde \gamma G^2}{{\gamma_g^2}}\right).
\end{align*}

Summing both sides over $r$, invoking Jensen's inequality and using the definition of $\bar{x}^R$, we obtain
\begin{align*}
& \mathbb{E}\left[(f_\eta(\bar{x}^{R})-f_\eta(x_\eta^*))\right]\le\tfrac{3}{\tilde \gamma}\tfrac{1}{R}\textstyle{\sum_{r=1}^{R}}\mathbb{E}\left[\|\RRme{\bar x}^{r-1}-x_\eta^*\|^2\right]-\tfrac{3}{\tilde \gamma}\tfrac{1}{R}\sum_{r=1}^{R}\mathbb{E}\left[\|\RRme{\bar x}^r-x_\eta^*\|^2\right]\\
&+3\tilde \gamma\left(\tfrac{(2\eta^2\sigma_f^2+2\sigma_h^2)}{KS}+\left(1-\tfrac{S}{N}\right)\tfrac{4}{S}G^2+\tfrac{9\tilde \gamma(L_h+\eta L_f)(2\eta^2\sigma_f^2+2\sigma_h^2)}{K\gamma_g^2}+\tfrac{18(L_h+\eta L_f)\tilde \gamma G^2}{{\gamma_g^2}}\right).
\end{align*}

By expanding the summations and simplifying the terms, we obtain the final result.
\end{proof}
{\bf F \quad Proof of Proposition~\ref{Prop:Prop9}}
\begin{proof}\noindent [Case i] Invoking [Case ii] of Proposition~\ref{Prop:Prop2}, we obtain
\begin{align*}
&\Rme{\mathbb{E}\left[\|{\RRme{\bar x}}^{t}_{\eta_t}-x^{{t}*}\|^2\mid\mathcal{F}_t\right]}\leq \tfrac{2}{\Rme{ \eta_t\tilde \gamma_t}\sum_{r=1}^{R_t}\theta^{r}}\mathbb{E}\left[\|\RRme{\bar x}^{0}-x^{{t}*}\|^2\mid\mathcal{F}_t\right]\\
&+\Rme{\tfrac{2\tilde \gamma_t}{\eta_t}}\left(\tfrac{(2\eta^2\sigma_f^2+2\sigma_h^2)}{KS}+\left(1-\tfrac{S}{N}\right)\tfrac{4}{S}G_{ncvx}^2\right)+\RRRme{\tfrac{2\tilde \gamma_t^2}{\eta_t}}\left(\tfrac{9(L_h+\eta)(2\eta^2\sigma_f^2+2\sigma_h^2)}{K\Rme{\gamma_{g_t}^2}}+\tfrac{18(L_h+\eta )G_{ncvx}^2}{{\Rme{\gamma_{g_t}^2}}}\right)\\
&+{2\|\nabla f(x^{t*})\|}\left(\tfrac{1}{\alpha}\Big( \tfrac{1}{ \Rme{\tilde \gamma_t}\sum_{r=1}^{R_t}\theta^{r}}\mathbb{E}\left[\|\RRme{\bar x}^{0}-x^{t*}\|^2\mid\mathcal{F}_t\right]+\Rme{{\tilde \gamma_t}}\left(\tfrac{(2\eta^2\sigma_f^2+2\sigma_h^2)}{KS}+\left(1-\tfrac{S}{N}\right)\tfrac{4}{S}G_{ncvx}^2\right)\right.\\
&\left.+\Rme{{\tilde \gamma_t^2}}\left(\tfrac{9(L_h+\eta )(2\eta^2\sigma_f^2+2\sigma_h^2)}{K\Rme{\gamma_{g_t}^2}}+\tfrac{18(L_h+\eta )G_{ncvx}^2}{{\Rme{\gamma_{g_t}^2}}}\right)+2\Rme{\eta_t} M\Big)\right)^\frac{1}{\kappa}. 
\end{align*}
Invoking Theorem~\ref{thm:thm2}, for \Rme{$\tilde \gamma_t:= \tfrac{1}{R_t^a}$ and $\eta_t:= \tfrac{p\RRme{\ln}(R_t)}{R_t^b}$}, for some $0\le b<a\le1$, we obtain the result.

\noindent [Case ii] Based on [Case iii] of Proposition~\ref{Prop:Prop2}, we obtain
   
\begin{align*}
&\Rme{\mathbb{E}\left[\|{\RRme{\bar x}}^{t}_{\eta_t}-x^{{t}*}\|^2\mid\mathcal{F}_t\right]}\leq \tfrac{2}{\Rme{ \eta_t\tilde \gamma_t}\sum_{r=1}^{R_t}\theta^{r}}\mathbb{E}\left[\|\RRme{\bar x}^{0}-x^{{t}*}\|^2\mid\mathcal{F}_t\right]\\
&+\Rme{\tfrac{2\tilde \gamma_t}{\eta_t}}\left(\tfrac{(2\eta^2\sigma_f^2+2\sigma_h^2)}{KS}+\left(1-\tfrac{S}{N}\right)\tfrac{4}{S}G_{ncvx}^2\right)+\RRRme{\tfrac{2\tilde \gamma_t^2}{\eta_t}}\left(\tfrac{9(L_h+\eta)(2\eta^2\sigma_f^2+2\sigma_h^2)}{K\Rme{\gamma_{g_t}^2}}+\tfrac{18(L_h+\eta )G_{ncvx}^2}{{\Rme{\gamma_{g_t}^2}}}\right).
\end{align*}
By invoking Theorem~\ref{thm:thm2}, for \Rme{$\tilde \gamma_t:= \tfrac{p\RRme{\ln}(R_t)}{R_t}$ and $\eta_t:= \eta\le{\alpha}/{2\|\nabla F(x^{t*})\|}$}, we derive the result.
\end{proof}

\end{document}